\documentclass [11pt]{amsart}
\usepackage {amsmath, amssymb, amscd, mathrsfs, amsthm, stmaryrd,  bbm, diagbox, enumerate, slashed, graphicx, color, subfig, transparent,comment}
\usepackage[all, cmtip]{xy}
\usepackage[text={6.3in,8.8in},centering,letterpaper,dvips]{geometry}
\usepackage{footmisc}
\usepackage{url, hyperref}
\makeatletter
\newcommand{\addresseshere}{%
  \enddoc@text\let\enddoc@text\relax
}
\makeatother

\definecolor{darkgreen}{rgb}{0.33, 0.42, 0.18}
\newtheorem {theorem}{Theorem}[section]
\newtheorem {lemma}[theorem]{Lemma}
\newtheorem {proposition}[theorem]{Proposition}
\newtheorem {corollary}[theorem]{Corollary}

\newtheorem {definition}[theorem]{Definition}
\newtheorem {question}[theorem]{Question}

\theoremstyle{remark}
\newtheorem {remark}[theorem]{Remark}
\newtheorem {example}[theorem]{Example}

\def\Z {{\mathbb{Z}}}

\def\Q {{\mathbb{Q}}}
\def\CP{\mathbb{CP}^2}
\def\bCP{\overline{\mathbb{CP}^2}}
\def\Arf{\operatorname{Arf}}
\def\Wo{W^{\circ}}
\newcommand{\mK} {\mkern3mu \overline{\mkern-3mu K \mkern-1mu} \mkern1mu}  
\def\F{\mathbb{F}}
\def\sigmaTL{\sigma_{\operatorname{LT}}}
\def\del{\partial}

\newcommand\intB{\smash{\mathring{B}^4}}
\newcommand\intD{\smash{\mathring{D}^2}}
\def\Span{\operatorname{Span}}

\begin{document}

\title{From zero surgeries to candidates for exotic definite four-manifolds}

\author[Ciprian Manolescu]{Ciprian Manolescu}
\address {Department of Mathematics, Stanford University\\
Stanford, CA 94305}
\email {\href{mailto:cm5@stanford.edu}{cm5@stanford.edu}}
\thanks {CM was partially supported by NSF grant DMS-2003488 and a Simons Investigator award.}

\author[Lisa Piccirillo]{Lisa Piccirillo}
\address{Department of Mathematics, Massachusetts Institute of Technology\\
Cambridge, MA 02139}
\email{\href{mailto:piccirli@mit.edu}{piccirli@mit.edu}}
\thanks{LP was partially supported by NSF postdoctoral fellowship DMS-1902735 and the Max Planck Institute for Mathematics} 

\begin{abstract}
One strategy for distinguishing smooth structures on closed $4$-manifolds is to produce a knot $K$ in $S^3$ that is slice in one smooth filling $W$ of $S^3$ but not slice in some homeomorphic smooth filling $W'$. In this paper we explore how $0$-surgery homeomorphisms can be used to potentially construct exotic pairs of this form. In order to systematically generate a plethora of candidates for exotic pairs, we give a fully general construction of pairs of knots with the same zero surgeries. By computer experimentation, we find $5$ topologically slice knots such that, if any of them were slice, we would obtain an exotic four-sphere. We also investigate the possibility of constructing exotic smooth structures on $\#^n \CP$ in a similar fashion. 
\end{abstract}

\maketitle

\section{Introduction}
Ever since the work of Freedman \cite{Freedman} and Donaldson \cite{Donaldson}, the following strategy for disproving the smooth $4$-dimensional Poincar\'e conjecture has garnered interest: Find a homotopy $4$-sphere $W$ and a knot $K \subset S^3=\partial(W \setminus \intB)$ which bounds a smoothly embedded disk in $W \setminus \intB$ but which is not slice (in $B^4$). It would then follow that $W$ is homeomorphic but not diffeomorphic to $S^4$. This strategy has a technical advantage over directly distinguishing $W$ from $S^4$ by computing some diffeomorphism invariant for $W$: there are no known diffeomorphism invariants for homotopy spheres, but there is an invariant (Rasmussen's $s$ invariant from \cite{Rasmussen}) which could obstruct $K$ in the above strategy from being slice in $B^4$. 

In \cite{FGMW}, Freedman, Gompf, Morrison and Walker explicitly attempted this strategy; for one homotopy 4-sphere from the literature they found  a knot $K$ which bounds a smooth disk in the homotopy sphere, and tried to use Rasmussen's $s$ invariant to show that $K$ is not slice in $B^4$. The $s$ invariant is known to be zero for slice knots, but (unlike other similar invariants \cite{OStau, km}) it is unknown whether $s$ necessarily vanishes when $K$ bounds a smooth disk in a homotopy $4$-ball. For the example in \cite{FGMW}, however, the $s$ invariant was $0$ and in fact the homotopy $4$-sphere was proven almost immediately to be standard \cite{akbulutCS}. 

Given that the strategy above seems to be the only presently tractable approach to the smooth $4$-dimensional Poincar\'e conjecture, it is of marked interest to pursue it more systematically.  The goal of this paper is to develop constructions of homotopy spheres $W$ which come equipped with a knot $K$ which bounds a smoothly embedded disk in $W \setminus \intB$ but which does not appear slice. Our constructions are broad in scope, but can also produce homotopy spheres and knots which are simple enough to be studied explicitly, a process we also begin here.

We use pairs $K$ and $K'$ with the same $0$-surgeries to produce such examples as follows: If $K$ is slice and $S^3_0(K)\cong S^3_0(K')$, then by gluing the complement of the slice disk for $K$ to the trace of the $0$-surgery for $K'$, we obtain a homotopy $4$-sphere $W$, such that $K'$ bounds a disk in $W \setminus \intB$. (See Problem 1.19 in \cite{KirbyList}.) If $s(K') \neq 0$, then $K'$ is not slice and $W$ is an exotic four-sphere.


In principle, the same idea can be used to produce examples of exotic smooth structures on $\#^n \CP$ for $n \geq 1$. The work of Freedman \cite{Freedman} and Donaldson \cite{Donaldson} implies that every simply-connected, positive definite, smooth, closed $4$-manifold is homeomorphic to $\#^n \CP$ for some $n$. It is unknown if $W=\#^n \CP$ admits exotic smooth structures. Let $\Wo = W \setminus \intB$ and define a knot to be {\em H-slice in $W$} it bounds a smoothly embedded nullhomologous disk in $\Wo$. If we found knots $K$ and $K'$ such that 
\[
S^3_0(K)\cong S^3_0(K'), \ \ \ K \text{ is H-slice in }W,   \ \ \ K' \text{ is not H-slice in }W,
\]
we would produce an exotic smooth structure on $W$. Note that we could obstruct $K'$ from being H-slice in $\#^n \CP$ by showing that Rasmussen's $s$ invariant satisfies $s(K') < 0$; see \cite{MMSW}.  

Techniques for constructing pairs of knots with homeomorphic $n$-surgeries first appeared in the late 70's, see \cite{Lickorish1, Akb2Dhom, Lickorish2}; for $n=0$, see  \cite{Brakes}. Other fundamentally distinct constructions were given in \cite{Osoinach} and \cite{Yasui}. Some of these constructions always produce knots such that not only are the 0-surgeries homeomorphic, but in fact the traces are diffeomorphic \cite{Akb2Dhom, Lickorish2, Brakes}. Other constructions sometimes produce knots with diffeomorphic traces \cite{AJOT}. Producing knots with diffeomorphic traces is useful for some purposes, for example for the proof that Conway's knot is not slice \cite{Pic2}. But pairs of knots with the same trace are not useful for obtaining exotic structures in the manner described above, as the trace embedding lemma (\cite{FoxMilnor}, see Lemma \ref{lem:TEL}) readily implies that if $K$ is H-slice in some manifold $W$, then so is $K'$.

In this paper, in order to produce the broadest possible selection of candidates for exotic homotopy spheres built using 0-surgery homeomorphisms we give a fully general framework for constructing pairs of knots with homeomorphic $0$-surgeries. Our framework is based on $3$-component links of the following form. 

\begin{definition} 
\label{def:RBG}
An \emph{RBG link} $L=R\cup B\cup G \subset S^3$ is a 3-component rationally framed link, with framings  $r, b, g$ respectively, such that $H_1(S^3_{r,b,g}(R\cup B\cup G);\mathbb{Z})=\mathbb{Z}$, together with homeomorphisms $\psi_B:S^3_{r,g}(R\cup G)\to S^3$ and $\psi_G:S^3_{r,b}(R\cup B)\to S^3$.
\end{definition}

\begin{theorem}
\label{thm:RBG}
Any RBG link $L$ has a pair of associated knots $K_B$ and $K_G$ and homeomorphism $\phi_L:S^3_0(K_B)\to S^3_0(K_G)$. Conversely, for any 0-surgery homeomorphism $\phi:S^3_0(K)\to S^3_0(K')$ there is an associated RBG link $L_\phi$ with $K_B=K'$, $K_G=K$, and $\phi_L = \phi$. 
\end{theorem}

We will explain how particular cases of RBG links recover other constructions from the literature, such as annulus twisting, dualizable patterns, and Yasui's construction. Moreover, there is a straightforward condition on the homeomorphism $\phi$ that guarantees that it does not extend to a {diffeomorphism} of the corresponding traces.\footnote{This condition is necessary to avoid building homotopy spheres which are immediately diffeomorphic to $S^4$; this was overlooked in  \cite{AJOT} and \cite{AbeTange}.}

In order to build homotopy spheres via 0-surgery homeomorphisms such that the accompanying  knots $K'$  are simple enough to study en masse, we study the following special type of RBG links: We take $R$ to be an $r$-framed knot with $r\in \Z$, and $B$ and $G$ to be $0$-framed unknots with linking number $l$ such that $l=0$ or $rl=2$. Further, letting $\mu_R$ denote a meridian for $R$, we ask that there exist link isotopies
$$R \cup B \cong R \cup \mu_R \cong R \cup G.$$ 
We call RBG links of this form {\em special}. Special RBG links are easy to draw, and suffice to produce many examples of knots with the same $0$-surgeries where the corresponding traces are not  even
homeomorphic. Further, special RBG links can produce pairs where both knots are very low crossing number. For example, we produce pairs $K$ and $K'$ with $12$ and $14$ crossings, respectively; to our knowledge this minimizes $c(K)+c(K')$ among all pairs of knots with homeomorphic 0-surgeries in the literature. See Example \ref{ex:smallrbg}. 

Using the computer programs \emph{SnapPy} \cite{SnapPy}, \emph{KnotTheory\!\`{}} \cite{KnotTheory}, \emph{SKnotJob} \cite{SKnotJob}, and the {\em Knot Floer homology calculator} \cite{HFK}, we investigated a $6$-parameter family consisting of $3375$ special RBG links. This yielded $21$ interesting pairs $(K, K')$ for which $S^3_0(K) \cong S^3_0(K')$,  $s(K') = -2 \neq 0$, and for which we could not determine whether $K$ is slice. Shortly after the original version of this paper was posted to the arXiv, Nathan Dunfield and Sherry Gong informed us that, using the twisted Alexander polynomial obstructions from \cite{HKL}, they were able to prove that 16 of our 21 knots are not slice \cite{DG, DG2}. Thus, we are left with $5$ knots $K$ with the following property.

\begin{theorem}
\label{thm:21}
If any of the $5$ knots shown in Figure~\ref{fig:21a} are slice, then an exotic four-sphere exists.
\end{theorem}

\captionsetup[subfigure]{font=normalsize, labelformat=empty}
\begin{figure}[htb]
\centering
\subfloat[$K_1$]{\includegraphics[scale=.15]{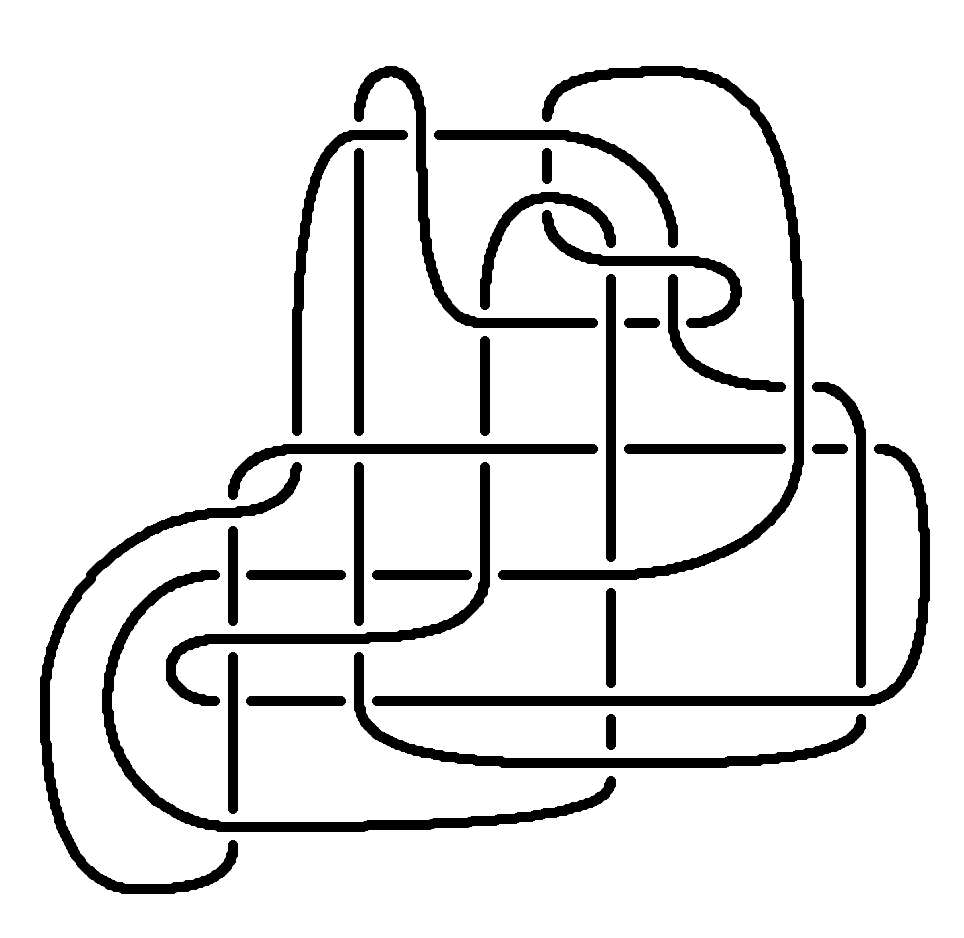}}\ \ \
\subfloat[$K_2$]{\includegraphics[scale=.08]{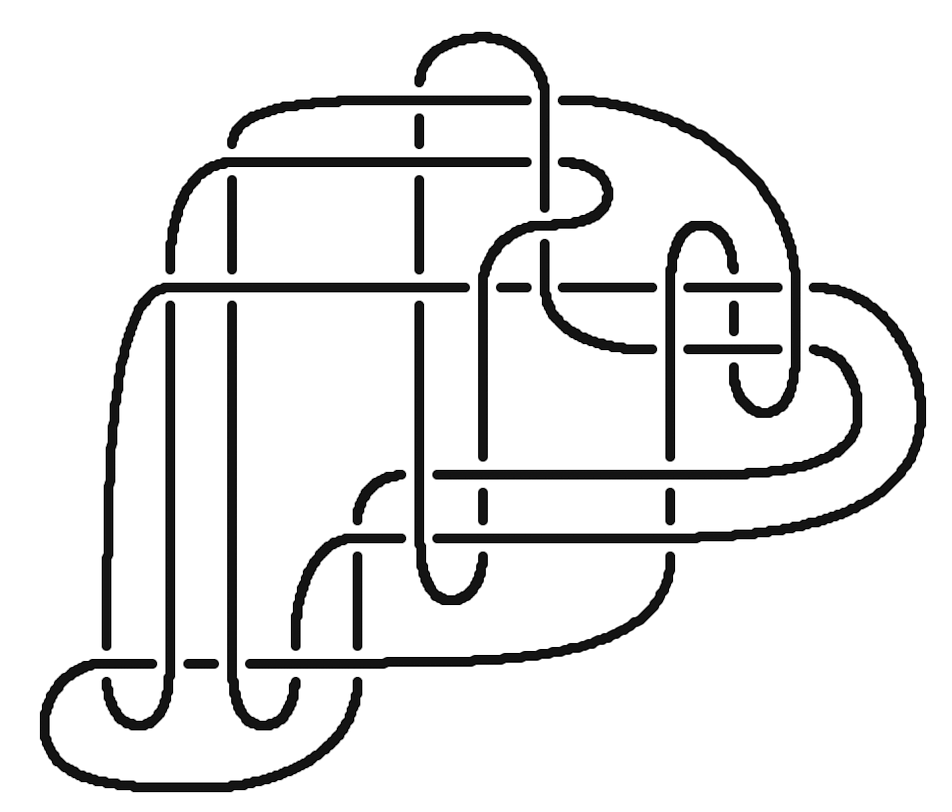}}\ \ \
\subfloat[$K_3$]{\includegraphics[scale=.08]{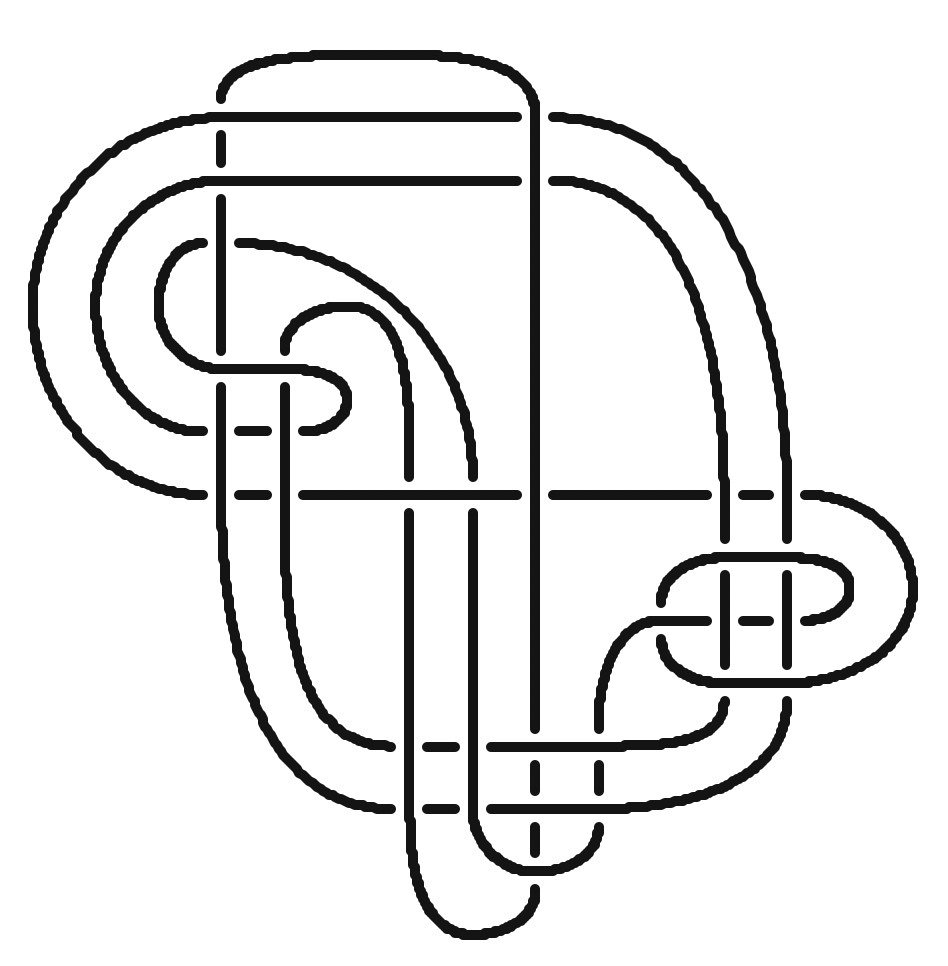}}\ \ \
\subfloat[$K_4$]{\includegraphics[scale=.08]{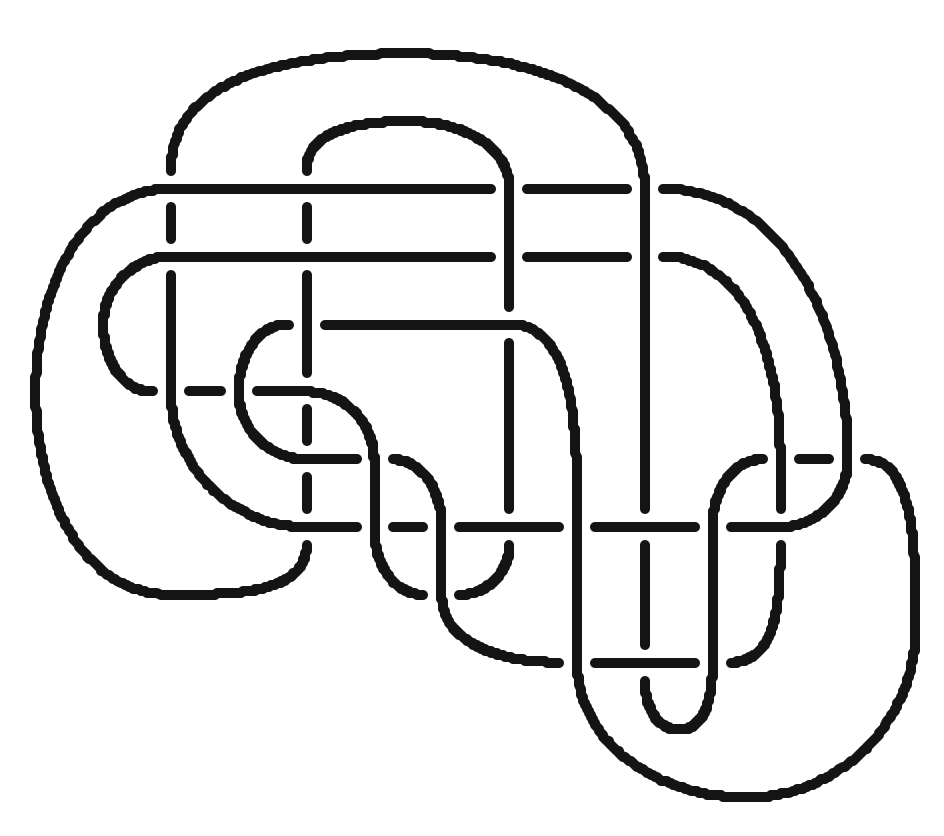}}\ \ \
\subfloat[$K_5$]{\includegraphics[scale=.08]{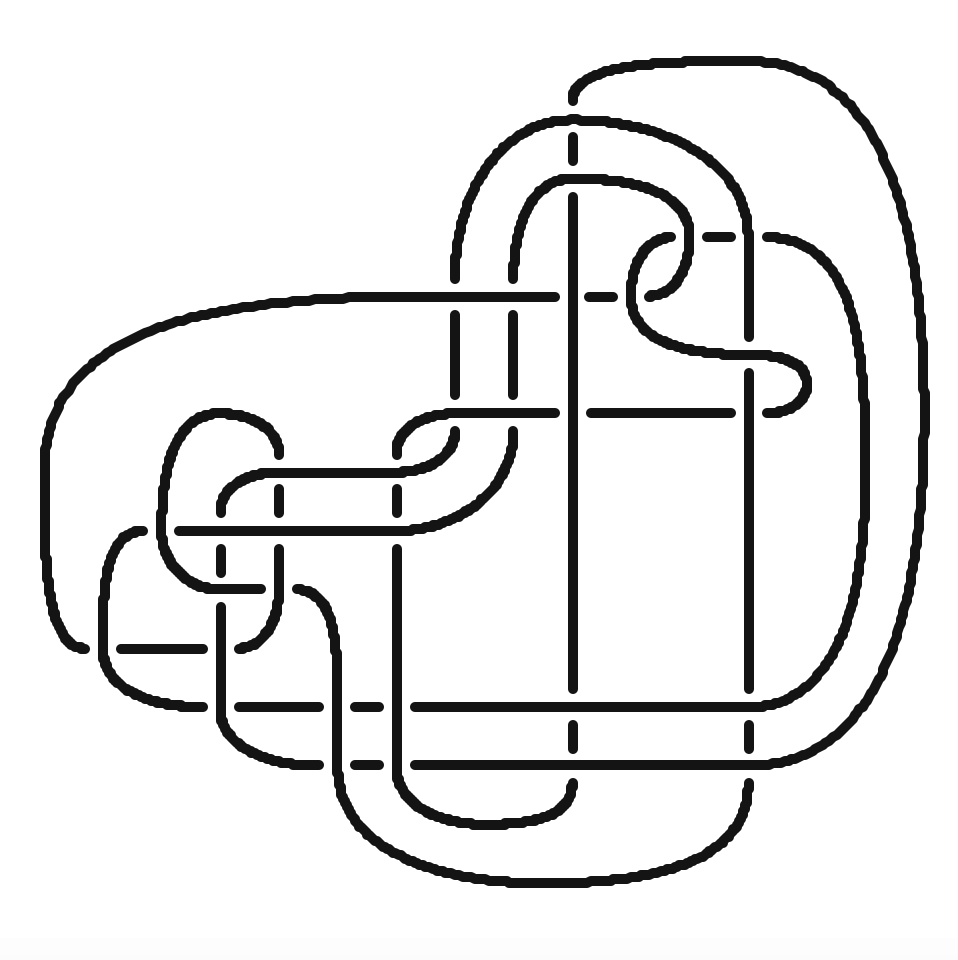}}
\caption{Candidates for slice knots.}
\label{fig:21a}
\end{figure}

We have verified that these knots pass many of the known obstructions to sliceness. Specifically:
\begin{itemize}
\item They have Alexander polynomial 1, and are therefore topologically slice by \cite{Freedman}. 
\item Their $\tau$, $\epsilon$ and $\nu$ invariants from knot Floer homology vanish;
\item Rasmussen's $s$ invariant equals zero;
\item The variants $s^{\F_2}$ and $s^{\F_3}$ of Rasmussen's invariant (from Khovanov homology over the fields $\F_2$ and $\F_3$) also vanish;
\item The Lipshitz-Sarkar $\operatorname{Sq}^1$ $s$-invariants vanish;
\item For at least $3$ of these knots ($K_1$, $K_4$ and $K_5$), the given homeomorphism $\phi: S^3_0(K) \to S^3_0(K')$ does not extend to a trace diffeomorphism, so the non-sliceness of $K'$ does not immediately obstruct $K$ from being slice. 
\end{itemize}

The other $16$ knots from our original list are denoted $K_6$ through $K_{21}$ and shown in Figure~\ref{fig:2}. They are algebraically but not topologically slice, and satisfy
\begin{equation}
\label{eq:sliceinvariants}
\tau=\epsilon=\nu=s=s^{\F_2}=s^{\F_3}=s^{\operatorname{Sq}^1} =0.
\end{equation}
This leaves open the possibility that they could lead to exotic structures on $\#^n \CP$. Moreover, our computer experiments produced two other knots ($K_{22}$ and $K_{23}$ in Figure~\ref{fig:2}) which are not even algebraically slice (they fail the Fox-Milnor condition on the Alexander polynomial), but have vanishing Levine-Tristram signature function, satisfy \eqref{eq:sliceinvariants}, and have a companion knot $K'$ with $s(K')=-2< 0$. These two knots are additional candidates for producing exotic smooth structures on $\#^n \CP$. 

\begin{theorem}
\label{thm:23}
If any of the $23$ knots shown in Figures~\ref{fig:21a} and \ref{fig:2} are H-slice in $\#^n \CP$ for some $n$, then an exotic $\#^n \CP$ exists.
\end{theorem}

By contrast, one can show that all $23$ of these knots are H-slice in $\#^n \bCP$ for some $n > 0$. Knots that are H-slice in both $\#^n \bCP$ and $\#^n \CP$ (for some $n$) are called \emph{biprojectively H-slice}, or \emph{BPH-slice}. BPH-slice knots have vanishing Levine-Tristram signature function, and satisfy $\tau=\epsilon=s=0$; cf. \cite{CochranHarveyHorn}, \cite{MMSW}. We observe in Section \ref{sec:BPH} that many of the small knots for which these invariants vanish can be shown to be BPH-slice. 

So far we have focused on pairs of knots $(K, K')$ with the same $0$-surgery, for which we know that $K'$ is not slice (or not H-slice in $\#^n \CP$), and we are unsure about $K$. We could also look at pairs $(K, K')$ with the same $0$-surgery for which we know that $K$ is slice (or H-slice in $\#^n \CP$), and $s(K')=0$. (If $s(K')\neq 0$, this would be a different paper.) There are plenty of such examples coming from special RBG links or from annulus twisting. In some situations, we know that $K'$ is in fact slice, and in others we are not sure. In either case, interesting homotopy $4$-spheres can be constructed using the RBG link $L_\phi$ from a  homeomorphism $\phi: S^3_0(K) \to S^3_0(K')$. The challenge then becomes to determine whether these homotopy 4-spheres are standard. In Figure~\ref{fig:htpy88} we exhibit an explicit infinite family of examples of homotopy 4-spheres constructed by this method.

\begin{remark}
\label{rem:nakamura}
 After our paper was posted on the arXiv, Nakamura \cite{Nakamura} showed that the knots in Figures~\ref{fig:21a} and \ref{fig:2} are not H-slice in $\#^n \CP$ for any $n$ (and in particular not slice). Thus, they cannot be used to produce an exotic $S^4$ or $\#^n \CP$. More generally, he proved that the $s$-invariant cannot be used to construct an exotic $\#^n \CP$ by starting from a special RBG link where the $R$ component is the unknot. Furthermore, Nakamura also showed that the homotopy 4-spheres from Figure~\ref{fig:htpy88} are standard. Nevertheless, the methods developed in this paper could potentially still be used to find exotic $S^4$ or $\#^n \CP$, by considering either more general RBG links or other concordance invariants.
\end{remark}

\subsection{Organization of the paper} In Section~\ref{sec:BPH} we introduce BPH-slice knots and give examples. In Section~\ref{sec:RBG} we present the general RBG construction and prove Theorem~\ref{thm:RBG}, and discuss when 0-surgery homeomorphisms extend to trace homeomorphisms or diffeomorphisms. In Section~\ref{sec:smallspecial} we restrict attention to special RBG links, and introduce a concept (small RBG links) that ensures the resulting diagrams for $K_B$ and $K_G$ are manageable. In Section~\ref{sec:computer} we describe our computer experiments, and explain how we arrived at the knots  in Figures~\ref{fig:21a} and \ref{fig:2}. In Section~\ref{sec:annulus} we review how annulus twisting gives rise to 0-surgery homeomorphisms, and give examples of homotopy 4-spheres arising from this construction. Finally, in Section~\ref{sec:otherconstr} we relate RBG links to other known ways to produce 0-surgery homeomorphisms: annulus twisting, Yasui's construction, and dualizable patterns.

\subsection{Conventions}
All manifolds are smooth and oriented and all homeomorphisms are orientation preserving. Boundaries are oriented with outward normal first. Slice refers to the existence of a smooth disk, and topologically slice to that of a locally flat disk. Homology has integral coefficients. The symbol $\nu$ denotes a tubular neighborhood, and $U$ denotes the unknot.

\subsection{Acknowledgements.} We would like to thank the organizers of the CRM 50th Anniversary Program in Low-Dimensional Topology (Montr\'eal, 2019), where this collaboration started. We thank Dror Bar-Natan, Kyle Hayden, Chuck Livingston, Marco Marengon, Allison Miller and Qianhe Qin for helpful comments on previous versions of this paper. In particular, we are grateful to Kyle Hayden for pointing out that some of the homotopy $4$-spheres we previously considered were in fact standard, and to Nathan Dunfield and Sherry Gong for checking that some of the knots in our list were not topologically slice. We are also grateful to the referee for the careful reading of our paper.


\section{BPH-slice knots}
\label{sec:BPH}
Let $W$ be a closed, smooth, oriented $4$-manifold. We let $\Wo := W \setminus \intB$.
\begin{definition}
We say that a knot $K \subset S^3$ is {\em H-slice in $\Wo$} if it bounds a smooth, properly embedded disk $\Delta$ in $\Wo$, such that $[\Delta]=0\in H_2(\Wo, \del \Wo).$ For convenience, we will also sometimes use the terminology {\em H-slice in $W$} to mean H-slice in $\Wo$.
\end{definition}
Observe that if a knot is slice in the usual sense, then it is H-slice in any $W$.

\begin{remark}
It is shown in \cite[Corollary 1.5]{MMP} and that the set of H-slice knots can detect exotic smooth structures on some 4-manifold with indefinite intersection form. See \cite{MMP, IMT} for more obstructions to H-sliceness in such manifolds.
\end{remark}

Knots that are H-slice in some simply connected $4$-manifold with a positive definite (or  negative definite, resp.) intersection form are called {\em $0$-positive} ({\em $0$-negative}, resp.) in \cite{CochranHarveyHorn}. Note that $K$ is $0$-negative iff the mirror $\mK$ is $0$-positive. Several obstructions to $0$-positivity are collected in \cite[Proposition 1.1]{CochranHarveyHorn}. If $K$ is $0$-positive, then
\begin{itemize}
\item The signature of the knot satisfies $\sigma(K) \leq 0$;
\item More generally, the Levine-Tristram signature function $\sigmaTL(K)$ (evaluated away from the roots of the Alexander polynomial) is non-positive;
\item The Ozsv\'ath-Szab\'o concordance invariant satisfies $\tau(K) \geq 0$; cf \cite[Theorem 1.1]{OStau}.
\end{itemize}
There are additional obstructions from the Heegaard Floer correction terms of cyclic branched covers or $\pm 1$-surgeries on $K$, and from Yang-Mills theory; see \cite{CochranHarveyHorn}, \cite[Corollary 5.5]{km} and \cite[Theorem 4.1]{DaemiScaduto}.

When $W = \#^n \CP$, another obstruction comes from Khovanov homology: From \cite[Corollary 1.9]{MMSW}, it follows that if $K$ is H-slice in $\#^n \CP$ for some $n$, then Rasmussen's $s$ invariant satisfies
\begin{equation}
\label{eq:sless}
 s(K) \geq 0.
 \end{equation}
 
 \begin{remark}
Note that invariants such as $\sigma$, $\sigma_{LT}$ and $\tau$ behave in the same way with regard to H-slice knots in $\#^n \CP$ as in any other positive definite manifold, whereas for Rasmussen's invariant, the inequality \eqref{eq:sless} was only proved for H-slice knots in $\#^n \CP$. This is what makes the $s$ invariant of particular interest for the purpose of detecting exotic smooth structures on $\#^n \CP$ via H-sliceness.
 \end{remark}

Examples of knots that are H-slice in $ \#^n \CP$ can be easily constructed using the following well-known lemma.
\begin{lemma}
\label{lem:HCP}
Let $K, K' \subset S^3$ be knots such that $K'$ is obtained from $K$ by changing a negative crossing to a positive crossing in a diagram of $K$. If $K$ is H-slice in $\#^{n-1} \CP$, then $K'$ is H-slice in $\#^n \CP$.
\end{lemma}

\begin{figure}

{
   \fontsize{9pt}{11pt}\selectfont
   \def\svgwidth{3in}
\begingroup%
  \makeatletter%
  \providecommand\color[2][]{%
    \errmessage{(Inkscape) Color is used for the text in Inkscape, but the package 'color.sty' is not loaded}%
    \renewcommand\color[2][]{}%
  }%
  \providecommand\transparent[1]{%
    \errmessage{(Inkscape) Transparency is used (non-zero) for the text in Inkscape, but the package 'transparent.sty' is not loaded}%
    \renewcommand\transparent[1]{}%
  }%
  \providecommand\rotatebox[2]{#2}%
  \newcommand*\fsize{\dimexpr\f@size pt\relax}%
  \newcommand*\lineheight[1]{\fontsize{\fsize}{#1\fsize}\selectfont}%
  \ifx\svgwidth\undefined%
    \setlength{\unitlength}{208.16441345bp}%
    \ifx\svgscale\undefined%
      \relax%
    \else%
      \setlength{\unitlength}{\unitlength * \real{\svgscale}}%
    \fi%
  \else%
    \setlength{\unitlength}{\svgwidth}%
  \fi%
  \global\let\svgwidth\undefined%
  \global\let\svgscale\undefined%
  \makeatother%
  \begin{picture}(1,0.41871695)%
    \lineheight{1}%
    \setlength\tabcolsep{0pt}%
    \put(0,0){\includegraphics[width=\unitlength,page=1]{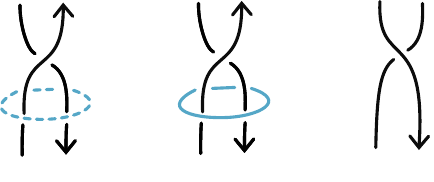}}%
    \put(0.13528679,0.27414406){\color[rgb]{0,0,0}\makebox(0,0)[lt]{\lineheight{1.25}\smash{\begin{tabular}[t]{l}$c$\end{tabular}}}}%
    \put(0.04362443,0.0136739){\color[rgb]{0,0,0}\makebox(0,0)[lt]{\lineheight{1.25}\smash{\begin{tabular}[t]{l}$\partial^-X$\end{tabular}}}}%
    \put(0.18884707,0.12229624){\color[rgb]{0.30196078,0.63921569,0.76862745}\makebox(0,0)[lt]{\lineheight{1.25}\smash{\begin{tabular}[t]{l}$\gamma$\end{tabular}}}}%
    \put(0.46337704,0.0115969){\color[rgb]{0,0,0}\makebox(0,0)[lt]{\lineheight{1.25}\smash{\begin{tabular}[t]{l}$\partial^+X$\end{tabular}}}}%
    \put(0.83087507,0.0115969){\color[rgb]{0,0,0}\makebox(0,0)[lt]{\lineheight{1.25}\smash{\begin{tabular}[t]{l}$\partial^+X\cong S^3$\end{tabular}}}}%
    \put(0.61399181,0.12229624){\color[rgb]{0.30196078,0.63921569,0.76862745}\makebox(0,0)[lt]{\lineheight{1.25}\smash{\begin{tabular}[t]{l}$1$\end{tabular}}}}%
    \put(0,0){\includegraphics[width=\unitlength,page=2]{crossingchange.pdf}}%
  \end{picture}%
\endgroup%

}
\caption{An annular cobordism in $\CP \setminus (\intB \sqcup \intB)$ from $K$ to $K'$}
\label{fig:crossingchange}
\end{figure}

\begin{proof}
Let $c$ be a negative-to-positive crossing change in a diagram of $K$ which yields $K'$.  Consider $K\subset S^3\times \{0\}\subset S^3\times I$ and attach a 1-framed 2-handle to $S^3\times \{1\}$ along a curve $\gamma$ which links $c$ as in Figure \ref{fig:crossingchange}. This handle attachment yields the cobordism $X=\CP \setminus (\intB \sqcup \intB)$ and there is a natural nullhomologous annular cobordism from $K\subset \partial^-X$ to the knot $K^+\subset \partial^+X$ depicted in the center frame of Figure \ref{fig:crossingchange}. When we identify $\partial^+X$ with the standard diagram of $S^3$ (via, say, a Rolfsen twist) as in the right frame of Figure \ref{fig:crossingchange}, we can identify $K^+$ as $K'$. The claim follows by stacking $X$ on top of $\#^{n-1} \CP$.
\end{proof}

More generally, the conclusion of Lemma~\ref{lem:HCP} also holds when $K'$ is obtained from $K$ by adding a generalized positive crossing in the sense of  \cite[Definition 2.7]{CochranTweedy}.

The following concept will be of particular interest to us.
\begin{definition}
A knot $K \subset S^3$ is called {\em biprojectively H-slice} (or {\em BPH-slice}) if it is H-slice in both $\#^n \CP$ and $\#^n \bCP$, for some $n \geq 0$.
\end{definition}
In \cite{CochranHarveyHorn}, knots that are both $0$-positive and $0$-negative are called $0$-bipolar. BPH-slice knots are $0$-bipolar. Moreover, note that every simply connected, positive definite, smooth closed $4$-manifold is homeomorphic to $\#^n \CP$ for some $n$, and there are no known exotic smooth structures on such manifolds. Thus, $0$-bipolar and BPH-slice might be the same notion.

By applying the obstructions above for both $K$ and $\mK$, we see that for BPH-slice knots they become equalities instead of inequalities. Therefore, if $K$ is BPH-slice then:
$$\sigma(K)=0, \ \ \sigmaTL(K)=0, \ \ \tau(K) = 0, \ \ s(K)=0.$$
Further, Hom's $\epsilon$ invariant from knot Floer homology \cite{Hom} also has to vanish for BPH-slice knots; see \cite[Proposition 4.10]{CochranHarveyHorn}.

Slice knots are BPH-slice. We see that many of the obstructions that vanish for ordinary slice knots also vanish for BPH-slice knots. The main difference is the Fox-Milnor condition on the Alexander polynomial, which does not need to hold for BPH-slice knots.

The following is an immediate consequence of Lemma~\ref{lem:HCP}. 
\begin{lemma}
\label{lem:ConstructBPH}
Suppose that $K \subset S^3$ is a knot with the following properties:
\begin{itemize}
\item There exists a diagram of $K$ and a negative crossing in that diagram, such that when we change it to a positive crossing, we get a BPH-slice knot;
\item There exists a (possibly different) diagram of $K$ and a positive crossing in that diagram, such that when we change it to a negative crossing, we get a BPH-slice knot.
\end{itemize}
Then, $K$ is BPH-slice.
\end{lemma}

\begin{example}
The simplest nontrivial BPH-slice knot is the figure-eight $4_1$. Its standard diagram (shown in  Figure~\ref{fig:41818}) has two negative and two positive crossings, and changing the sign of any of the crossings produces the unknot.
\end{example}

\begin{figure}

{
   \fontsize{9pt}{11pt}\selectfont
   \def\svgwidth{3.5in}
   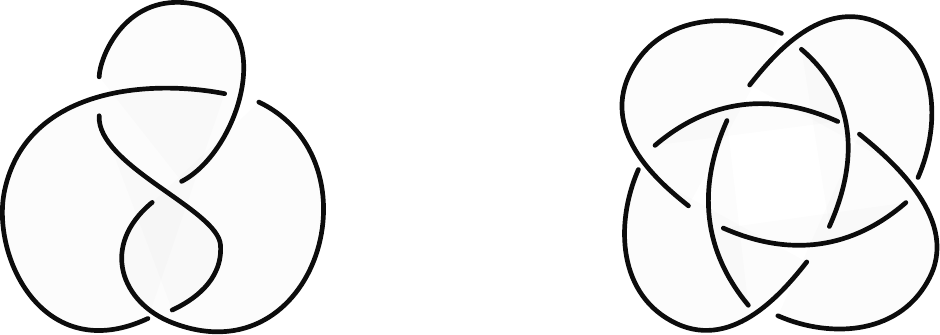
}
\caption{Left: the BPH-slice knot $4_1$. Right: The knot $8_{18}$, whose BPH-sliceness is unknown.}
\label{fig:41818}
\end{figure}

Using Lemma~\ref{lem:ConstructBPH}, we found that BPH-slice knots are significantly more common than slice knots among small knots. Indeed, most small knots with $\sigma=0$ are BPH-slice. Here is a list of all the prime knots with at most 9 crossings and $\sigma=0$:
\begin{itemize}
\item {\em slice knots}: $6_1, 8_8, 8_9, 8_{20}, 9_{27}, 9_{41}, 9_{46}$; 
\item {\em amphichiral non-slice knots}: $4_1, 6_3, 8_3, 8_{12}, 8_{17}, 8_{18}$;
\item {\em non-amphichiral, non-slice knots}: $7_7, 8_1, 8_{13}, 9_{14}, 9_{19}, 9_{24}, 9_{30}, 9_{33}, 9_{34}, 9_{37}, 9_{44}$.
\end{itemize}

Of these, all except $8_{18}$ can be shown to be BPH-slice by starting with the diagrams found in knot tables, checking which crossing changes produce smaller BPH-slice knots, and applying Lemma~\ref{lem:ConstructBPH}.

We could not determine if $8_{18}$ is BPH-slice. Changing any of the crossings in its minimal diagram (shown in  Figure~\ref{fig:41818}) gives a trefoil knot, whose signature is nonzero. Observe, however, that for the composite knot $3_1 \# (-3_1)$, it is also true that changing any crossing gives a trefoil; nevertheless, the knot is slice.

\section{A general RBG construction}
\label{sec:RBG}
In this section we give a fully general framework for describing homeomorphisms between manifolds arising as $0$-surgeries on knots. Using this framework we discuss when a 0-surgery homeomorphism can be extended to a trace homeomorphism or diffeomorphism. In Section \ref{sec:smallspecial} we will make some simplifying assumptions that lead to more user-friendly results and examples.

Our construction is based on certain three-component links, called {\em RBG links}, which generalize those already considered in \cite[Section 2]{Pic1}.

\subsection{RBG links}
Let $\vec{f}$ denote a finite ordered list with values in $\mathbb{Q}\cup\{\infty\}\cup\{*\}$. We will use the notation $\smash{S^3_{\vec{f}}(L)}$ to denote the $\smash{\vec{f}}$ surgery on a framed link $L$, where an $*$ denotes a complement (i.e., removing a neighborhood of that component and not filling it in). Given a 3-manifold homeomorphism $\phi:M\to N$ we will sometimes abusively still use $\phi$ to refer to a restriction of $\phi$ to some codimension zero submanifold of $M$.

{
\renewcommand{\thetheorem}{\ref{def:RBG}}
\begin{definition}
An \emph{RBG link} $L=R\cup B\cup G \subset S^3$ is a 3-component rationally framed link, with framings  $r, b, g$ respectively, such that $H_1(S^3_{r,b,g}(R\cup B\cup G);\mathbb{Z})=\mathbb{Z}$, together with homeomorphisms $\psi_B:S^3_{r,g}(R\cup G)\to S^3$ and $\psi_G:S^3_{r,b}(R\cup B)\to S^3$.
\end{definition}
\addtocounter{theorem}{-1}
}

For examples of RBG links, see Section~\ref{sec:smallspecial}.

{
\renewcommand{\thetheorem}{\ref{thm:RBG}}
\begin{theorem}
Any RBG link $L$ has a pair of associated knots $K_B$ and $K_G$ and homeomorphism $\phi_L:S^3_0(K_B)\to S^3_0(K_G)$. Conversely, for any 0-surgery homeomorphism $\phi:S^3_0(K)\to S^3_0(K')$ there is an associated RBG link $L_\phi$ with $K_B=K'$, $K_G=K$, and $\phi_L = \phi$. 
\end{theorem}
\addtocounter{theorem}{-1}
}

\begin{remark}
We fix particular homeomorphisms $\psi_B$ and $\psi_G$ in Definition~\ref{def:RBG} because for other  choices $\psi_B'$ and $\psi_G'$ (which are necessarily isotopic to $\psi_B$ and $\psi_G$, since there is a unique homeomorphism of $S^3$ up to isotopy) it is possible to produce homeomorphisms $\phi_L$ and $\phi_L'$ which are distinct \emph{up to isotopy}. This can be seen by choosing say $\psi_B'$ to be $\tau\circ\psi_B$ where $\tau$ is a homeomorphism of $S^3$ inducing a nontrivial symmetry of $K_B$.  

For the remainder of the paper, when we are only concerned with the existence of a homeomorphism $\phi_L$, rather than the particular isotopy class of the homeomorphism,  we will not reference the choices of $\psi_B$ and $\psi_G$.  
\end{remark}

\begin{proof}
For the first claim, define $(K_B, f_b)$ to be the framed knot in $S^3$ satisfying $\psi_B^*:S^3_{r,b,g}(L)\to S^3_{f_b}(K_B)$, where $\psi_B^*$ is induced from $\psi_B$ by pushing forward the framed knot $(B,b)$. Similarly take $(K_G, f_g)$ to be the framed knot satisfying $\psi_G^*:S^3_{r,b,g}(L)\to S^3_{f_g}(K_G)$. Then $\phi_L:= \psi_G^*\circ \psi_B^{*(-1)}$ is a homeomorphism $\phi_L:S^3_{f_b}(K_B)\to S^3_{f_g}(K_G)$; we define $\phi_L$ to be the homeomorphism associated to the RBG link.  The homology assumption on $S^3_{r,b,g}(L)$ implies $f_b=f_g=0$. 

For the second claim let $\mu_{K'}$ be the meridian for $K'$, and let $(R,r)$ be the framed curve given as the image of $(\mu_{K'}, 0)$ under the homeomorphism $\phi^{-1}$.  We will define our RBG link $L$ to be $R\cup K\cup \mu_R$, with $b=g=0$. The homeomorphism $\psi_B: S^3_{r,0,0}(R,K,\mu_R)\to S^3_0(K)$ is given by the slam dunk on $R$ and $\mu_R$.  Pushing $R$ and $\mu_R$ across $\phi$ induces a homeomorphism $\phi^*: S^3_{r,0,0}(R,K,\mu_R)\to S^3_{0,0,0}(\mu_{K'},K',\mu_{\mu_{K'}})$.  There is a natural slam dunk homeomorphism $s:S^3_{0,0,0}(\mu_{K'},K',\mu_{\mu_{K'}})\to S^3_0(K')$. Taking $\psi_G$ to be $s\circ \phi^*$, we have that the homeomorphism induced by $L$ is $\psi_G\circ\psi_B^{-1}$ which is isotopic to $\phi$.
\end{proof}

\begin{remark}
There can be many distinct RBG links producing the same $0$-surgery homeomorphism $\phi: S^3_0(K) \to S^3_0(K')$.
\end{remark}

There are three primary techniques presently in the literature for constructing a 0-surgery homeomorphism: dualizable patterns \cite{Akb2Dhom, BakerMotegi, Brakes, GM, Lickorish2}, annulus twisting \cite{Osoinach}, 
and Yasui's construction \cite{Yasui}. In Section \ref{sec:otherconstr} we will give explicit RBG links for each of these constructions. 

\subsection{Constructing H-slice knots and candidates for exotic pairs}
 In order to build knots $K'$ that are H-slice in an exotic copy of a simply connected four-manifold $W$, we will begin instead with a knot $K$ which is H-slice in $W$ and then construct a $K'$ with $S^3_0(K)\cong S^3_0(K')$. That $K'$ is then H-slice in a homotopy  $W$ follows from the following folklore:

\begin{lemma}\label{lem:slicehomeo}
Let $W$ be a smooth, closed, oriented, simply connected four-manifold. If there is a homeomorphism  $\phi:S^3_0(K)\to S^3_0(K')$ and $K$ is H-slice in $W$, then $K'$ is H-slice in a 4-manifold $X$ with the homotopy type of $W$. 
\end{lemma}

To prove the lemma, we require a definition. 
\begin{definition}
The \emph{trace} of a knot $K$, denoted $X(K)$, is the 4-manifold obtained by attaching a single 0-framed 2-handle to $B^4$ along $K$. 
\end{definition}
\begin{proof}[Proof of Lemma \ref{lem:slicehomeo}]
Since $K$ is H-slice in $W$, we can choose a slice disk for $K$ in $\Wo$ and consider the 4-manifold $V$ obtained by excising an open tubular neighborhood of that disk from $\Wo$. It is routine to confirm that $\partial V\cong S^3_0(K)$ and that $\pi_1(V)$ is normally generated by $\iota_*(\pi_1(\partial V))$, where $\iota: \partial V \to V$ is the inclusion.

Now consider the 4-manifold $$X:=X(-K')\cup_\phi V,$$ where $\phi$ is the assumed homeomorphism from $-\partial(X(-K')) = S^3_0(K')$ to $\partial V \cong S^3_0(K)$. Let $\mu_{K'}$ denote the meridian of $K'$ in $S^3_0(K')$. By thinking of gluing an upside down $X(-K')$ onto a rightside up $V$, we see that $X$ has a handle diagram obtained from that of $V$ by adding an additional 2-handle along the framed curve $\phi(\mu_{K'}, 0)$, followed by a 4-handle. As such, $\pi_1(X)=\pi_1(V)/\langle[\phi^{-1}(\mu_{K'})]\rangle$. Since $\pi_1(S^3_0(K'))/\langle[\mu_{K'}]\rangle=1$ we have  $\pi_1(\partial V)/\langle[\phi^{-1}(\mu_{K'})]\rangle=1$. Since $\pi_1(V)$ is normally generated by $\pi_1(\partial V)$, we have that $\pi_1(X)=1$. It is routine to confirm that $X$ has the homology type of $W$, and the intersection form on $H_2(X)$ and $H_2(W)$ coincide. It is then a consequence of Whitehead's Theorem (see \cite[p.103, Theorem 1.5]{MilnorSBF}  or \cite[Theorem 1.2.25]{GS}) that $X$ is homotopy equivalent to $W$. (In fact, $X$ is homeomorphic to $W$ by Freedman's theorem \cite{Freedman}.)

To see that $K'$ is H-slice in $X$, let $X^\circ(-K')$ denote $X(-K')$ with an open ball removed, and observe that $K'\subset S^4\subset \partial(X^\circ(-K'))$ bounds a disk $D$ in $X^\circ (-K')$ made up of the product cobordism in $S^3\times I$ and the core of the 2-handle. This slice disk survives into $X$, and it is again routine to check that $[D]=0\in H_2(X,\partial X)$.
\end{proof}

\subsection{Trace homeomorphisms}
To build an exotic copy of some closed simply connected four-manifold $W$, we will want to start with a knot $K$ that is H-slice in $W$ and construct a knot $K'$ with $S^3_0(K)\cong S^3_0(K')$ such that $K'$ is hopefully not $H$-slice in $W$. We now observe that the following lemma implies that it is unproductive to construct a $K'$ with the stronger property that $X(K)$ is diffeomorphic to $X(K')$:

\begin{lemma}[Trace Embedding Lemma, originally \cite{FoxMilnor}, cf. \cite{HP19} Lemma 3.3]\label{lem:TEL}
Let $W$ be a smooth, closed 4-manifold. Then $K\subset \partial S^3$ is H-slice in $W$ if and only if there is a smooth embedding of $X(-K)$ in $W$ which induces the 0-map on $H_2$.  
\end{lemma}

Hence, in this paper we are particularly interested in knots with homeomorphic zero-surgeries which do not have diffeomorphic traces. In full generality, it is a subtle problem to demonstrate that a pair of knot traces with homomorphic boundaries are not diffeomorphic, see \cite{Yasui, HMP19}. In fact, even the easier problem of determining whether given zero surgery homeomorphism can be extended to a trace diffeomorphism is open in general. There is some luck though: many knots have that the mapping class group of $S^3_0(K)$ is just a single element, hence if there was a trace diffeomorphism it would restrict to any given boundary homeomorphism. That $MCG(S^3_0(K))=1$ for a particular knot $K$ can be verified in {\em SnapPy} and {\em Sage} \cite{SnapPy, Sage}. \setlength{\footnotemargin}{20pt} \footnote{{\em SnapPy} computes the symmetry group of a hyperbolic manifold by finding a canonical cellulation. This is done using numerical  methods. If one wants a mathematical proof, then the {\em SnapPy} answer needs to be certified rigorously, e.g. using interval arithmetic as in \cite{DHL}.\label{snap}}

Therefore, we will be especially interested in  zero-surgery homeomorphisms which do not extend to trace diffeomorphisms. First, we remind the reader that the zero-surgery homeomorphisms which do not extend to trace \emph{homeomorphisms} are well understood:

\begin{definition} A 0-surgery homeomorphism $\phi:S^3_0(K)\to S^3_0(K')$ is \emph{even} if the 4-manifold $Z:=X(K')\cup_\phi -X(K)$ has even intersection form, and is \emph{odd} otherwise. An RBG link is \emph{even} (resp. \emph{odd}) if the associated 0-surgery homeomorphism is even (resp odd). 
\end{definition}
\begin{theorem}[\cite{Boyer} Theorem 0.7 and Proposition 0.8]\label{thm:Boyer}
A 0-surgery homeomorphism $\phi:S^3_0(K)\to S^3_0(K')$ extends to a trace homeomorphism $\Phi:X(K)\to X(K')$ if and only if $\phi$ is even. 
\end{theorem}
\begin{proof}
For the reader's convenience, we include a proof of the easy `only if' direction.

Suppose for a contradiction that $\phi$ is not even, and that there is some homeomorphism $\Phi:X(K)\to X(K')$ extending $\phi$. Let $W$ denote the 4-manifold obtained from $X(K)$ by attaching a 0-framed 2-handle along $\mu_K$ followed by a 4-handle, and observe that $W$ has even intersection form. Define $Z:=X(K')\cup_\phi -X(K)$, which has odd intersection form by hypothesis. Observe that $\Phi$ gives a natural homeomorphism from $W$ to $Z$, a contradiction. 
\end{proof}

\begin{remark}
We emphasize that Theorem \ref{thm:Boyer} only shows that a particular boundary homeomorphism does not extend, \emph{not} that the traces are not homeomorphic. For example, there exist  boundary homeomorphisms $X(U)\to X(U)$ that don't extend to trace homeomorphisms \cite{gluck}.
\end{remark}

We observe that the knots whose 0-surgery admits an odd homeomorphism are somewhat restricted.

\begin{lemma}\label{lem:oddArf}
If a homeomorphism $\phi:S^3_0(K)\to S^3_0(K')$ is odd then $\Arf(K)=\Arf(K')=0$.
\end{lemma}

\begin{proof}
Robertello \cite{robertello} showed that if $X$ is a simply connected smooth 4-manifold with $S^3$ boundary and $K\subset S^3$ bounds a smooth disk $D$ in $X$ such that $[D]\in H_2(X,\partial X)\cong H_2(X)$ is characteristic then $$\Arf(K) = \frac{[D]\cdot[D]-\sigma(X)}{2} \ \text{mod 2.}$$ 
Consider the 4-manifold $Z$ obtained by gluing $X(K)$ (upside down) to $X(K')$ (right side up) along $S^3_0(K')$ via $\phi$. Remove the 0-handle of $X(K')$ to get $X$ with $S^3$ boundary, and observe that $K'\subset S^3$ bounds a 0-framed disk $D$ in $X$ (the core of the 2-handle). Further,  this handle decomposition of $X$ gives a natural presentation of $H_2$ with basis $\{\alpha,\beta\}$ and intersection form
$$Q_X=\begin{pmatrix}
0 & 1\\
1 & n
\end{pmatrix}.$$ Further, we see that $[D]=\alpha$ and, since $n$ is odd, $\alpha$ is characteristic. Then Robertello's result applies and we get $\Arf(K')=0$. We can obtain the same conclusion about $K$ by turning $Z$ upside down. 
\end{proof}

\begin{remark}
The converse is false, as can be seen by considering the identity homeomorphism on $S^3_0(K)$ for any knot with $\Arf(K)=0$. 
\end{remark}

We remark that Lemma \ref{lem:oddArf} pairs with Theorem \ref{thm:Boyer} to demonstrate that 
\begin{corollary}
If $\Arf(K)=1$ then every 0-surgery homeomorphism $\phi:S^3_0(K)\to S^3_0(K')$ extends to a trace homeomorphism. 
\end{corollary}

\subsection{Trace diffeomorphisms}
It remains well out of reach to classify when 0-surgery homeomorphisms extend to trace diffeomorphisms.  We give a sufficient condition for a homeomorphism to extend. 

\begin{definition} Let $\phi: S^3_0(K) \to S^3_0(K')$ be a $0$-surgery homeomorphism. Let $\gamma\subset S^3_0(K')$ be the framed knot given by the image of the 0-framed meridian of $K$ under $\phi$. We say that $\phi$ has \emph{property $U$} if there is some diagrammatic choice of $\gamma$ in the standard diagram of $S^3_0(K')$ where  $\gamma$ is $0$ framed and appears unknotted in the diagram. 
\end{definition}

\begin{theorem}\label{thm:propU}
If $\phi$ has property $U$ then there exists a diffeomorphism $\Phi:X(K)\to X(K')$ with $\Phi|_\partial=\phi$. 
\end{theorem}

\begin{proof}
Let $X\subset X(K)$ be a (closed) tubular neighborhood of the cocore disk of the 2-handle, which is naturally identified with $D^2\times D^2$, and let $X'\subset X$ be the open neighborhood of the cocore disk, which is a $D^2\times \intD$. Since the image $\phi(\mu_K)$ appears unknotted in the standard handle diagram of $X(K')$, we can identify the standard slice disk $\Delta$ for $\phi(\mu_K)$ in the $0$-handle $B^4$. Let $Y\cong D^2\times D^2$ be a closed tubular neighborhood of $\Delta$ and let $Y'\cong D^2\times \intD \subset Y$ an open neighborhood. Since $\phi$ preserves the 0-framing on $\mu_K$, the natural bundle diffeomorphism $F':X\to Y$ has that $F'|_{\partial(X(K))}$ agrees with $\phi|_{\nu(\mu_K)}$. 

Let $X_1$ denote $X(K)\setminus X'$ and $Y_1$ denote $X(K')\setminus Y'$. It is evident that $X_1$ is diffeomorphic to $B^4$. We will argue momentarily that $Y_1$ is also diffeomorphic to $B^4$. Assuming this for now, we will finish the argument. Observe that $\phi|_{\partial X(K)\setminus \nu(\mu_K)}$ and $F'|_{D^2\times\partial(D^2)}$ give a piecewise homeomorphism $f$ from $\partial X_1\cong S^3$ to $\partial Y_1 \cong S^3$. Since there is only one homeomorphism of $S^3$ up to isotopy \cite{Cerf}, and it extends smoothly over $B^4$, we can extend $f$ to a diffeomorphism $F:X_1\to Y_1$. By construction $F$ and $F'$ together produce a  diffeomorphism from $X(K)$ to $X(K')$ extending $\phi$. 

Now we argue that $Y_1$ is diffeomorphic to $B^4$. Observe that $Y_1$ has $S^3$ boundary and a handle decomposition given by dotting $\phi(\mu_K)$ in the standard handle decomposition of $X(K')$. So $\partial Y_1$ is naturally described as $(0,0)$ surgery on the 2-component link $K' \cup \phi(\mu_K) \subset S^3$, and one link component is an unknot. Then, by performing the 0-surgery on $\phi(\mu_K)$ first, we can think of $\partial Y_1$ as $S^3$ obtained by surgery on some knot $\ell$ in $S^1\times S^2$. By Gabai's proof of property $R$ \cite{GabaiR}, $\ell$ is isotopic to $S^1\times \{pt\}$. Performing this isotopy on the attaching sphere of our 2-handle (sliding over the 1-handle as needed) yields a handle decomposition of $Y_1$ which is just a canceling 1-2 pair, hence $Y_1$ is diffeomorphic to $B^4$. 
\end{proof}
\begin{question}\label{ques:propU}
Is the converse of Theorem \ref{thm:propU} true?
\end{question}

We comment now about these properties (parity and property $U$) for some constructions of 0-surgery homeomorphisms from the literature. (These constructions are discussed further in Section~\ref{sec:otherconstr}.) Dualizable pattern homeomorphisms \cite{Lickorish1, Brakes, GM} are all even and have property $U$. Annulus twisting  \cite{Osoinach} can be odd, and sometimes has property $U$; see Figure \ref{fig:Jmktw} and Remark \ref{rem:annulusparity}. Yasui's homeomorphisms \cite{Yasui} are constructed such that the boundary diffeomorphism extends to a homeomorphism, hence are all even. Yasui's homeomorphism sometimes does not have property $U$, in fact it may never have property $U$; see Section \ref{sec:Yasui}.

\section{Special and small RBG links}\label{sec:smallspecial}
We are interested in constructing many explicit pairs $K$ and $K'$ with homeomorphic 0-surgeries, such that $K$ and $K'$ are both simple enough that we have some hope of computing their concordance invariants or constructing slice disks for them in practice. Towards these ends, in this section we collect several user-friendly results about certain subclasses of RBG links.

\subsection{Special RBG links}
\begin{definition}
Let $\mu_R$ denote a meridian of $R$. An RBG link $L$ is \emph{special} if $b=g=0$, $r \in \Z$, and there exist link isotopies
$$R \cup B \cong R \cup \mu_R \cong R \cup G.$$ 
\end{definition} 

Notice that for a special RBG link, $H_1(S^3_{r, b, g}(L); \Z) \cong \Z$ if and only if the determinant of the framing matrix is zero. Let $l$ denote the linking number of $B$ with $G$. Then we have
\[
\begin{vmatrix}
r & 1 & 1 \\
1 & 0 & l \\
1 & l & 0
\end{vmatrix} = 2l - rl^2 = 0.
\]
Therefore, special RBG links have either $l=0$ or $rl = 2$.
\begin{figure}
{
   \fontsize{9pt}{11pt}\selectfont
   \def\svgwidth{3.8in}
   \input{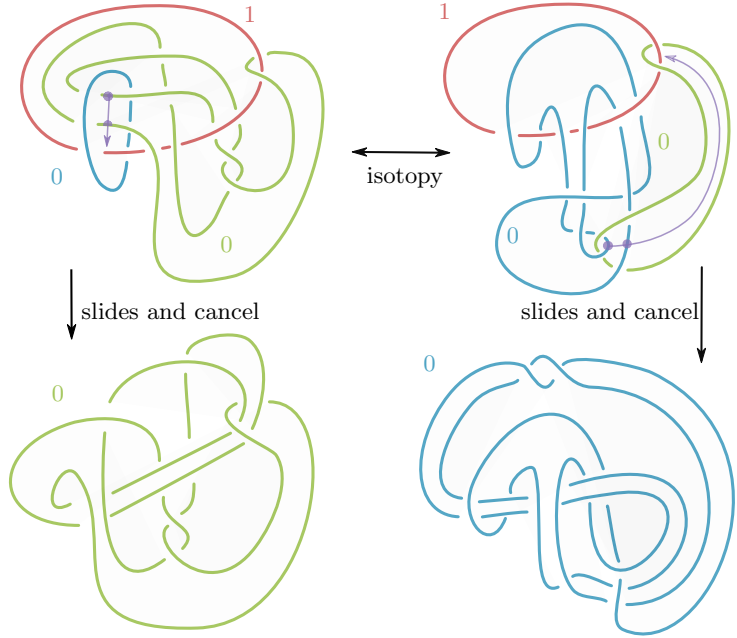}
}
\caption{A small RBG link and the associated knots $K_G$ and $K_B$.}
\label{fig:smallRBG}
\end{figure}

For a special RBG link, the homeomorphism $\psi_G:S^3_{r,0}(R,B)\to S^3$ is given by the slam-dunk homeomorphism of $B$ over $R$. 
Therefore, to exhibit the knot $K_G$  one should slide $G$ over $R$ until $G$ no longer intersects the disk $\Delta_B$ bounded by $B$. The knot $K_B$ can be exhibited in the same manner, everywhere reversing the roles of $B$ and $G$. See Figure \ref{fig:smallRBG} for an example where these slides are marked and performed.

It is easy to recognize the parity of a special RBG link:

\begin{lemma}
\label{lem:specialeven}
A special RBG link $L$ is even if and only if $r$ is. 
\end{lemma}
\begin{proof}
Let $\gamma\subset S^3_0(K_B)$ be the framed knot which is the image of the 0-framed meridian of $K_G$ under $\phi_L$. Observe that $X(K_B)\cup_{\phi_L}X(-K_G)$ admits a handle decomposition given by attaching two 2-handles to $B^4$ along the framed link $K_B \cup \gamma$ followed by a single 4-handle. Since $K_B$ is 0-framed and $[\gamma]$ generates $H_1(S^3_0(K_B))$, we have $lk(K_B,\gamma)=1$. So $X(K_B)\cup_{\phi_L}X(-K_G)$ has even intersection form if and only if the framing on $\gamma$ is even. 

Now we will argue that $\gamma$ is $r$-framed. First we observe that in the $S^3_{r,b,g}(R,B,G)$ surgery diagram of $S^3_0(K_G)$, the 0-framed meridian of $K_G$ is represented by the 0-framed meridian $\mu_G$ of $G$. Now we will consider the framed image of this in the surgery diagram of $S^3_0(K_B)$. Observe that the image of $\mu_G$ under the slam dunk homeomorphism $S^3_{r,0}(R,G)\cong S^3$ is isotopic (as a framed curve) to the curve you get from sliding $\mu_G$ over $R$ to clear the (single) intersection of $\mu_G$ with the disk $G$ bounds. See Figure~\ref{fig:gamma}. As such, $\gamma\subset S^3_0(K_B)$ looks like the (framed) ghost of $R$, in particular the framing on $\gamma$ is $r$. 
\end{proof}

\begin{figure}
{
   \fontsize{9pt}{11pt}\selectfont
   \def\svgwidth{5in}
   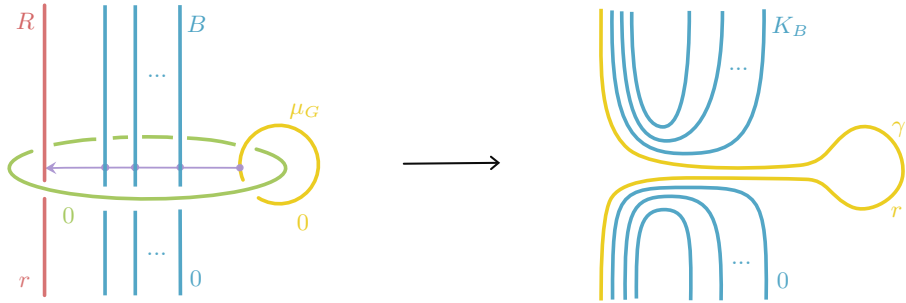
}
\caption{The curve $\gamma$ obtained from $\mu_G$ after some handle slides and a slam-dunk.}\label{fig:gamma} 
\end{figure}

\begin{lemma}\label{lem:specialpropU}
If a special RBG link has $R=U$ and $r=0$ then the associated homeomorphism has property U.
\end{lemma}
\begin{proof}
Let $L$ be a special RBG link and let $\gamma\subset S^3_0(K_B)$ be the framed knot given by image of the 0-framed meridian of $K_G$ under $\phi_L$. In the proof of Lemma \ref{lem:specialeven} we argued that  $\gamma\subset S^3_0(K_B)$ looks like the (framed) ghost of $R$, in particular $\gamma$ has the knot type and framing of $R$.
\end{proof}

Ideally, we would like to produce knots with homeomorphic zero surgeries such that exactly one of them is slice. Thus, one might like to start with their favorite slice knot $K$ and produce a distinct knot $K'$ with the same 0-surgery. It is not presently well-understood for which (slice) knots this is possible, but we give some lemmas that  sometimes help:

\begin{lemma}[\cite{Pic2} Proposition 3.2]
If $K$ has unknotting number one then there exists a special RBG link $L$ with $K_B$ isotopic to $K$.
\end{lemma}

\begin{remark}
In \cite{Pic2} it is not emphasized that the resulting RBG link is special, but one quickly checks that the proof given there does indeed produce special RBG links.
\end{remark}

\begin{lemma}\label{lem:untwisting}
If $K$ can be unknotted by performing a tangle replacement as depicted in Figure \ref{fig:untwistingmove}, then there exists a special RBG link $L$ with $K_G$ isotopic to $K$. 

\begin{figure}
{
   \fontsize{9pt}{11pt}\selectfont
   \def\svgwidth{2.5in}
\begingroup%
  \makeatletter%
  \providecommand\color[2][]{%
    \errmessage{(Inkscape) Color is used for the text in Inkscape, but the package 'color.sty' is not loaded}%
    \renewcommand\color[2][]{}%
  }%
  \providecommand\transparent[1]{%
    \errmessage{(Inkscape) Transparency is used (non-zero) for the text in Inkscape, but the package 'transparent.sty' is not loaded}%
    \renewcommand\transparent[1]{}%
  }%
  \providecommand\rotatebox[2]{#2}%
  \newcommand*\fsize{\dimexpr\f@size pt\relax}%
  \newcommand*\lineheight[1]{\fontsize{\fsize}{#1\fsize}\selectfont}%
  \ifx\svgwidth\undefined%
    \setlength{\unitlength}{595.1918335bp}%
    \ifx\svgscale\undefined%
      \relax%
    \else%
      \setlength{\unitlength}{\unitlength * \real{\svgscale}}%
    \fi%
  \else%
    \setlength{\unitlength}{\svgwidth}%
  \fi%
  \global\let\svgwidth\undefined%
  \global\let\svgscale\undefined%
  \makeatother%
  \begin{picture}(1,0.46975823)%
    \lineheight{1}%
    \setlength\tabcolsep{0pt}%
    \put(0,0){\includegraphics[width=\unitlength,page=1]{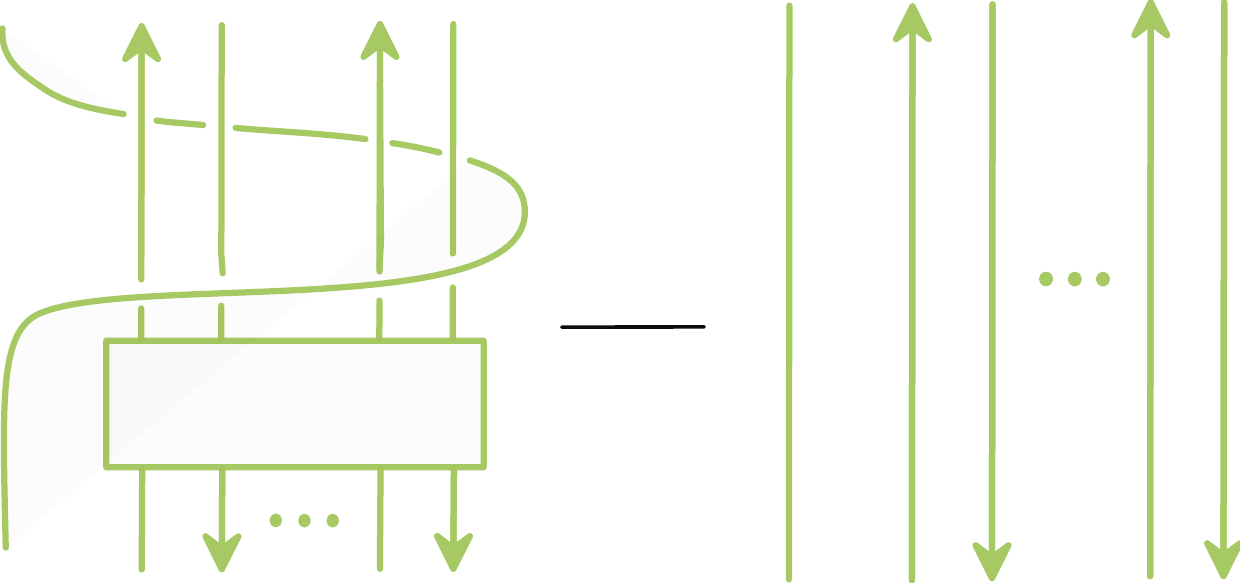}}%
    \put(0.22302736,0.13218353){\color[rgb]{0.63921569,0.77647059,0.35686275}\makebox(0,0)[lt]{\lineheight{1.25}\smash{\begin{tabular}[t]{l}$r$\end{tabular}}}}%
    \put(0,0){\includegraphics[width=\unitlength,page=2]{genunknot.pdf}}%
  \end{picture}%
\endgroup%

}
\caption{The leftmost strand may have any orientation.}\label{fig:untwistingmove} 
\end{figure}

\end{lemma}
\begin{proof}
In Figure \ref{fig:untwisting} we demonstrate a homeomorphism from $S^3_0(K)$ to $S^3_{r,b,g}(R,B,G)$ for an RBG link $L$, and it is evident from the figure that $K$ is $K_G$. 
\end{proof}
\begin{figure}
{
   \fontsize{9pt}{11pt}\selectfont
   \def\svgwidth{4.9in}
   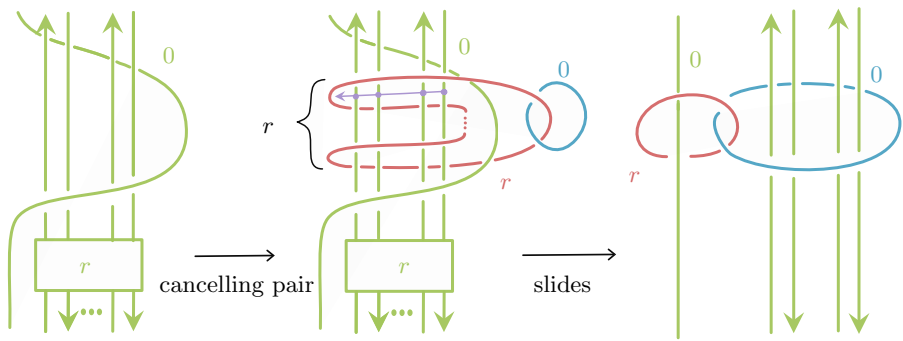
}
\caption{A homeomorphism from $S^3_0(K)$ to surgery on an RBG link $L$.}\label{fig:untwisting}
\end{figure}

\begin{remark}
In fact, all knots can be unknotted by performing a tangle replacement as depicted in Figure \ref{fig:untwistingmove} with $r=0$. (This can be readily deduced from \cite[Lemma 1]{Suz}, for example.)  However when $r=0$ one can check that the RBG link from the proof of Lemma \ref{lem:untwisting} yields knots with $K_B=K_G$. 
\end{remark}
We now give a criterion under which special RBG links produce pairs of slice knots; for our purposes this will be a setting we will want to avoid.  The following lemma is  analogous to Theorem 2.3 of \cite{Pic1}.

\begin{lemma}\label{lem:splitslice}
If $L$ is a special RBG link and $B\cup G$ is split then $K_B$ and $K_G$ are ribbon.
\end{lemma}

\begin{proof}
Since $B$ and $G$ are split, there is some sequence of crossing changes of $B$ with $R$ which turns $L$ into the link $L'=R\cup\mu_R\cup\mu_R$ in the right frame of Figure \ref{fig:slicelemma1}. Each of these crossing changes can instead be exhibited by banding $B$ to itself, as in the left frame of Figure \ref{fig:slicelemma2}, at the expense of generating an additional (blue) meridian of $R$. Therefore, by adding bands from the blue component to itself $L$ can be turned into a link $J'$ as in the right frame of Figure \ref{fig:slicelemma2} (with some number of blue meridians). Thus we can isotope the link $L$ into the form shown in the left frame of Figure \ref{fig:slicelemma3}; i.e. so that $L$ looks like $J'$ with (dual) bands connecting the blue components. 

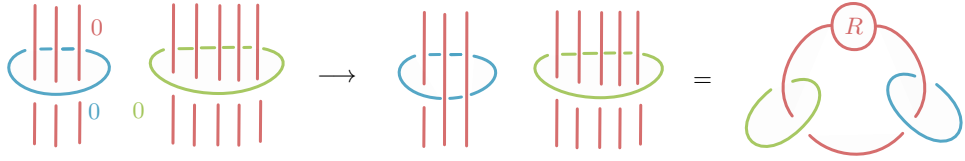
\begin{figure}
{
   \fontsize{9pt}{11pt}\selectfont
   \def\svgwidth{5in}
\begingroup%
  \makeatletter%
  \providecommand\color[2][]{%
    \errmessage{(Inkscape) Color is used for the text in Inkscape, but the package 'color.sty' is not loaded}%
    \renewcommand\color[2][]{}%
  }%
  \providecommand\transparent[1]{%
    \errmessage{(Inkscape) Transparency is used (non-zero) for the text in Inkscape, but the package 'transparent.sty' is not loaded}%
    \renewcommand\transparent[1]{}%
  }%
  \providecommand\rotatebox[2]{#2}%
  \newcommand*\fsize{\dimexpr\f@size pt\relax}%
  \newcommand*\lineheight[1]{\fontsize{\fsize}{#1\fsize}\selectfont}%
  \ifx\svgwidth\undefined%
    \setlength{\unitlength}{714.31278992bp}%
    \ifx\svgscale\undefined%
      \relax%
    \else%
      \setlength{\unitlength}{\unitlength * \real{\svgscale}}%
    \fi%
  \else%
    \setlength{\unitlength}{\svgwidth}%
  \fi%
  \global\let\svgwidth\undefined%
  \global\let\svgscale\undefined%
  \makeatother%
  \begin{picture}(1,0.18325063)%
    \lineheight{1}%
    \setlength\tabcolsep{0pt}%
    \put(0,0){\includegraphics[width=\unitlength,page=1]{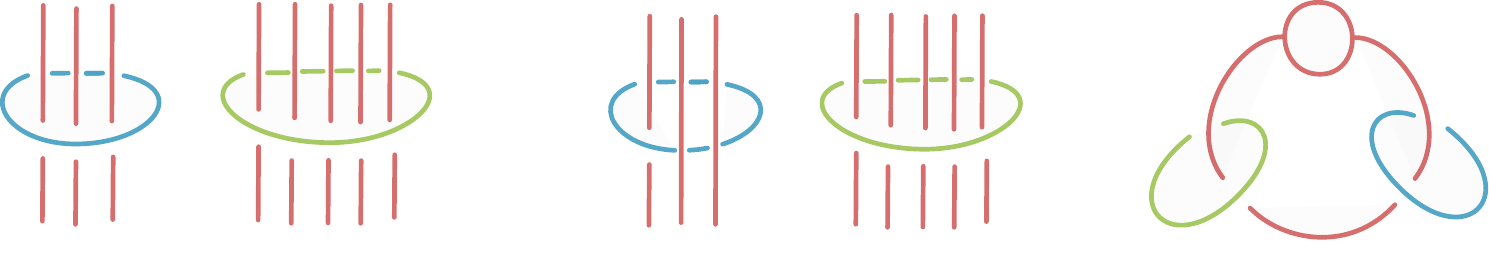}}%
    \put(0.8764075,0.15015788){\color[rgb]{0.82352941,0.4,0.4}\makebox(0,0)[lt]{\lineheight{1.25}\smash{\begin{tabular}[t]{l}$R$\end{tabular}}}}%
    \put(0.13108441,0.06159883){\color[rgb]{0.63921569,0.77647059,0.35686275}\makebox(0,0)[lt]{\lineheight{1.25}\smash{\begin{tabular}[t]{l}$0$\end{tabular}}}}%
    \put(0.08480832,0.0611944){\color[rgb]{0.30196078,0.63921569,0.76862745}\makebox(0,0)[lt]{\lineheight{1.25}\smash{\begin{tabular}[t]{l}$0$\end{tabular}}}}%
    \put(0.08741628,0.14839338){\color[rgb]{0.82352941,0.4,0.4}\makebox(0,0)[lt]{\lineheight{1.25}\smash{\begin{tabular}[t]{l}$0$\end{tabular}}}}%
    \put(0.71799546,0.09523131){\color[rgb]{0,0,0}\makebox(0,0)[lt]{\lineheight{1.25}\smash{\begin{tabular}[t]{l}=\end{tabular}}}}%
    \put(0.32437391,0.09925381){\color[rgb]{0,0,0}\makebox(0,0)[lt]{\lineheight{1.25}\smash{\begin{tabular}[t]{l}$\longrightarrow$\end{tabular}}}}%
  \end{picture}%
\endgroup%

}
\caption{Crossing changes of $B$ with $R$ to change $L$ (left frame) into $L'$ (right frame).} 
\label{fig:slicelemma1}
\end{figure}

\begin{figure}
{
   \fontsize{9pt}{11pt}\selectfont
   \def\svgwidth{5in}
\begingroup%
  \makeatletter%
  \providecommand\color[2][]{%
    \errmessage{(Inkscape) Color is used for the text in Inkscape, but the package 'color.sty' is not loaded}%
    \renewcommand\color[2][]{}%
  }%
  \providecommand\transparent[1]{%
    \errmessage{(Inkscape) Transparency is used (non-zero) for the text in Inkscape, but the package 'transparent.sty' is not loaded}%
    \renewcommand\transparent[1]{}%
  }%
  \providecommand\rotatebox[2]{#2}%
  \newcommand*\fsize{\dimexpr\f@size pt\relax}%
  \newcommand*\lineheight[1]{\fontsize{\fsize}{#1\fsize}\selectfont}%
  \ifx\svgwidth\undefined%
    \setlength{\unitlength}{635.09971619bp}%
    \ifx\svgscale\undefined%
      \relax%
    \else%
      \setlength{\unitlength}{\unitlength * \real{\svgscale}}%
    \fi%
  \else%
    \setlength{\unitlength}{\svgwidth}%
  \fi%
  \global\let\svgwidth\undefined%
  \global\let\svgscale\undefined%
  \makeatother%
  \begin{picture}(1,0.2052365)%
    \lineheight{1}%
    \setlength\tabcolsep{0pt}%
    \put(0,0){\includegraphics[width=\unitlength,page=1]{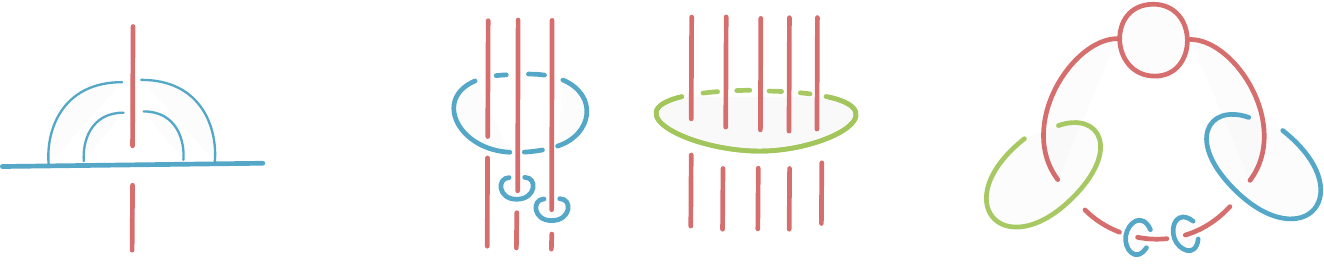}}%
    \put(0.86178284,0.1688864){\color[rgb]{0.82352941,0.4,0.4}\makebox(0,0)[lt]{\lineheight{1.25}\smash{\begin{tabular}[t]{l}$R$\end{tabular}}}}%
  \end{picture}%
\endgroup%

}
\caption{Banding $B$ to itself to turn $L$ into $J'$ (right frame).}
\label{fig:slicelemma2}
\end{figure}

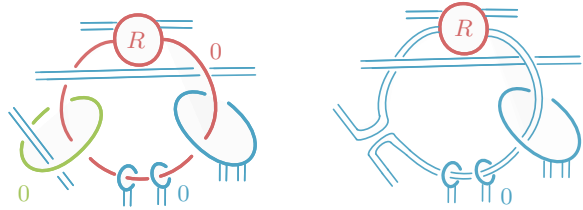
\begin{figure}
{
   \fontsize{9pt}{11pt}\selectfont
   \def\svgwidth{3in}
\begingroup%
  \makeatletter%
  \providecommand\color[2][]{%
    \errmessage{(Inkscape) Color is used for the text in Inkscape, but the package 'color.sty' is not loaded}%
    \renewcommand\color[2][]{}%
  }%
  \providecommand\transparent[1]{%
    \errmessage{(Inkscape) Transparency is used (non-zero) for the text in Inkscape, but the package 'transparent.sty' is not loaded}%
    \renewcommand\transparent[1]{}%
  }%
  \providecommand\rotatebox[2]{#2}%
  \newcommand*\fsize{\dimexpr\f@size pt\relax}%
  \newcommand*\lineheight[1]{\fontsize{\fsize}{#1\fsize}\selectfont}%
  \ifx\svgwidth\undefined%
    \setlength{\unitlength}{393.50564575bp}%
    \ifx\svgscale\undefined%
      \relax%
    \else%
      \setlength{\unitlength}{\unitlength * \real{\svgscale}}%
    \fi%
  \else%
    \setlength{\unitlength}{\svgwidth}%
  \fi%
  \global\let\svgwidth\undefined%
  \global\let\svgscale\undefined%
  \makeatother%
  \begin{picture}(1,0.36628743)%
    \lineheight{1}%
    \setlength\tabcolsep{0pt}%
    \put(0,0){\includegraphics[width=\unitlength,page=1]{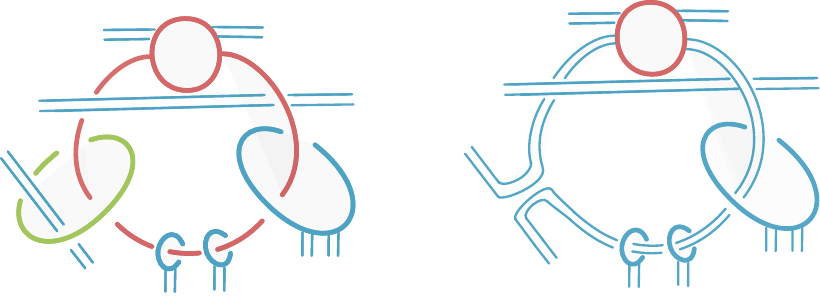}}%
    \put(0.20607876,0.28547397){\color[rgb]{0.82352941,0.4,0.4}\makebox(0,0)[lt]{\lineheight{1.25}\smash{\begin{tabular}[t]{l}$R$\end{tabular}}}}%
    \put(0.77781686,0.30601547){\color[rgb]{0.82352941,0.4,0.4}\makebox(0,0)[lt]{\lineheight{1.25}\smash{\begin{tabular}[t]{l}$R$\end{tabular}}}}%
    \put(0.85883731,0.0134384){\color[rgb]{0.30196078,0.63921569,0.76862745}\makebox(0,0)[lt]{\lineheight{1.25}\smash{\begin{tabular}[t]{l}$0$\end{tabular}}}}%
    \put(0.29524161,0.01899654){\color[rgb]{0.30196078,0.63921569,0.76862745}\makebox(0,0)[lt]{\lineheight{1.25}\smash{\begin{tabular}[t]{l}$0$\end{tabular}}}}%
    \put(0.01680596,0.01841966){\color[rgb]{0.63921569,0.77254902,0.35686275}\makebox(0,0)[lt]{\lineheight{1.25}\smash{\begin{tabular}[t]{l}$0$\end{tabular}}}}%
    \put(0.3514924,0.26646695){\color[rgb]{0.82352941,0.4,0.4}\makebox(0,0)[lt]{\lineheight{1.25}\smash{\begin{tabular}[t]{l}$0$\end{tabular}}}}%
  \end{picture}%
\endgroup%

}
\caption{The dual bands from $J'$ give a diagram of $L$ from which it is easy to spot that $K_B$ is slice.}
\label{fig:slicelemma3}
\end{figure}

From this picture of $L$, we can easily cancel $G\cup R$ so that we are left with a surgery diagram of $K_B$; to perform the cancellation we just first  have to slide every band that geometrically links $G$ across $R$. As such, we see that $K_B$ has a diagram of the form shown in the right frame of Figure \ref{fig:slicelemma3}. That $K_B$ is ribbon follows from the picture. The claim that $K_G$ is ribbon  follows by symmetry of hypothesis. 
\end{proof}

\subsection{Small RBG links}

\begin{definition}
\label{def:small}
An RBG link $L$ is \emph{small} if $L$ is special and in addition 
\begin{itemize}
\item $B$ bounds a properly embedded disk $\Delta_B$ that intersects $R$ in exactly one point, and intersects $G$ in at most $2$ points.
\item $G$ bounds a properly embedded disk $\Delta_G$ that intersects $R$ in exactly one point, and intersects $B$ in at most $2$ points.
\end{itemize} 
(All intersections are required to be transverse.)
\end{definition}

\begin{example}\label{ex:smallrbg}
Consider the small RBG link $L$ and associated knots $K_B$ and $K_G$ in Figure \ref{fig:smallRBG}. Our natural diagrams of $K_B$ and $K_G$ have $20$ and $16$ crossings, respectively. Using SnapPy to identify the complements, we find that in fact $K_B$ is the 12 crossing knot $12n309$ and $K_G$ is the 14 crossing knot $14n14254$. To the authors' knowledge, this pair minimizes $c(K)+c(K')$ among all pairs of distinct knots with $S^3_0(K)\cong S^3_0(K')$ in the literature. 
\end{example}

Notice that the diagrammatic conditions making a special RBG link small ensure that one only has to perform at most two slides of $G$ over $R$ in order to exhibit the knot $K_G$ (resp. two slides of $B$ over $R$ to exhibit $K_B$). This helps keep the crossing number of $K_G$ (resp. $K_B$) somewhat small.

 In fact, for a small RBG link to produce $K_B\neq K_G$ we show that two slides are necessary.

\begin{proposition}\label{prop:notthatsmall}
If $L$ is a small RBG link with $\Delta_B$ intersecting $G$ in less than $2$ points then $K_B=K_G$.
\end{proposition}

To prove the proposition, we require two lemmas.

\begin{lemma}\label{lem:delta1}
Let $L$ be a small RBG link and $\Delta_B$ and $\Delta_G$ be disks as in Definition~\ref{def:small}. If $\Delta_B\cap\Delta_G$ is a single non-proper arc in each of $\Delta_B$ and $\Delta_G$, then $K_B=K_G$.
\end{lemma}
\begin{proof}
Since we also know that both $\Delta_B$ and $\Delta_G$ intersect $R$ in exactly one point, we can isotope $L$ in $S^3$ so that a neighborhood of $\Delta_B\cup\Delta_G$ looks as in the middle frame of Figure \ref{fig:smallRBG1int}. Performing the slam dunks to exhibit $K_B$ and $K_G$ from this picture of $L$ yields diagrams of $K_B$ and $K_G$ as in the left and right frames of Figure \ref{fig:smallRBG1int}; these diagrams are identical outside the neighborhood shown and isotopic inside the neighborhood shown. 
\end{proof}
\begin{figure}
{
   \fontsize{9pt}{11pt}\selectfont
   \def\svgwidth{4in}
\begingroup%
  \makeatletter%
  \providecommand\color[2][]{%
    \errmessage{(Inkscape) Color is used for the text in Inkscape, but the package 'color.sty' is not loaded}%
    \renewcommand\color[2][]{}%
  }%
  \providecommand\transparent[1]{%
    \errmessage{(Inkscape) Transparency is used (non-zero) for the text in Inkscape, but the package 'transparent.sty' is not loaded}%
    \renewcommand\transparent[1]{}%
  }%
  \providecommand\rotatebox[2]{#2}%
  \newcommand*\fsize{\dimexpr\f@size pt\relax}%
  \newcommand*\lineheight[1]{\fontsize{\fsize}{#1\fsize}\selectfont}%
  \ifx\svgwidth\undefined%
    \setlength{\unitlength}{310.74865723bp}%
    \ifx\svgscale\undefined%
      \relax%
    \else%
      \setlength{\unitlength}{\unitlength * \real{\svgscale}}%
    \fi%
  \else%
    \setlength{\unitlength}{\svgwidth}%
  \fi%
  \global\let\svgwidth\undefined%
  \global\let\svgscale\undefined%
  \makeatother%
  \begin{picture}(1,0.23487287)%
    \lineheight{1}%
    \setlength\tabcolsep{0pt}%
    \put(0.37131836,0.02488774){\color[rgb]{0.83137255,0.40392157,0.40392157}\makebox(0,0)[lt]{\lineheight{1.25}\smash{\begin{tabular}[t]{l}$\pm 2$\end{tabular}}}}%
    \put(0,0){\includegraphics[width=\unitlength,page=1]{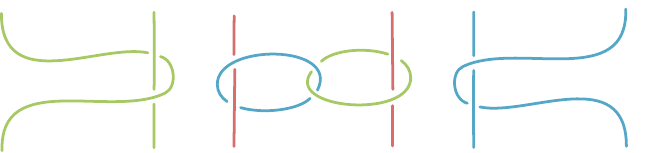}}%
    \put(0.05706209,0.03550092){\color[rgb]{0.63921569,0.77647059,0.35686275}\makebox(0,0)[lt]{\lineheight{1.25}\smash{\begin{tabular}[t]{l}$0$\end{tabular}}}}%
    \put(0.85950226,0.0431634){\color[rgb]{0.30196078,0.63921569,0.76862745}\makebox(0,0)[lt]{\lineheight{1.25}\smash{\begin{tabular}[t]{l}$0$\end{tabular}}}}%
    \put(0.37721808,0.16058152){\color[rgb]{0.30196078,0.63921569,0.76862745}\makebox(0,0)[lt]{\lineheight{1.25}\smash{\begin{tabular}[t]{l}$0$\end{tabular}}}}%
    \put(0.54880817,0.03599885){\color[rgb]{0.63921569,0.77647059,0.35686275}\makebox(0,0)[lt]{\lineheight{1.25}\smash{\begin{tabular}[t]{l}$0$\end{tabular}}}}%
    \put(0,0){\includegraphics[width=\unitlength,page=2]{smallrbg1int.pdf}}%
  \end{picture}%
\endgroup%

}
\caption{A small RBG link with $\Delta_B\cap\Delta_G$ a single arc in each disk.}
\label{fig:smallRBG1int}
\end{figure}

\begin{lemma}\label{lem:delta1boring}
Let $L$ be a small RBG link and $\Delta_B$ and $\Delta_G$ be disks as in Definition~\ref{def:small}. If $|\Delta_B \cap G|=1$ then, after an isotopy of $\Delta_G$ rel boundary, we can arrange so that $\Delta_B\cap \Delta_G$ is as described in Lemma \ref{lem:delta1}. 
\end{lemma}

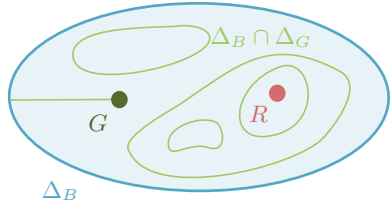
\begin{figure}
{
   \fontsize{9pt}{11pt}\selectfont
   \def\svgwidth{2in}
\begingroup%
  \makeatletter%
  \providecommand\color[2][]{%
    \errmessage{(Inkscape) Color is used for the text in Inkscape, but the package 'color.sty' is not loaded}%
    \renewcommand\color[2][]{}%
  }%
  \providecommand\transparent[1]{%
    \errmessage{(Inkscape) Transparency is used (non-zero) for the text in Inkscape, but the package 'transparent.sty' is not loaded}%
    \renewcommand\transparent[1]{}%
  }%
  \providecommand\rotatebox[2]{#2}%
  \newcommand*\fsize{\dimexpr\f@size pt\relax}%
  \newcommand*\lineheight[1]{\fontsize{\fsize}{#1\fsize}\selectfont}%
  \ifx\svgwidth\undefined%
    \setlength{\unitlength}{315.97194672bp}%
    \ifx\svgscale\undefined%
      \relax%
    \else%
      \setlength{\unitlength}{\unitlength * \real{\svgscale}}%
    \fi%
  \else%
    \setlength{\unitlength}{\svgwidth}%
  \fi%
  \global\let\svgwidth\undefined%
  \global\let\svgscale\undefined%
  \makeatother%
  \begin{picture}(1,0.52916192)%
    \lineheight{1}%
    \setlength\tabcolsep{0pt}%
    \put(0,0){\includegraphics[width=\unitlength,page=1]{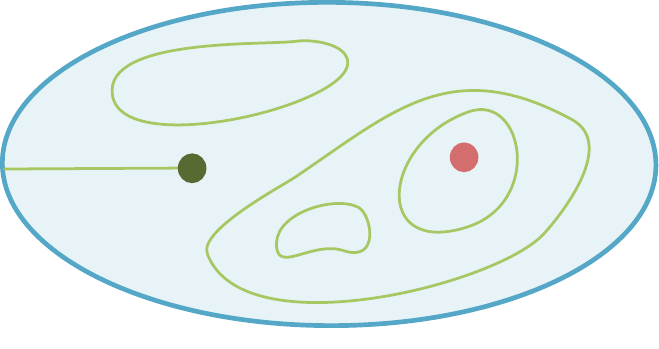}}%
    \put(0.08846345,0.01673593){\color[rgb]{0.30196078,0.63921569,0.76862745}\makebox(0,0)[lt]{\lineheight{1.25}\smash{\begin{tabular}[t]{l}$\Delta_B$\end{tabular}}}}%
    \put(0.53028126,0.4185097){\color[rgb]{0.63529412,0.77647059,0.36078431}\makebox(0,0)[lt]{\lineheight{1.25}\smash{\begin{tabular}[t]{l}$\Delta_B\cap\Delta_G$\end{tabular}}}}%
    \put(0.2110812,0.19205901){\color[rgb]{0.31372549,0.39607843,0.16078431}\makebox(0,0)[lt]{\lineheight{1.25}\smash{\begin{tabular}[t]{l}$G$\end{tabular}}}}%
    \put(0.63176358,0.21521058){\color[rgb]{0.82745098,0.40784314,0.40784314}\makebox(0,0)[lt]{\lineheight{1.25}\smash{\begin{tabular}[t]{l}$R$\end{tabular}}}}%
  \end{picture}%
\endgroup%

}
\caption{The disk $\Delta_B$ for a small RBG link with $\Delta_B\cap G$ a single point.}
\label{fig:deltaB}
\end{figure}

\begin{proof}
Since $\Delta_B$ intersects $G$ once geometrically, any disk bounded by $G$ must intersect $B$ once algebraically. Thus, the hypothesis that $L$ is small implies that $\Delta_G$ intersects $B$ exactly once geometrically. 

Consider the intersection $\Delta_B\cap\Delta_G\hookrightarrow \Delta_B$; by general position this is a 1-dimensional submanifold of $\Delta_B$(with a finite number of components, which are not necessarily properly embedded). Since $G\cap\Delta_B$ is a point, there must be exactly one arc of intersection that has exactly one endpoint in $int(\Delta_B)$. The same analysis of $\Delta_B\cap\Delta_G\hookrightarrow \Delta_G$ shows that there is exactly one arc in $\Delta_G$ that has exactly one endpoint in $int(\Delta_G)$. Endpoints of arcs in $int(\Delta_G)$ correspond to endpoints of arcs in $\partial \Delta_B$. Thus there is exactly one arc in $\Delta_B$ that has exactly one endpoint in $\partial \Delta_B$. We deduce that there is exactly one arc of intersection in $\Delta_B$, and it has one endpoint in $int(\Delta_B)$ and one endpoint in $\partial \Delta_B$. 

The other intersections are circles on $\Delta_B$. See Figure \ref{fig:deltaB}. Recall that $\Delta_B\cap R=\{pt\}$, so each circle either contains $\{pt\}$ or does not. Consider an innermost circle $\gamma\subset \Delta_B$ which does not contain $\{pt\}$. Then $\gamma$ bounds a disk $\Delta'_B \subset \Delta_B$ with $\Delta'_B \cap R = \emptyset$. Since $\gamma$ is also a circle in $\Delta_G$, it bounds a disk $\Delta'_G \subset \Delta_G$. Note that $\Delta'_G$ does not contain $\{pt'\}:=\Delta_G\cap R$ (since if it did the disks $\Delta_B'$ and $\Delta_G'$ would give an $S^2\looparrowright S^3$ with $R\cap S^2=\{pt\}$, which is impossible for homology reasons). Thus we can replace $\Delta_G'\subset \Delta_G$ with a parallel copy of $\Delta_B'$; this yields a new $\Delta_G$ that has fewer circles of intersection with $\Delta_B$. This replacement does not change the arc intersections of $\Delta_B$ with $\Delta_G$ nor the intersections of either disk with $R$. Continue until no such $\gamma$ remain.

Now consider an innermost circle $\gamma$ on $\Delta_B$ which does contain $\{pt\}$. By a similar analysis, $\gamma\subset\Delta_G$ bounds a disk $\Delta'_G \subset \Delta_G$ which does contain $\{pt'\}$. Now replacing $\Delta_G'\subset \Delta_G$ with a parallel copy of $\Delta'_B$ does change $\Delta_G\cap R$, but the new $\Delta_G\cap R$ is still a single point. Continue until no such $\gamma$ remain.  
\end{proof}

\begin{proof}[Proof of Proposition \ref{prop:notthatsmall}]
The case that $\Delta_B\cap G$ in a single point is a consequece of Lemmas \ref{lem:delta1boring} and \ref{lem:delta1}. In the case that $\Delta_B\cap G$ is empty, use $\Delta_B$ to isotope $B$ to a meridian of $R$ which does not link $G$. Since $R\cup G\cong R\cup \mu_G$, we can isotope $B\cup R\cup G$ to $\mu_R\cup R\cup \mu_R$. From here the slam dunk homeomorphisms give $K_B=B=U=G=K_G$. 
\end{proof}

\begin{remark}
When $L$ is small and $\Delta_B$ intersects $G$ in two points, it is possible that $\Delta_B\cap\Delta_G$ contains no circles of intersection and yet the knots $K_B$ and $K_G$ are distinct; the reader can check that the RBG link in Example \ref{ex:smallrbg} has this property. 
\end{remark} 

In view of Proposition~\ref{prop:notthatsmall} (and its analogue with $B$ and $G$ switched), we see that if we want to produce $K_B \neq K_G$ from a small RBG link, we should only consider the cases when $|\Delta_B \cap G|=|\Delta_G \cap B|=2$.

\section{Computer experiments}
\label{sec:computer}
\subsection{A six-parameter family of RBG links}
To generate many examples of pairs of knots with the same $0$-surgery, we studied the family of small RBG links shown in Figure~\ref{fig:RBGforbigdata}. The link depends on $6$ parameters ($a$, $b$, $c$, $d$, $e$,  $f$), corresponding to the numbers of full twists in each box. The first parameter $a$ represents full twists between $R$ and the 2-handle framing curve of $R$. We also have twists between $R$ and its 2-handle framing in the box with $b$ full twists, so overall the red curve has framing $$r=a+b.$$ Thus, in view of Lemma~\ref{lem:specialeven}, the parity of the RBG link is given by $a+b$ mod $2$.

The RBG link from Figure~\ref{fig:RBGforbigdata} produces the knots $K_G(a, b, c, d, e, f)$ and $K_B(a, b, c, d, e, f)$ with the same $0$-surgery, as shown in Figure~\ref{fig:KbKg}. We investigated these knots for values of the parameters ranging between $-2$ and $2$ for the full twists on two strands, and between $-1$ and $1$ for the full twists on four strands:
$$ a, c, e \in [-2,2], \ \ \ \  b, d, f \in [-1, 1].$$
These parameters ensure that the crossing numbers of the knots $K_B$ and $K_G$ are at most 55. 
\begin{figure}
{
   \fontsize{9pt}{11pt}\selectfont 
   \def\svgwidth{2.4in}
   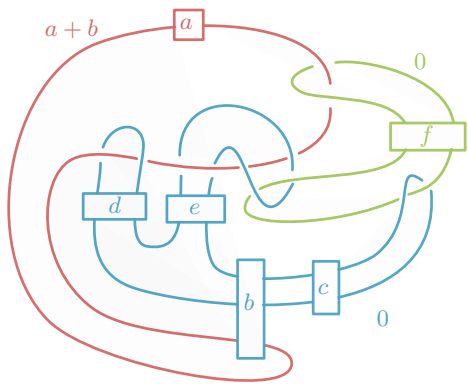
}
\caption{The 6-parameter RBG link which generated our 3375 examples of 0-surgery homeomorphisms.}
\label{fig:RBGforbigdata}
\end{figure}

\begin{figure}
{
   \fontsize{9pt}{11pt}\selectfont
   \def\svgwidth{4.8in}
   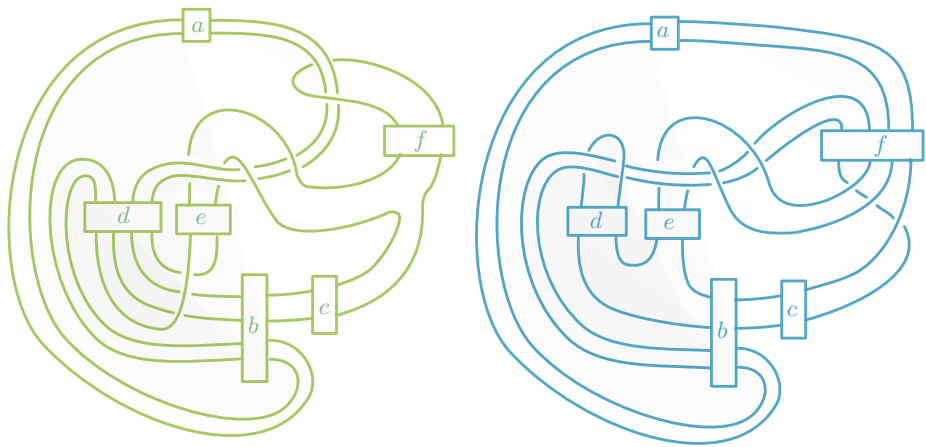
}
\caption{Knots $K_G(a, b, c, d, e, f)$ and $K_B(a, b, c, d, e, f)$ share a 0-surgery.}
\label{fig:KbKg}
\end{figure}

We obtained a family of $3375$ pairs of knots. We used SnapPy \cite{SnapPy} to generate a list of these knots, compute their hyperbolic volumes, and identify some of them with knots from knot tables. Our family is sufficiently general that it includes many small knots; for example, out of the 31 hyperbolic knots with at most 8 crossings, SnapPy recognized 19 among our data. In fact, one can show that all the knots in our family have Seifert genus at most 2 and four-ball genus at most 1; our knots include 19 of the 21 hyperbolic knots with at most 8 crossings that have these properties. 

The results of our investigations are described below, and supporting files can be found online at 
\href{http://web.stanford.edu/~cm5/RBG.html}{\tt http://web.stanford.edu/$\sim$cm5/RBG.html}.

\subsection{Methodology}
We searched our list for promising pairs of knots, in particular pairs such that one knot has $s < 0$ (and therefore is not H-slice in any $\#^n \CP$), whereas the other has $s=\sigma=0$, so has a chance of being H-slice in some $\#^n \CP$ (or perhaps even slice in $B^4$). (For the reason we considered $s<0$ instead of $s>0$, see Remark \ref{rem:sign}.) From the start, we could exclude some such pairs from the promising sublist using the following lemma.

\begin{lemma}
\label{lemma:ab}
(a) If $b=-1$, then $K_G(a, b, c, d, e, f)=K_B(a, b, c, d, e, f)$.

(b) If $a+b=0$, then the  knots $K_G(a, b, c, d, e, f)$ and $K_B(a, b, c, d, e, f)$ have diffeomorphic traces.
\end{lemma}
Item (b) is relevant because if the knots have diffeomorphic traces and one has $s<0$ then the other is not slice in any $\#^n \CP$; cf. Lemma~\ref{lem:TEL}. 

\begin{proof}[Proof of Lemma~\ref{lemma:ab}]
(a) One can check that for $b=-1$, the RBG link from Figure~\ref{fig:RBGforbigdata} is isotopic to the one shown in Figure~\ref{fig:bminus1}, which has a symmetry exchanging the $B$ and $G$ components. Therefore, the resulting $K_B$ and $K_G$ knots are isotopic. 
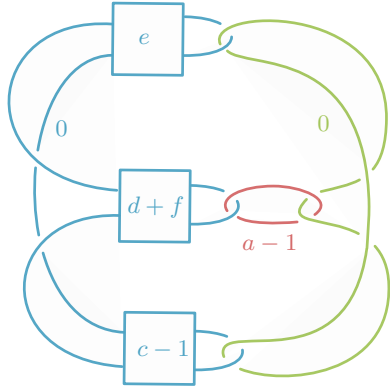
\begin{figure}
{
   \fontsize{9pt}{11pt}\selectfont 
   \def\svgwidth{2in}
\begingroup%
  \makeatletter%
  \providecommand\color[2][]{%
    \errmessage{(Inkscape) Color is used for the text in Inkscape, but the package 'color.sty' is not loaded}%
    \renewcommand\color[2][]{}%
  }%
  \providecommand\transparent[1]{%
    \errmessage{(Inkscape) Transparency is used (non-zero) for the text in Inkscape, but the package 'transparent.sty' is not loaded}%
    \renewcommand\transparent[1]{}%
  }%
  \providecommand\rotatebox[2]{#2}%
  \newcommand*\fsize{\dimexpr\f@size pt\relax}%
  \newcommand*\lineheight[1]{\fontsize{\fsize}{#1\fsize}\selectfont}%
  \ifx\svgwidth\undefined%
    \setlength{\unitlength}{214.98561552bp}%
    \ifx\svgscale\undefined%
      \relax%
    \else%
      \setlength{\unitlength}{\unitlength * \real{\svgscale}}%
    \fi%
  \else%
    \setlength{\unitlength}{\svgwidth}%
  \fi%
  \global\let\svgwidth\undefined%
  \global\let\svgscale\undefined%
  \makeatother%
  \begin{picture}(1,0.92751875)%
    \lineheight{1}%
    \setlength\tabcolsep{0pt}%
    \put(0,0){\includegraphics[width=\unitlength,page=1]{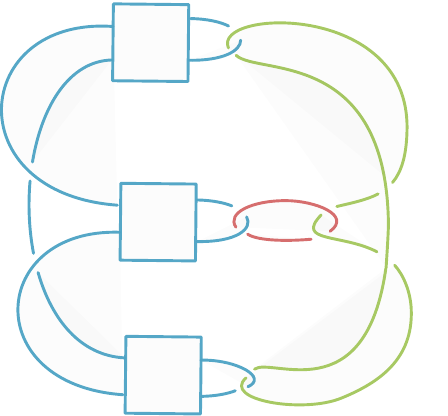}}%
    \put(0.31425584,0.82013511){\color[rgb]{0.30196078,0.63921569,0.76862745}\makebox(0,0)[lt]{\lineheight{1.25}\smash{\begin{tabular}[t]{l}$e$\end{tabular}}}}%
    \put(0.28743545,0.41482429){\color[rgb]{0.30196078,0.63921569,0.76862745}\makebox(0,0)[lt]{\lineheight{1.25}\smash{\begin{tabular}[t]{l}$d+f$\end{tabular}}}}%
    \put(0.29914493,0.07415134){\color[rgb]{0.30196078,0.63921569,0.76862745}\makebox(0,0)[lt]{\lineheight{1.25}\smash{\begin{tabular}[t]{l}$c-1$\end{tabular}}}}%
    \put(0.11526265,0.59887882){\color[rgb]{0.30196078,0.63921569,0.76862745}\makebox(0,0)[lt]{\lineheight{1.25}\smash{\begin{tabular}[t]{l}$0$\end{tabular}}}}%
    \put(0.74326044,0.61064703){\color[rgb]{0.63921569,0.77647059,0.35686275}\makebox(0,0)[lt]{\lineheight{1.25}\smash{\begin{tabular}[t]{l}$0$\end{tabular}}}}%
    \put(0.56290545,0.32526509){\color[rgb]{0.83137255,0.40392157,0.40392157}\makebox(0,0)[lt]{\lineheight{1.25}\smash{\begin{tabular}[t]{l}$a-1$\end{tabular}}}}%
  \end{picture}%
\endgroup%

}
\caption{The RBG link from Figure~\ref{fig:RBGforbigdata}, when $b=-1$. }
\label{fig:bminus1}
\end{figure}

(b) By Lemma \ref{lem:specialpropU}, the fact that $r=0$ implies that the 0-surgery homeomorphism has property $U$, and therefore (by Theorem~\ref{thm:propU}) it extends to a trace diffeomorphism. \end{proof}

We then computed the signatures and the $s$ invariants of the remaining $1800$ pairs of knots. To do this, we wrote a general formula for the DT (Dowker-Thistlethwaite) codes of the knots $K_G(a, b, c, d, e, f)$ and $K_B(a, b, c, d, e, f)$, depending on the six parameters. We plugged it into {\em Mathematica} \cite{mathematica} and used the functions {\tt KnotSignature} and {\tt sInvariant} from the \emph{KnotTheory\!\`{}} package \cite{KnotTheory}. (Note that the signature is a $0$-surgery invariant, so the signature of $K_B$ always equals that of $K_G$.) For the $s$ invariant, for some values of the parameters, the program took too long to compute (more than a few minutes). In such cases we had the option of simplifying the knot diagram in {\em SnapPy} with {\em Sage} \cite{SnapPy, Sage}, using the command $\mathtt{K.simplify('global')}$. We then plugged it back into either \emph{KnotTheory\!\`{}} or \emph{SKnotJob} \cite{SKnotJob}, tried to decrease girth as in \cite[Section 5.1]{FGMW}, and then re-compute. 

We also used the following trick to reduce the number of knots for which we had to explicitly compute $s$:
If we found two knots of the same type ($K_B$ or $K_G$) corresponding to values $(a, b, c, d, e, f)$ and $(a'', b'', c'', d'', e'', f'')$ with
$$ a \leq a'', \ \ b \leq b'', \ \ c\leq c'', \ \ d \leq d'' , \ \ e \leq e'', \ \  f \leq f''$$
and with the same value of $s$, then we knew this value of $s$ also holds for all the knots of the same type with intermediate values $(a', b', c', d', e', f')$, i.e. such that
$$ a \leq a' \leq a'', \ \ b \leq b' \leq b'', \ \ c\leq c' \leq c'', \ \ d \leq d'\leq d'' , \ \ e \leq e' \leq e'', \ \  f \leq f' \leq f''.$$
This is a consequence of the monotonicity of the $s$-invariant under generalized crossing changes, which was proved in \cite[Theorem 1.1]{MMSW}.

\subsection{An exclusion}
As a result of our calculations, we found $24$ pairs of knots with $\sigma=0$, such that one knot in the pair has $s=0$ and the other $s=-2$. If the knot with $s=0$ were H-slice in $\#^n \CP$, this would produce an exotic $\#^n \CP$. For one of the $24$ knots with $s=0$, namely
$$K= K_G(0,1,0,-1,2,1) $$
we could prove that $K$ is not H-slice in $\#^n \CP$, using the following method.

If the knot $K_G(0,1,0,-1,2,1)$ were H-slice in some $\#^n \CP$, then $K_G(-1,1,0,-1,2,1)$, which differs from it by a crossing change from negative to positive, would be H-slice in $\#^{n+1} \CP$ by Lemma~\ref{lem:HCP}. However, $K_G(-1,1,0,-1,2,1)$ has the same trace as $K_B(-1,1,0,-1,2,1)$ by Lemma~\ref{lemma:ab}(b). Therefore, $K_B(-1,1,0,-1,2,1)$ would also be H-slice in $\#^{n+1} \CP$, according to Lemma~\ref{lem:TEL}. Direct computation using \emph{KnotTheory\!\`{}} gives that 
$$s(K_B(-1,1,0,-1,2,1)) = -2,$$
which contradicts \eqref{eq:sless}.

\subsection{Interesting examples} 
We were left with $23$ promising pairs of knots. For the knot in each pair with $s=0$ (i.e. each candidate for being H-slice in $\#^n \CP$), we simplified the diagrams in {\em SnapPy} with {\em Sage}, using $\mathtt{K.simplify('global')}$. The resulting diagrams are shown in Figures~\ref{fig:21a} and \ref{fig:2}. In Table~\ref{tab:table1} we list the apparent number of crossings (in the simplified diagram), volume and Alexander polynomials of these knots. 

\renewcommand{\arraystretch}{1.2}
\begin{table}
  \begin{center}
    \begin{tabular}{|c| l |  c | c | l |} 
    \hline
     Name & Identifier & \# crossings & Volume & Alexander polynomial \\
     \hline
    $K_1$ & $ K_B(0, 1, 0, 1, 2, -1)$ & $29$ & 15.451403388 & $1$ \\
    $K_2$ & $K_B(1,1,0,1,1,-1)$ & 29 & 14.698440095 & $1$ \\
    $K_3$ &  $K_G(1, 1, 0,-1, 1, 1)$ & 32 & 20.930658865 & $1$\\
    $K_{4}$ &  $K_B(2, 1, -1, 1, 1, -1)$  & 29 & 15.552102256 & $1$ \\
   $K_{5}$ & $K_G(2, 1, -1, -1, 1, 1)$ & 32 & 21.888892554 & $1$ \\     
   $ K_6$ &  $K_B(1, 1, -1, 1, 2, -1)$ & 29 & 16.583603453 & $t^2 -2t+ 3 -2t^{-1} + t^{-2}$  \\
   $K_7$ & $K_B(1, 1, 1, 1, 0, -1)$ & 31 & 18.694759676 & $t^2 -6t + 11 -6t^{-1} + t^{-2}$ \\
   $K_8$ & $ K_B(1, 1, 2, 1, -1, -1)$ & 36 & 21.768651216 & $ 4t^2 - 20t +33 - 20t^{-1} + 4 t^{-2}$\\
   $K_9$ & $K_G(1, 1, -1, -1, 2, 1)$ & 32 & 21.917877366 & $t^2 - 2t+ 3-2t^{-1} + t^{-2}$ \\
   $K_{10}$ & $K_G(1,1,1,-1,0,1)$ & 32 &  23.276452522 & $t^2 -6t + 11 -6t^{-1} + t^{-2}$\\
   $K_{11}$ & $K_G(1, 1, 2, -1, -1, 1)$ & 41 & 25.720923264 & $4t^2 -20t + 33 -20t^{-1} + 4t^{-2}$\\
   $K_{12}$ & $K_G(1, 1, 2, 0, -1, 1)$ & 20 & 20.032239211 & $2t^2-12t+21-12t^{-1}+2t^{-2}$\\
   $K_{13}$ & $K_B(2, 1,-2, 1, 2, -1)$ & 35 &  18.623983982 & $2t^2-6t+9-6t^{-1}+2t^{-2}$\\
   $K_{14}$ & $K_B(2, 1, 0, 1, 0, -1)$ & 31 & 16.662235002 & $-2t+5 - 2t^{-1}$ \\
   $K_{15}$ & $K_B(2, 1, 1, 1, -1, -1)$ & 33 & 20.505101934 & $2t^2-12t+21-12t^{-1}+2t^{-2}$ \\
   $K_{16}$ &  $K_B(2, 1, 2, 1, -2, -1)$ & 37 & 22.919098178 & $6t^2 - 30t + 49 -30t^{-1} + 6t^{-2}$ \\
   $K_{17}$ & $K_G(2,1,-2,-1,2,1)$ &37 &23.396805316 &  $2t^2-6t+9-6t^{-1}+2t^{-2}$ \\
   $K_{18}$ & $K_G(2,1,-2,0,2,1)$ & 16 & 17.009749601 & $t^2 -2t+ 3 -2t^{-1} + t^{-2}$ \\
   $K_{19}$ & $K_G(2, 1, 0, -1, 0, 1)$ & 34 & 22.384645541 & $-2t+5 - 2t^{-1}$\\ 
   $K_{20}$ & $K_G(2, 1, 1, -1, -1, 1)$ & 36 & 24.655381040 & $2t^2-12t+21-12t^{-1}+2t^{-2}$ \\
   $K_{21}$ & $K_G(2, 1, 2, -1, -2, 1)$ & 42 & 26.731842490 & $2t^2-12t+21-12t^{-1}+2t^{-2}$ \\
   $K_{22}$ & $K_G(2, 1, 1, 0, -1, 1)$ & 18 & 19.113083865 & $t^2 -8t+15 - 8t^{-1} + t^{-2}$\\
   $K_{23}$ & $K_G(2, 1, 2, 0, -2, 1)$ & 22 & 21.642574192 & $5t^2 - 26t + 43 - 26t^{-1} + 5t^{-2}$\\   \hline
    \end{tabular}
  \end{center}
  \caption{Examples coming out of our computer experiments.} 
     \label{tab:table1}
\end{table}

The knots $K_1$ through $K_5$ have trivial Alexander polynomial, and are therefore topologically slice by Freedman's theorem \cite{Freedman}; they are candidates for being slice. The knots $K_6$ through $K_{21}$ satisfy the Fox-Milnor condition on the Alexander polynomial, but were shown to not be topologically slice by Dunfield and Gong \cite{DG}, using the program \cite{DG2} for computing twisted Alexander polynomials. The last two knots $K_{22}$ and $K_{23}$ do not satisfy the Fox-Milnor condition. Still, all of these knots are candidates for being H-slice in $\#^n \CP$.

\begin{proof}[Proofs of Theorems~\ref{thm:21} and \ref{thm:23}] If $K_i$ ($i=1,\dots, 23$) is a knot from our list, then there is a companion knot $K'_i$ such that $S^3_0(K) \cong S^3_0(K'_i)$ and $s(K'_i)=-2$. Therefore, $K'_i$ is not H-slice in $\#^n \CP$ by the inequality \eqref{eq:sless}. Suppose $K_i$ were H-slice in $W=\#^n \CP$ for some $n$. (The case $n=0$ corresponds to $S^4$.)  Then Lemma~\ref{lem:slicehomeo} would show that $K'$ is H-slice in a $4$-manifold $X$ that is homotopy equivalent to $W$, and therefore homeomorphic to it by Freedman's theorem \cite{Freedman}. On the other hand, $X$ could not be diffeomorphic to $W$, because $K'_i$ is not H-slice in $W$.
\end{proof}

Note that if any of the $23$ knots were H-slice in $\#^n \CP$, they would actually be BPH-slice, in view of the following lemma.
\begin{lemma}
The $23$ knots in Figures~\ref{fig:21a} and \ref{fig:2} are all H-slice in $\#^3 \bCP$. 
\end{lemma}

\begin{proof}
Observe that for all our knots, we have $b+c+e\in \{1,2,3\}$. When $b+c+e=0$, the RBG link in Figure~\ref{fig:RBGforbigdata} has the property that its $B$ and $G$ components are split. By Lemma~\ref{lem:splitslice}, the corresponding knots $K_B$ and $K_G$  (with $b+c+e=0$) are slice. Increasing $b$, $c$ or $e$ by one corresponds to an annular cobordism in $\bCP \setminus (\intB \sqcup \intB)$ between the respective knots. Therefore, if $n:=b+c+e > 0$, the knots $K_B(a, b, c, d, e, f)$ and $K_G(a, b, c, d, e, f)$ are H-slice in $\#^n \bCP$. \end{proof} 

\begin{remark}\label{rem:sign}
We also searched in our data for pairs of knots such that one has $s > 0$ and the other $s=\sigma=0$ (which would be relevant for H-sliceness in $\#^n \bCP$ instead of $\#^n \CP$), but we found no pairs of this type. Of course, we can obtain such pairs by taking the mirrors of the $23$ knots in our table. 
\end{remark}

We analyzed the $23$ knots from Table~\ref{tab:table1} further, to make sure they pass some well-known obstructions to being slice (or BPH-slice, as the case may be).

First, we looked at algebraic obstructions coming from the Seifert matrix.
\begin{proposition}
\label{prop:algslice}
The knots $K_6$ through $K_{21}$ from Figure \ref{fig:21a} are algebraically slice.
\end{proposition}

\noindent {\em Proof.} Recall that the algebraic concordance class of a knot is given by its Seifert matrix up to S-equivalence \cite{Levinealg, Trotter}. Since the Seifert matrix can be read from the 0-surgery, any two knots with the same 0-surgery are algebraically concordant. Thus, for the knots in our table of the form $K_G$, it suffices to check algebraic sliceness for their companions $K_B$.

A genus $2$ Seifert surface for the knot $K_B(a,b,c,d,e,f)$ is shown in Figure~\ref{fig:SeifertSurface}. The associated Seifert matrix with respect to the basis $(\alpha, \beta, \gamma, \delta)$ is
\begin{equation}
\label{eq:A}
A = \begin{pmatrix} 0 & 0 & 0 & -1 \\ 1 & a+c+d+f & c+1 & f+c+d+1 \\
0 & c & b+c+e & b+c+1 \\ 0 & f+c+d & b+c+1 & b+c+d+f
\end{pmatrix}.
\end{equation}
To show that a knot $K_B(a,b,c,d,e,f)$ is algebraically slice, we will explicitly find a two-dimensional subspace $V$ on which the Seifert form $A$ restricts to the zero matrix. In other words, if we form a $4 \times 2$ matrix $S$ whose columns are the basis vectors of $V$, we should have
$$ S^T \! A \, S = 0.$$
Looking at the knots $K_1$ through $K_{21}$ in Table~\ref{tab:table1}, we see that with the exception of $K_{12}$ and $K_{18}$, all the other values $(a, b, c, d, e, f)$ satisfy $b=1, d+f=0, a+c+e=2$. In these cases, we can take
$$ V=\Span\{ (0,1,1,-1), (e-2,1,0,0)\}.$$
\begin{figure}
{
   \fontsize{9pt}{11pt}\selectfont
   \def\svgwidth{3.5in}
   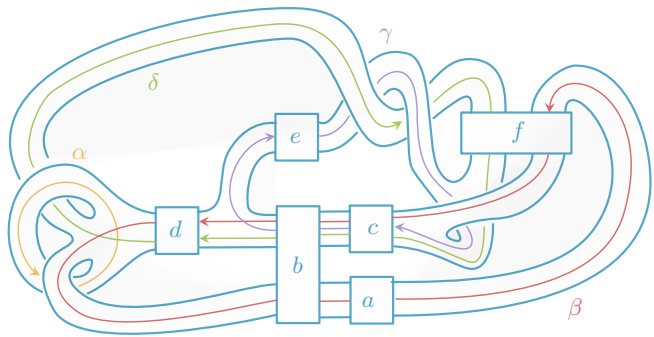
}
\caption{A Seifert surface for the knot $K_B(a,b,c,d,e,f)$.}
\label{fig:SeifertSurface}
\end{figure}

For $K_{12}$ we have $(a, b, c, d, e, f)=(1,1,2,0,-1,1)$ and we choose
$$ V=\Span\{ (2,0,1,-1), (2,-2,0,1)\}.$$

For $K_{18}$ we have $(a, b, c, d, e, f)=(2,1,-2,0,2,1)$ and we choose
\[
\pushQED{\qed} 
 V=\Span\{ (1,1,1,0), (2,1,0,1)\}.\qedhere
\popQED
\] 

Recall from Section~\ref{sec:BPH} that a necessary condition for a knot to be BPH slice is that it its Levine-Tristram signature function $\sigmaTL$ vanishes. Algebraically slice knots satisfy this, so Proposition~\ref{prop:algslice} ensures that $K_6$ through $K_{21}$ have $\sigmaTL=0$. Of course, we also have $\sigmaTL=0$ for the topologically slice examples $K_1$ through $K_5$. For the two remaining knots $K_{22}$ and $K_{23}$, the Alexander polynomial has no roots on the unit circle, and therefore $\sigmaTL$ is a constant function. Using the Seifert matrix \eqref{eq:A}, we checked that $\sigmaTL(-1) = \sigma=0$ and hence $\sigmaTL=0$. 

Second, for each of the $23$ knots, we computed the knot Floer homology using the {\em Knot Floer homology calculator} \cite{HFK}. The concordance invariants $\tau$ from \cite{OStau}, $\nu$ from \cite{RatSurg} and $\epsilon$ from \cite{Hom} vanish. As an aside, the program indicated that all the knots from our list are non-fibered, non-L-space, and have Seifert genus equal to $2$.

Third, we used the program \emph{SKnotJob} \cite{SKnotJob} to minimize the girth of the diagrams, and compute several Ramussen-type concordance invariants. Apart from the usual $s$ (which is defined from Khovanov homology over $\Q$), the program computed $s^{\F_2}$ and $s^{\F_3}$ (from Khovanov homology over $\F_2$ and $\F_3$), as well as the Lipshitz-Sarkar $s^{\operatorname{Sq}^1}$ invariant (from the first Steenrod square on Khovanov homology \cite{LS}). All the knots had girth at most $10$, and each calculation lasted only a few seconds. All the invariants turned out to be $0$. (On the other hand, computing the Lipshitz-Sarkar $s^{\operatorname{Sq}^2}_{\pm}$ invariants using a program such as \cite{SarkarSq} did not seem feasible, because our knots have at least $16$ crossings.)

Finally, from Lemma~\ref{lem:specialeven} we see that $14$ knots in Table~\ref{tab:table1} (namely, $K_1$, $K_4$, $K_5$, and $K_{13}$ through $K_{23}$) correspond to odd RBG links, because $r=a+b$ is odd. In such a case, Theorem~\ref{thm:Boyer} shows that the 0-surgery homeomorphism relating the knot and its companion does not extend to traces. Therefore, for those $14$ knots, sliceness and BPH-sliceness cannot be immediately excluded by the fact that we have obstructed their companion knot (which we knew not to be BPH-slice because $s=-2$). Moreover, modulo the caveat in footnote\footref{snap}, {\em SnapPy} indicates that the $0$-surgeries on all our $23$ knots are hyperbolic and have trivial isometry group (and hence trivial homeomorphism group, by Mostow rigidity). Hence, we expect that for the $14$ examples with $r$ odd, the trace of the knot is not even homeomorphic to that of its companion knot.

We also attempted (unsuccessfully) to show the $5$ topologically slice knots are slice. We searched for ribbon bands using Gong's program \cite{Gong}, but we could not found any simple bands that produce a strongly slice link. For one example, namely $K=K_2$, we also tried to find a slice derivative as follows: we wrote down a (genus 2) Seifert surface $F$ for $K$ and the associated Seifert matrix $A$, and then classified all dimension 2 subspaces of $H_1(F)$ on which $A$ restricts to the 0-matrix. For each such half-basis, we drew a 2-component link $L\subset S^3$ embedded on $F$ representing that basis; such a link is usually called a \emph{derivative} of $K$. It is well known (see \cite{Levinealg}) that if $K$ admits a strongly slice derivative then $K$ is slice. Unfortunately, none of the links $L$ that came out of this example were strongly slice; this was checked using Levine-Tristram signatures or covering link calculus.

Furthermore, for all $23$ knots, we tried to use Lemma~\ref{lem:ConstructBPH} to prove BPH-sliceness. However, changing any positive to a negative crossing in our diagrams resulted in knots with $\sigma=2$, which cannot be BPH-slice.

\section{Homotopy $4$-spheres from annulus twisting} \label{sec:annulus}
Annulus twisting is a construction of 0-surgery homeomorphisms which stands out for its ability to naturally produce infinitely many knots with the same 0-surgery. In view of Theorem~\ref{thm:RBG}, there are RBG links associated to annulus twisting; see Section~\ref{sec:annulusRBG} for their description. In this section we will discuss some homotopy $4$-spheres  that arise from annulus twisting, without explicit reference to the corresponding RBG links.

\subsection{Annulus twisting}

Annulus twisting was defined by Osoinach \cite{Osoinach}, and extended to other framings and to Klein bottle twisting by Abe-Jong-Luecke-Osoinach \cite{AJLO}. Since we will need to discuss it in some detail, we reproduce the proof that annulus twisting gives rise to $0$-surgery homeomorphisms.

Let $A:S^1\times I\to S^3$ be an embedding (by the standard abuse, we will conflate the embedding and its image) and let $\ell_1\cup\ell_2=\partial A$ a framed oriented link in $S^3$, where both the framing and the orientation are inherited from $A$; when $A$ is thought of an an oriented cobordism from $S^1\to S^1$, we are setting $\ell_2=\partial^+A$. Now let $\gamma$ denote a pair of pants and let $\Gamma:\gamma\to S^3$ be an embedding such that the framed oriented boundary of $\Gamma$ is $\ell_1\cup\ell_2\cup J$ for some 0-framed knot $J$. We can think of $J$ as obtained by joining a parallel copy of $\ell_1$ to a parallel copy of $\ell_2$ using a band.  See the left hand side of Figure \ref{fig:annulustwist} for examples of links $\ell_1\cup\ell_2\cup J$; there, the $K$ box represents  parallel strands that are tied in some 0-framed knot $K$, and the $k$ box represents $k$ positive full twists. See also Figure~\ref{fig:Jmk} for examples of knots $J$ that arise this way, when $K$ is the unknot.

\begin{figure}
{
   \fontsize{9pt}{11pt}\selectfont
   \def\svgwidth{6in}
   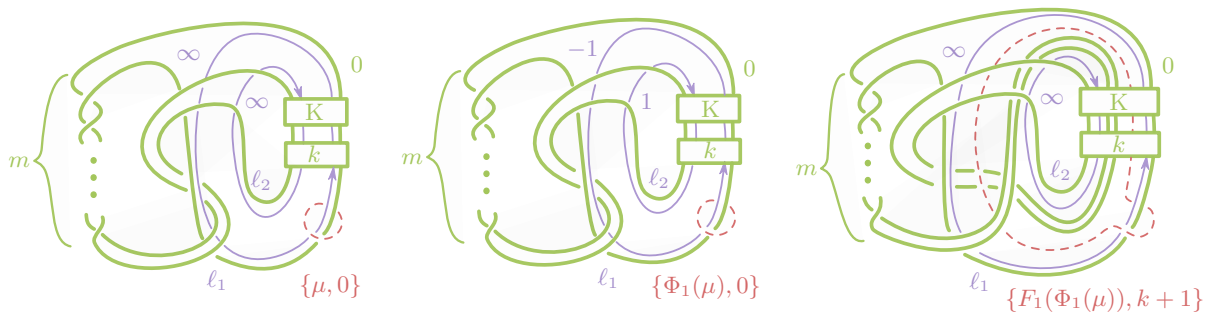
\vspace*{3mm}
}
\caption{Annulus twisting a link $\ell_1\cup\ell_2\cup J$, with $\ell_1$ and $\ell_2$ drawn in purple and $J$ in green. Here $n=1$ and we keep track of the image of a 0-framed meridian of $J$ under the annulus twist homeomorphism.} 
\label{fig:annulustwist}
\end{figure}

\begin{figure}
{
   \fontsize{10pt}{12pt}\selectfont
   \def\svgwidth{2.3in}
\begingroup%
  \makeatletter%
  \providecommand\color[2][]{%
    \errmessage{(Inkscape) Color is used for the text in Inkscape, but the package 'color.sty' is not loaded}%
    \renewcommand\color[2][]{}%
  }%
  \providecommand\transparent[1]{%
    \errmessage{(Inkscape) Transparency is used (non-zero) for the text in Inkscape, but the package 'transparent.sty' is not loaded}%
    \renewcommand\transparent[1]{}%
  }%
  \providecommand\rotatebox[2]{#2}%
  \newcommand*\fsize{\dimexpr\f@size pt\relax}%
  \newcommand*\lineheight[1]{\fontsize{\fsize}{#1\fsize}\selectfont}%
  \ifx\svgwidth\undefined%
    \setlength{\unitlength}{245.14730378bp}%
    \ifx\svgscale\undefined%
      \relax%
    \else%
      \setlength{\unitlength}{\unitlength * \real{\svgscale}}%
    \fi%
  \else%
    \setlength{\unitlength}{\svgwidth}%
  \fi%
  \global\let\svgwidth\undefined%
  \global\let\svgscale\undefined%
  \makeatother%
  \begin{picture}(1,0.6998421)%
    \lineheight{1}%
    \setlength\tabcolsep{0pt}%
    \put(0,0){\includegraphics[width=\unitlength,page=1]{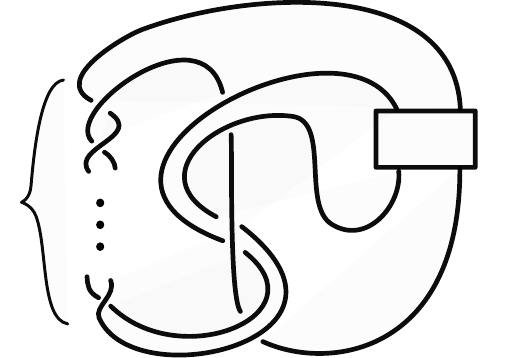}}%
    \put(-0.01538055,0.28812671){\color[rgb]{0,0,0}\makebox(0,0)[lt]{\lineheight{1.25}\smash{\begin{tabular}[t]{l}$m$\end{tabular}}}}%
    \put(0.8109403,0.41275517){\color[rgb]{0,0,0}\makebox(0,0)[lt]{\lineheight{1.25}\smash{\begin{tabular}[t]{l}$k$\end{tabular}}}}%
    \put(0,0){\includegraphics[width=\unitlength,page=2]{jkm.pdf}}%
  \end{picture}%
\endgroup%

}
\caption{A family of knots $J_m[k]$.  When $m=3-2k$ the knots are ribbon, which can be seen by performing the purple band move indicated.}
\label{fig:Jmk}
\end{figure}

\begin{theorem}[Main theorem of \cite{Osoinach}]\label{thm:annulus}
Associated to such a link $\ell_1\cup\ell_2\cup J$ there is an infinite family of knots $(J_n)_{n \in \Z}$ such that $S^3_0(J_n)\cong S^3_0(J)$.
\end{theorem}
\begin{remark}
We make no claim on the distinctness of the knots $J_n$ here, but the interested reader should consult \cite{BGL}, which gives some weak conditions on $\ell_1\cup\ell_2\cup J$ that guarantee infinitely many of the $J_n$ are pairwise distinct. 
\end{remark}

\begin{proof}
The proof follows from two claims: first we will show that $S^3_0(J)\cong S^3_{0,1/n,-1/n}(J,\ell_1,\ell_2)$. Second we will define $J_n$ and show that $S^3_{0,1/n,-1/n}(J,\ell_1,\ell_2)\cong S^3_0(J_n)$. We remark that these surgery coefficients are relative to the given framing (which often differs from the Seifert framing).

Towards the first claim, observe that in $S^3_0(J)$ the pair of pants $\Gamma$ can be capped off with the surgery disk to give an embedded annulus $\Gamma'$ with framed oriented boundary $\ell_1\cup \ell_2\subset S^3_0(J)$. We'll use $\Gamma'$ to define a homeomorphism $$\Phi_n: S^3_0(J) \to S^3_{0,1/n,-1/n}(J,\ell_1,\ell_2)$$ as follows. Consider $Y=S^3_0(J)\smallsetminus \nu(\ell_1\cup\ell_2)$ and observe that $\Gamma'$ restricts to a properly embedded annulus in $Y$. Consider $\nu(\Gamma')\cong S^1\times I\times I$, where the final $I$ factor comes from the normal direction. Define the homeomorphism $$\phi_n:S^1\times I\times I \to S^1\times I\times I, \ \ \ \phi_n(\theta,x,t)=(\theta+n(2\pi i)t,x,t).$$ Since $\phi_n|_{S^1\times I\times\partial I}$ is the identity map, we can extend $\phi_n$ to a self homeomorphism $\Phi_n$ of $Y$ by taking the identity map outside of $\nu(\Gamma')$.

We will now extend $\Phi_n$ to a homeomorphism with domain $S^3_0(J)$. We can do this by simply reattaching the excised neighborhoods $\nu(\ell_1\cup\ell_2)$, but in the target space we must take care to reattach along the gluing map modified by $\Phi_n$. We see that $\Phi_n|_{\partial Y}$ takes the meridional curve in $\partial(\nu(\ell_1))$ (resp. $\partial(\nu(\ell_2))$) to the $\smash{\frac{1}{n}}$ curve on $\partial(\nu(\ell_1))$ (resp. $\smash{\frac{-1}{n}}$ curve on $\partial(\nu(\ell_2))$). Thus the first claim follows. 

To prove the second claim, observe that by hypothesis $\ell_1\cup \ell_2$ cobound an annulus $A$ in $S^3$. We can twist $S^3$ along $A$ as in the proof of the first claim to produce a homeomorphism 
$$f_n:S^3_{1/n, -1/n}(\ell_1\cup\ell_2)\to S^3.$$ Since the knot $\iota_n(J)\subset S^3_{1/n, -1/n}(\ell_1\cup\ell_2)$ (where the embedding $\iota_n$ is given by any diagram of $J\cup\ell_1\cup\ell_2$) may intersect $A$, the image $f_n(\iota_n(J))$ is some a priori new knot $J_n\subset S^3$. Surgering on $\iota_n(J)$ and its image under $f_n$, we get a homeomorphism $$F_n:S^3_{0,1/n,-1/n}(J,\ell_1,\ell_2)\to S^3_{r_n}(J_n).$$ In general one would now need to inspect $f_n$ to compute $r_n$, but since $H_1(S^3_{0,1/n,-1/n}(J,\ell_1,\ell_2))\cong \Z$ we can conclude $r_n=0$.
\end{proof}

\subsection{Examples}
Consider the family of knots $J_m[k]$ pictured in Figure~\ref{fig:Jmk}. (When $k=1$, this recovers the family $J_m$ from \cite[Figure 4]{AJOT}.) We ask for $m$ to be odd but we allow it to be negative---in which case we have $|m|$ left-hand half twists in the picture. Similarly, we allow the number $k$ of full twists on the right to be an arbitrary integer. In all figures we are taking the annulus $A\cong S^1\times I$ to be oriented such that the $S^1$ factor is counterclockwise in the figure and the $I$ factor points radially inward. For this orientation of $A$, $\ell_2$ is the oriented curve marked in Figure \ref{fig:annulustwist} and in the standard orientation of $S^3$, the normal direction $t$ to $A$ points into the page.

Annulus twisting applied to each $J=J_m[k]$ (with $\ell_1$ and $\ell_2$ as in Figure~\ref{fig:rbgforannulus}) produces an infinite family of knots, $J_m[k]_n$, with the same $0$-surgeries as we vary $n$. Of these, we show $J_m[k]_{1}$ and $J_m[k]_{-1}$ in Figure~\ref{fig:Jmktw}.

\begin{figure}
{
   \fontsize{10pt}{12pt}\selectfont
   \def\svgwidth{5in}
\begingroup%
  \makeatletter%
  \providecommand\color[2][]{%
    \errmessage{(Inkscape) Color is used for the text in Inkscape, but the package 'color.sty' is not loaded}%
    \renewcommand\color[2][]{}%
  }%
  \providecommand\transparent[1]{%
    \errmessage{(Inkscape) Transparency is used (non-zero) for the text in Inkscape, but the package 'transparent.sty' is not loaded}%
    \renewcommand\transparent[1]{}%
  }%
  \providecommand\rotatebox[2]{#2}%
  \newcommand*\fsize{\dimexpr\f@size pt\relax}%
  \newcommand*\lineheight[1]{\fontsize{\fsize}{#1\fsize}\selectfont}%
  \ifx\svgwidth\undefined%
    \setlength{\unitlength}{713.80764771bp}%
    \ifx\svgscale\undefined%
      \relax%
    \else%
      \setlength{\unitlength}{\unitlength * \real{\svgscale}}%
    \fi%
  \else%
    \setlength{\unitlength}{\svgwidth}%
  \fi%
  \global\let\svgwidth\undefined%
  \global\let\svgscale\undefined%
  \makeatother%
  \begin{picture}(1,0.39276549)%
    \lineheight{1}%
    \setlength\tabcolsep{0pt}%
    \put(0,0){\includegraphics[width=\unitlength,page=1]{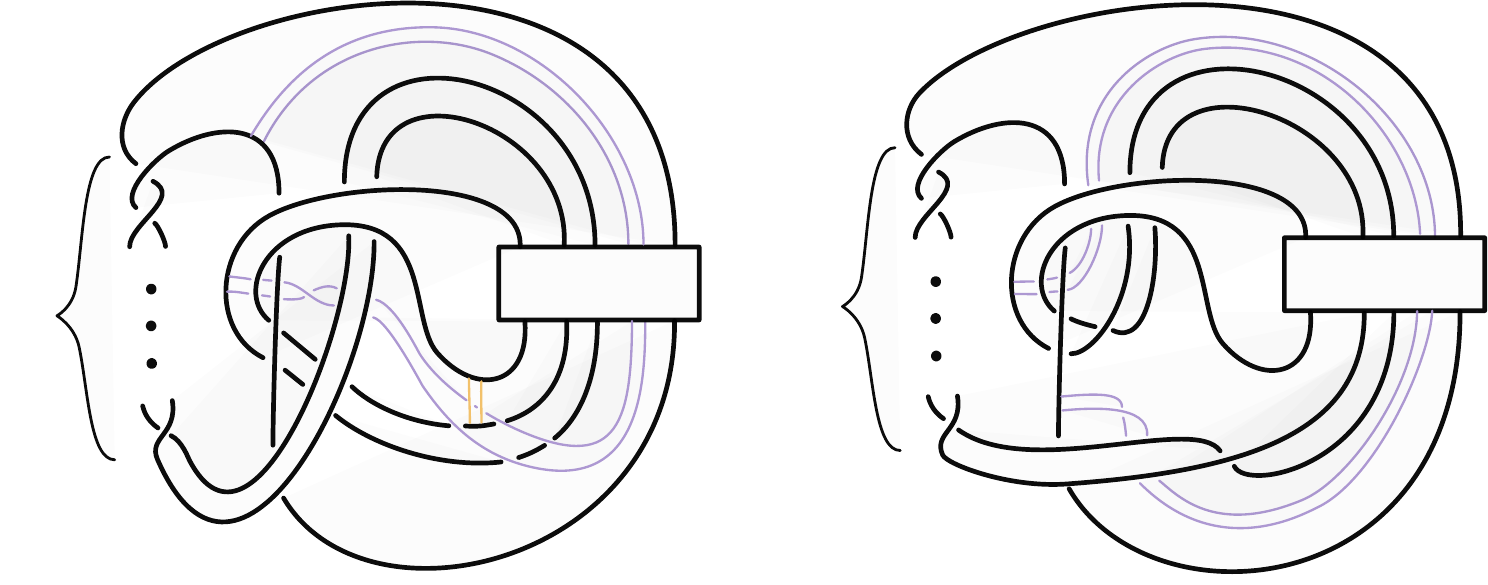}}%
    \put(0.00102197,0.17238895){\color[rgb]{0,0,0}\makebox(0,0)[lt]{\lineheight{1.25}\smash{\begin{tabular}[t]{l}$m$\end{tabular}}}}%
    \put(0.3908684,0.19465829){\color[rgb]{0,0,0}\makebox(0,0)[lt]{\lineheight{1.25}\smash{\begin{tabular}[t]{l}$k$\end{tabular}}}}%
    \put(0.5322473,0.17859572){\color[rgb]{0,0,0}\makebox(0,0)[lt]{\lineheight{1.25}\smash{\begin{tabular}[t]{l}$m$\end{tabular}}}}%
    \put(0.92209366,0.20086507){\color[rgb]{0,0,0}\makebox(0,0)[lt]{\lineheight{1.25}\smash{\begin{tabular}[t]{l}$k$\end{tabular}}}}%
    \put(0,0){\includegraphics[width=\unitlength,page=2]{jmktwist.pdf}}%
  \end{picture}%
\endgroup%

}
\caption{The knots $J_m[k]_{1}$ and $J_m[k]_{-1}$  obtained from the knots in Figure~\ref{fig:Jmk} by annulus twisting. When $m=3-2k$ the knots are ribbon, which can be seen by performing the purple and yellow band moves indicated.
}
\label{fig:Jmktw}
\end{figure}

As in Section~\ref{sec:computer}, in the hope of producing an exotic $S^4$ or $\#^n \CP$, we looked for examples where one of the $J_m[k]$ and $J_m[k]_{\pm 1}$ was slice (or H-slice in $\#^n \CP$), and another was not. For this, we would like that the knots satisfy $\sigma=0$, and the values of the $s$ invariant are different. 

Observe that if $k=0$, then $J_m[k]$ and $J_m[k]_{\pm 1}$ are unknotted. Furthermore, when $k=\pm 1$, the knots $J_m[k]$ and $J_m[k]_{\pm 1}$ share the same trace; this was originally proven in \cite[Theorem 2.8]{AJOT}; the readers can see this for themselves as a consequence of Theorem \ref{thm:propU} and Remark  \ref{rem:annulusparity}. Thus, in view of Lemma~\ref{lem:TEL}, the only hope for interesting examples comes from the cases $|k| \geq 2$. As noted in Remark \ref{rem:annulusparity}, for the annulus knots in Figure \ref{fig:Jmk}, the homeomorphism coming from twisting $J_m[k]$ once in either direction is odd if and only if $k$ is even, so the setting $k=2$ should generically yield knots which do not have homeomorphic traces. 

Unfortunately, we found no examples with $\sigma=0$ and different $s$ for small values of $k$ and $m$. Table~\ref{tab:table2} displays the values of the pair $(\sigma, s)$ for $J_m[k]$. For all the knots in the table except the one marked in blue, namely $J_1[-2]=11n79$, the value of $s$ stays the same when we do annulus twisting in either direction. For $J_1[-2]$, we have $s=0$ but both annulus twists, $J_1[-2]_{1}$ and $J_1[-2]_{-1}$, have $s=2$ instead of $s=0$. (However, the signature is nonzero in this example.)
 
\renewcommand{\arraystretch}{1.5}
\begin{table}
\begin{center}
\begin{tabular}{|r|| c | c | c | c | c | c | }
\hline
\backslashbox{$m$}{$k$} & -2 & -1 & 0 & 1 & 2 & 3 \\
\hline \hline
-5 & (2, 2) & (2, 2) &  \textcolor{red}{(0, 0)} & (0, 0) & (0, 0) & (0, 0) \\ \hline
-3 & (2,  2) & (2, 2) &  \textcolor{red}{(0, 0)} & (0, 0) & (0, 0) & \textcolor{red}{(0, 0)}\\ \hline
-1 & (2, 2) & (2, 2) &  \textcolor{red}{(0, 0)} & (0, 0) & \textcolor{red}{(0, 0)} & (0, 0)\\ \hline
1 & \textcolor{blue}{(2, 0)} & (2, 0) &  \textcolor{red}{(0, 0)} & \textcolor{red}{(0, 0)} & (0, 0) & (0, 0)\\ \hline
3 & (2, 0) & (2, 0) &  \textcolor{red}{(0, 0)} & (0, -2) & (0, -2) & (0, -2)\\ \hline
5 & (0, 0) &  \textcolor{red}{(0, 0)} &  \textcolor{red}{(0, 0)} & (0, -2) & (0, -2) & (0, -2)\\ \hline
7 &  \textcolor{red}{(0, 0)} & (0, 0) &  \textcolor{red}{(0, 0)} & (-2, -2) & (-2, -2) & (-2, -2)\\ \hline
\end{tabular}
\end{center}
 \caption{Values of $(\sigma, s)$ for knots of the form $J_m[k]$.} 
     \label{tab:table2}
\end{table}

A few remarks are in order about the knots in Table~\ref{tab:table2}. We marked in red those knots where we know that the knot and its annulus twists are slice. For the knots $J_{3-2k}[k]$ and their annulus twists, Figures \ref{fig:Jmk} and \ref{fig:Jmktw} show the existence of ribbon bands relating them to the unlink.  Note that many of these (untwisted) examples can be recognized from knot tables. Indeed, we have:
$$ J_{5}[-1]=13n469, \ \ \ J_3[0]=\text{unknot}, \ \ \ J_1[1]=8_{20}, \ \ \ J_{-1}[2] = 8_8, \ \ \ J_{-3}[3]=13n1209.$$

Observe also that all knots between the vertical unknot line and the diagonal slice line (that is, $J_m[k]$ with $k\geq 0, m \leq 3-2k$ or $k \leq 0, m \geq 3-2k$) are BPH-slice, and so are their annulus twists. This is because we can get from them to a slice knot by changing crossings in either direction; cf. Lemma~\ref{lem:ConstructBPH}. 

\begin{remark}
The knots $J_{-1}[2]=8_8$ and $J_{-1}[2]_1$ appear in the list considered in Section~\ref{sec:computer}, as $K_G(2,1,0,0,1,-1)$ and $K_B(2,1,0,0,1,-1)$. Furthermore, the knots $J_{-1}[2]=8_8$ and $J_{-1}[2]_{-1}$ appear in the list as $K_G(-1,0,-2,1,0,1)$ and $K_B(-1,0,-2,1,0,1)$; and also as $K_G(1,0,0,0,2,-1)$ and $K_B(1,0,0,0,2,-1)$.
\end{remark} 

\subsection{Homotopy 4-spheres} 
Suppose that $K$ is a slice knot, and that $K$ and $K'$ admit a $0$-surgery homeomorphism $\phi: S^3_0(K) \to S^3_0(K')$. As in the proof of Lemma \ref{lem:slicehomeo}, we can construct the homotopy $4$-sphere
\begin{equation}
\label{eq:XV}
X=X(-K') \cup_{\phi} V,
\end{equation}
where $V$ is any slice disk exterior for $K$. Even when we know that both $K$ and $K'$ are slice, it is not clear that $X$ is a standard $4$-sphere. We present below some examples of such homotopy $4$-spheres, for which we could not verify that $X$ is $S^4$. Many more homotopy 4-spheres can be constructed via the techniques of this paper; for example by taking $K$ to be a slice knot which satisfies the hypothesis of Lemma \ref{lem:untwisting}. We demonstrate via our examples how one can draw explicit handle decompositions of homotopy spheres $X$ built as in \eqref{eq:XV}.

\begin{example}\label{ex:htpy88}
We first give examples coming from the knot $K=J_{-1}[2]=8_8$ and $K'$ any of its annulus twists. The knot $8_8$ is slice; we have exhibited a ribbon band, and hence ribbon disk $D$, in Figure \ref{fig:Jmk}. To draw a handle decomposition of $V=B^4\setminus\nu(D)$ we use the rising water principle; see \cite[Chapter 6.2]{GS}. This decomposition is given by the black and purple curves in Figure \ref{fig:htpy88}. 

\begin{figure}
{
   \fontsize{10pt}{12pt}\selectfont
   \def\svgwidth{5in}
   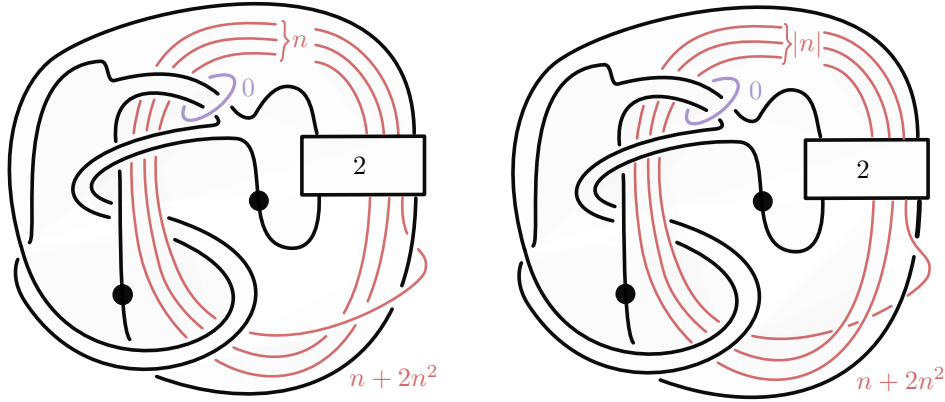
}
\caption{Homotopy spheres associated to annulus twisting $8_8$. Left $n>0$, right $n<0$.}
\label{fig:htpy88}
\end{figure}

As argued in the proof of Lemma \ref{lem:slicehomeo}, $X$ has a handle diagram obtained from that of $V$ by adding an additional 2-handle along $\phi(\mu_{K'}, 0)$, followed by a 4-handle. In Figure \ref{fig:annulustwist} we give an example identification of the framed curve $\phi(\mu_{K'})$ in $S^3_0(K)$, where $\phi$ is the homeomorphism given by annulus twisting $J_m(k)$ once and $\mu(K')$ is 0-framed. In Figure \ref{fig:htpy88} we exhibit in red the framed curve $\phi_n(\mu_{K'})\subset \partial V$ where $\phi_n$ is the homeomorphism given by annulus twisting $8_8$ $n$ times, for $n\neq 0$. These images are computed by inspection of the annulus twist homeomorphism, which we gave explicitly in the proof of Theorem \ref{thm:annulus}. 
\end{example}

\begin{remark}
Kyle Hayden informed us that for $n=1$, the homotopy $4$-sphere shown in Figure~\ref{fig:htpy88} is standard. 
\end{remark}

\begin{remark}\label{rem:annulusparity}
By keeping track of the image of a meridian as we have done in Figure \ref{fig:annulustwist}, it is straightforward to check that annulus twisting once in either direction is even if and only if $k$ is odd, and that annulus twisting has property $U$ when $K=U$ and $k=\pm 1$. 
\end{remark}

\section{Connections with other constructions}\label{sec:otherconstr}
In this section we draw RBG links for some of the 
0-surgery homeomorphisms in the literature. Theoretically, there is no work to this; the procedure to draw an RBG link for a fixed homeomorphism is given in the proof of Theorem \ref{thm:RBG}. However, we find the pictures to be a helpful reference, thus we include them here. We have also made some effort to give particularly elegant RBG links where possible. 

Some relationships between annulus twisting, the methods in \cite{Pic1}, and dualizable patterns have already been illustrated in the literature; see \cite{MP17, Pic1, Tagami20}. 

\subsection{Preliminaries}
Recall that Dehn surgery diagrams can be modified by ``handle sliding'' without changing the homeomorphism type of the manifold described. We observe now that certain handle slides of RBG links are RBG preserving; this proposition will allow us to draw particularly nice RBG links for some constructions. 

\begin{proposition}\label{prop:RBGslides}
Let $L$ be an RBG link and $L'$ be a framed link obtained from $L$ by some number of slides of $B$ over $R$. Then $L'$ (with suitable choices of homeomorphisms $\psi_B'$ and $\psi_G'$) is an RBG link,   and the pair of knots $K_B'$ and $K_G'$ associated to $L'$ are pairwise isotopic to the pair of knots $K_B$ and $K_G$ associated to $L$.
\end{proposition} 

\begin{remark}
We also note that by symmetry of the RBG construction, the proposition also holds with the roles of $B$ and $G$ reversed.  
\end{remark}

\begin{proof}
Recall that a handle slide is just an isotopy inside an ambient manifold. Fix a choice of handle slide isotopy of $B$ in $S^3_{r}(R)$ which, when viewed diagrammatically,  changes $B$ into $B'$.  First we check that $L'$ is an RBG link. Since $R$ and $G$ are unchanged, the framed link $G\cup R$ still surgers to $S^3$, and we can take the same $\psi_B$ as we used in $L$.  Since the framed link $B\cup R$  surgered to $S^3$ before the slides, the framed link $B'\cup R$ surgers to $S^3$. From our choice of the isotopy which slides $B$ over $R$ we extract a well-defined homeomorphism $h:S^3_{r,b}(R,B)\to S^3_{r,b'}(R,B')$ which is just the final map of the isotopy.  We then take $\psi'_G$ to be $\psi_G\circ h^{-1}$.  Finally observe that handle slides preserve the homeomorphism type, hence homology type, of the resulting manifold. 

To check that $K_B\cong K_B'$ we will exhibit a homeomorphism from $S^3_0(K_B)$ to $S^3_0(K_B')$ which carries a 0-framed meridian to a 0-framed meridian. By excising these meridians, we observe that there is a homeomorphism from $S^3\smallsetminus\nu(K_B)$ to $S^3\smallsetminus\nu(K_B')$ taking meridians to meridians. It follows that there is a homeomorphism of pairs $(S^3,K_B)\to (S^3,K_B')$, hence that the knots are isotopic. We will show that  $K_G\cong K_G'$  in the same manner. 

To define our homeomorphism $S^3_0(K_B)\to S^3_0(K_B')$ recall that (as in the proof of theorem \ref{thm:RBG}) the RBG construction comes induces homeomorphisms $\psi_B^*:S^3_{r,b,g}(L)\to S^3_0(K_B)$ and $\psi_{B'}^*:S^3_{r,b',g}(L')\to S^3_0(K_B')$. Let $f:S^3_{r,b,g}(L)\to S^3_{r,b',g}(L')$ be the homeomorphism induced by the handle slide, and consider $\psi_{B'}^*\circ f\circ (\psi_B^*)^{-1}:S^3_0(K_B)\to S^3_0(K_B')$. Now we'll watch a 0-framed meridian of $K_B$ under each leg of the map: under $(\psi_B^*)^{-1}$ it goes to a 0-framed meridian of $B$, under $f$ to a 0-framed meridian of $B'$, and under $\psi_{B'}^*$ to a 0-framed meridian of $K_B'$. The construction and meridian-watching for $G$ is similar and left to the reader.  
\end{proof}


\subsection{Annulus twisting}
\label{sec:annulusRBG}

To draw an RBG link for annulus twisting we follow the process outlined in the proof of Theorem \ref{thm:RBG}. In Figure \ref{fig:rbgforannulus} we have carried a (red) 0-framed meridian of $K$  through the homeomorphism to $S^3_0(J_1)$. Here the framing label on red is in terms of the usual convention, i.e. with respect to the diagramatic Seifert framing. Thus we obtain an RBG link for annulus twisting from the left hand frame of Figure \ref{fig:rbgforannulus} by adding a blue 0-framed meridian to the red curve.

\begin{figure}
{
   \fontsize{9pt}{11pt}\selectfont
   \def\svgwidth{7in}
   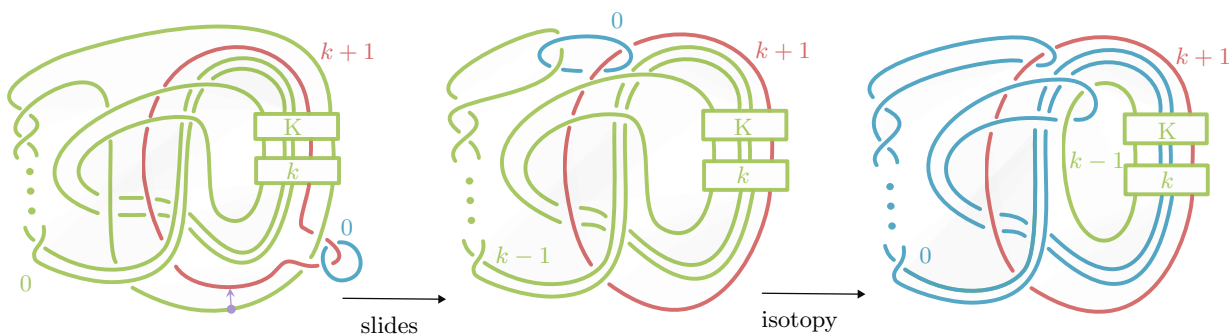
   \vspace*{-5mm}
}
\caption{Simplifying an RBG link for annulus twisting.}
\label{fig:rbgforannulus}
\end{figure}

By Proposition \ref{prop:RBGslides}, we can modify this RBG by slides of $B$ or $G$ over $R$ as we like.  We exhibit a particularly simple RBG link by considering the slide and isotopy shown in Figure \ref{fig:rbgforannulus}. We remark that if $K$ is non-trivial or $k\neq\pm 1$, then this RBG link is not special. 

\subsection{Yasui's construction}\label{sec:Yasui}
In \cite{Yasui} Yasui gives a construction of knots with homeomorphic 0-surgeries which he uses to disprove the Akbulut-Kirby conjecture. While we enthusiastically recommend his paper, in fact we will model our diagrams off of the (reproduced) proof of his construction given in Section 2.1 of \cite{HMP19}.

We restrict to the $n=0$ setting and consider a (red) 0-framed meridian of the green curve in Figure 3(a) of \cite{HMP19}. We then inspect the image of that red meridian under the handle calculus in Figure 3 of  \cite{HMP19}. Part (h) of that figure, together with image of the red curve, and a green 0-framed meridian of that red image, is the first frame of our Figure \ref{fig:rbgforyasui}. (We remark that the calculus described in Figure 3 of \cite{HMP19}  is local; the knot $K\subset \partial Z$ used in their Proposition 2.2 can go wherever it wants outside of the region shown in their Figure 3. Our Figure \ref{fig:rbgforyasui} is global and our  $Z\cong B^4$; thus we depict that $K$ may be knotted with the region marked $K$ in Figure \ref{fig:rbgforyasui}.) Appealing to Proposition \ref{prop:RBGslides}, we can tidy up this RBG link via the slides and isotopy marked in Figure \ref{fig:rbgforyasui} to obtain the rather nice simple RBG link in the right frame of Figure \ref{fig:rbgforyasui}.
\begin{figure}
{
   \fontsize{9pt}{11pt}\selectfont
   \def\svgwidth{3.5in}
   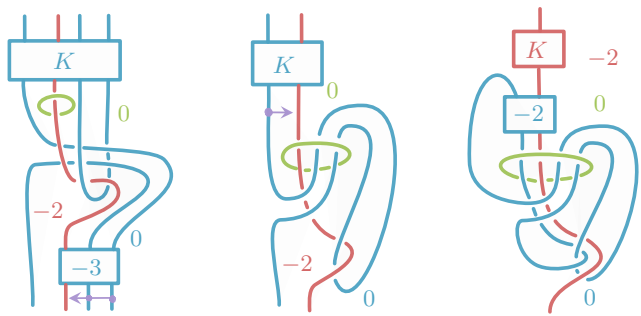
}
\caption{Braid closures of these tangles give RBG links for Yasui's homeomorphism.}\label{fig:rbgforyasui}
\end{figure}

\subsection{Dualizable patterns}
The dualizable patterns technique was developed and utilized in \cite{Akb2Dhom, Lickorish2, Brakes, GM, BakerMotegi, MP17}. In \cite{Pic1} an early version of the RBG construction was developed; in \cite{Pic1} an RBG link is required to have $r=0$, $B \cup R = B \cup\mu_B$, $G \cup R = G \cup \mu_G$ and $lk(B, G) =0$. (We remark that the RBG links studied there have property $U$, but are not necessarily special.) In the appendix to \cite{Pic1}, it is proven that any homeomorphism constructed by the dualizable patterns technique may be presented by an RBG link of the type studied in \cite{Pic1} and the converse: any RBG link of that type gives rise to a dualizable pattern. Instead of including a reproof here, which would require recalling the dualizable patterns construction, we refer the reader to the appendix of that paper.

\bibliographystyle{custom}
\bibliography{biblio}
\addresseshere

\newpage
\captionsetup[subfigure]{font=normalsize, labelformat=empty}
\begin{figure}
\centering
\subfloat[$K_6$]{\includegraphics[scale=.08]{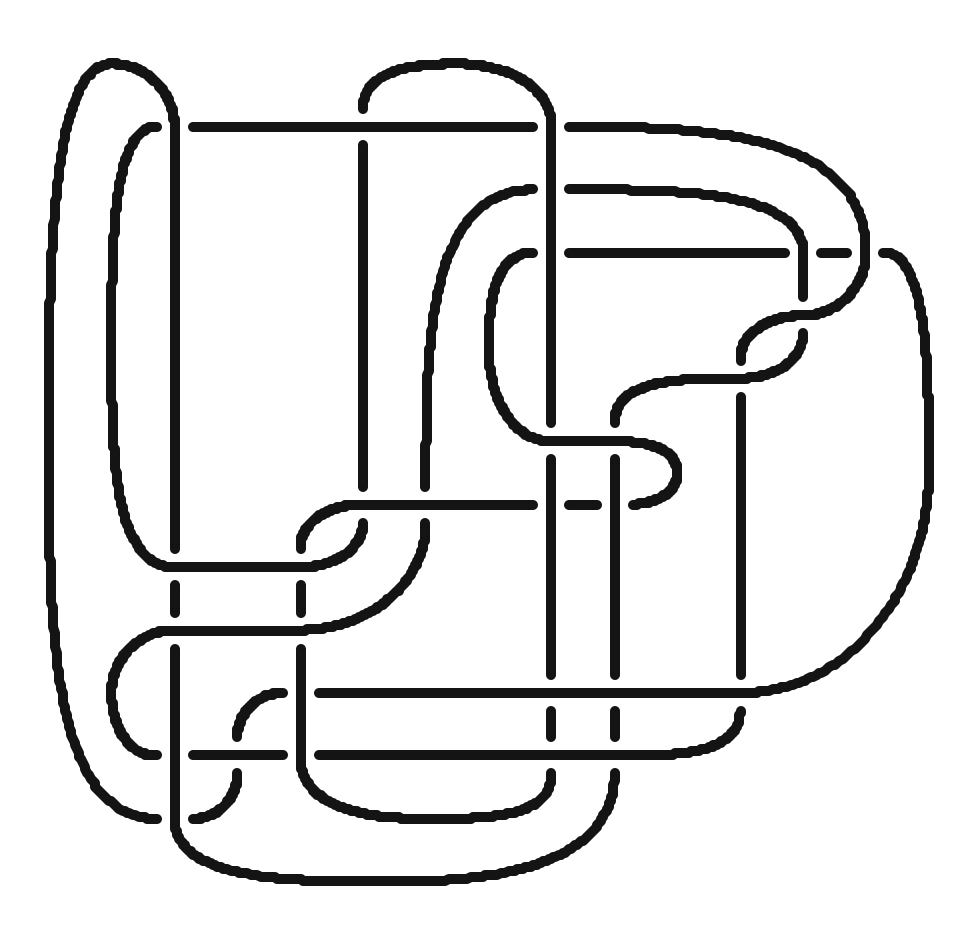}}\quad\quad\quad\quad\quad
\subfloat[$K_7$]{\includegraphics[scale=.08]{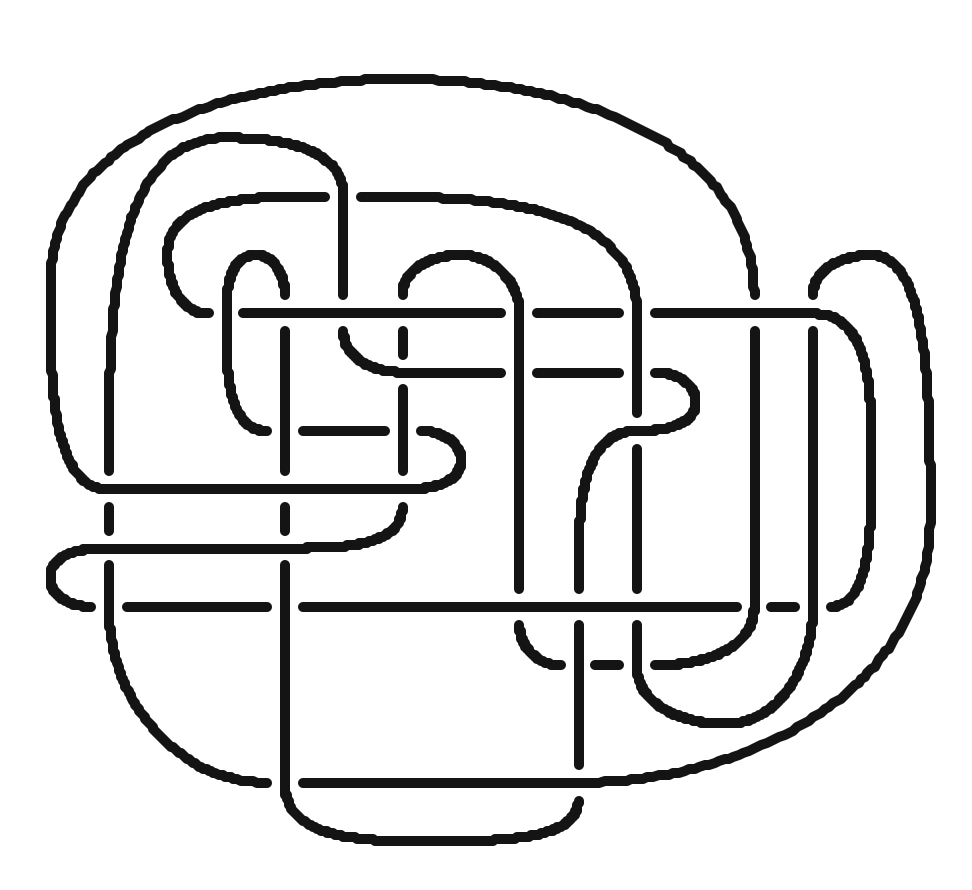}}\quad\quad\quad\quad\quad
\subfloat[$K_8$]{\includegraphics[scale=.08]{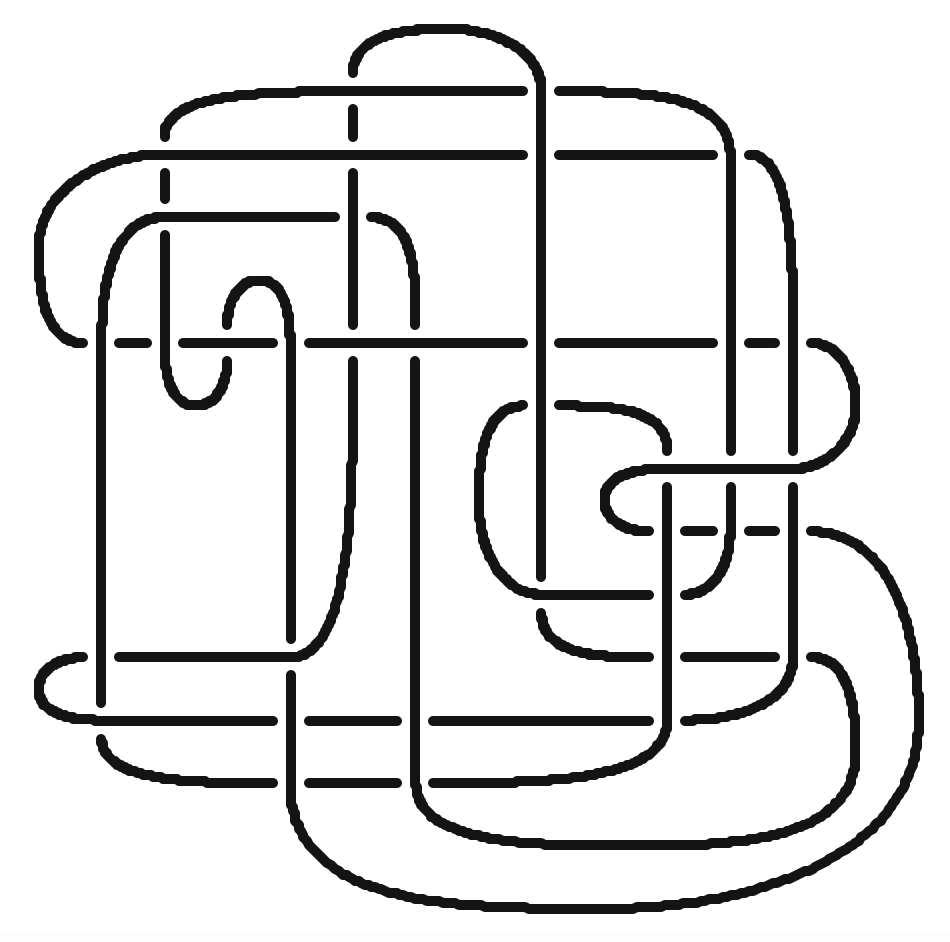}}\\
\subfloat[$K_9$]{\includegraphics[scale=.08]{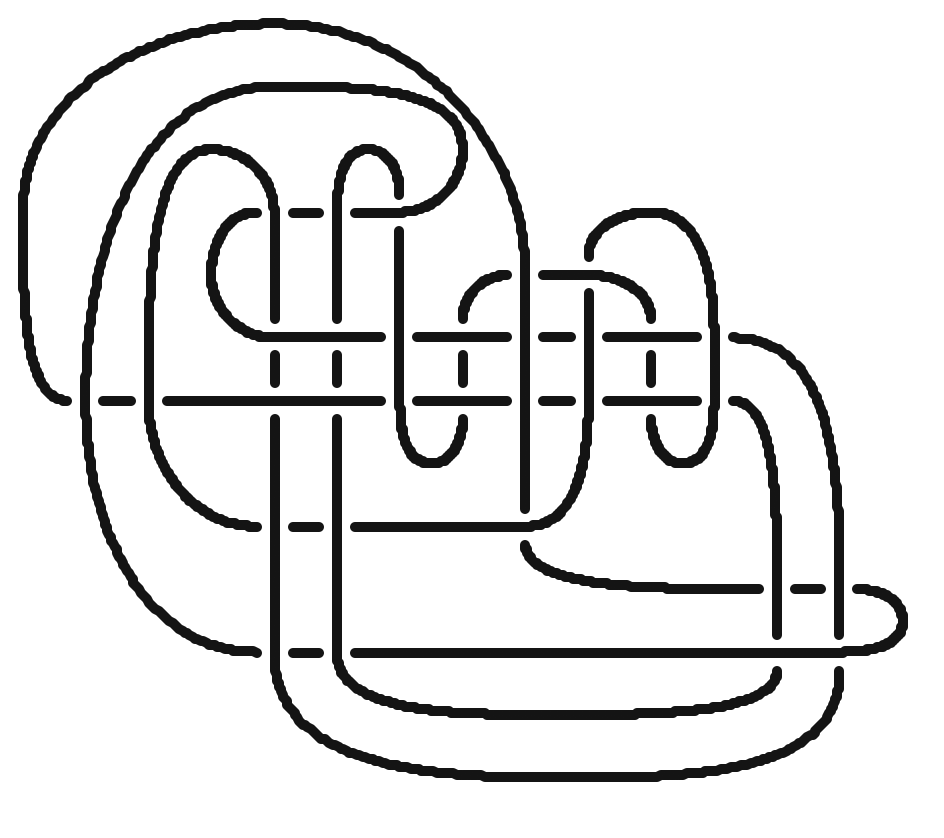}}\quad\quad\quad\quad\quad
\subfloat[$K_{10}$]{\includegraphics[scale=.08]{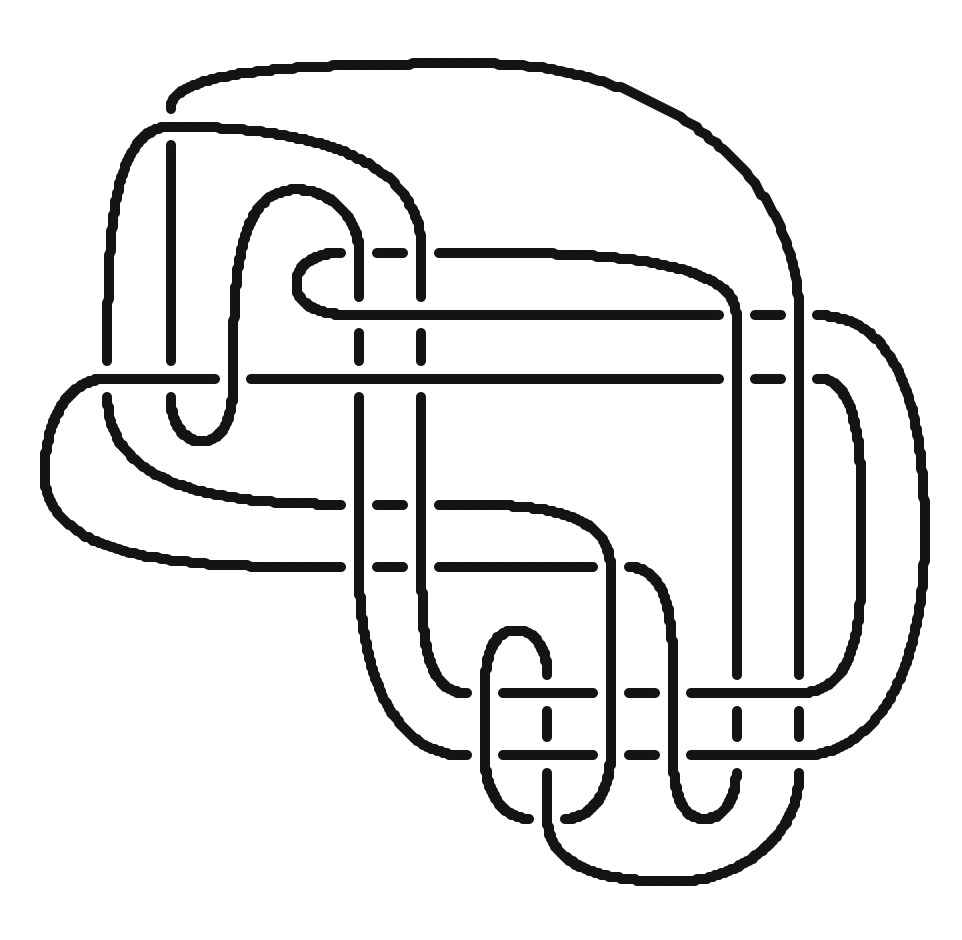}}\quad\quad\quad\quad\quad
\subfloat[$K_{11}$]{\includegraphics[scale=.08]{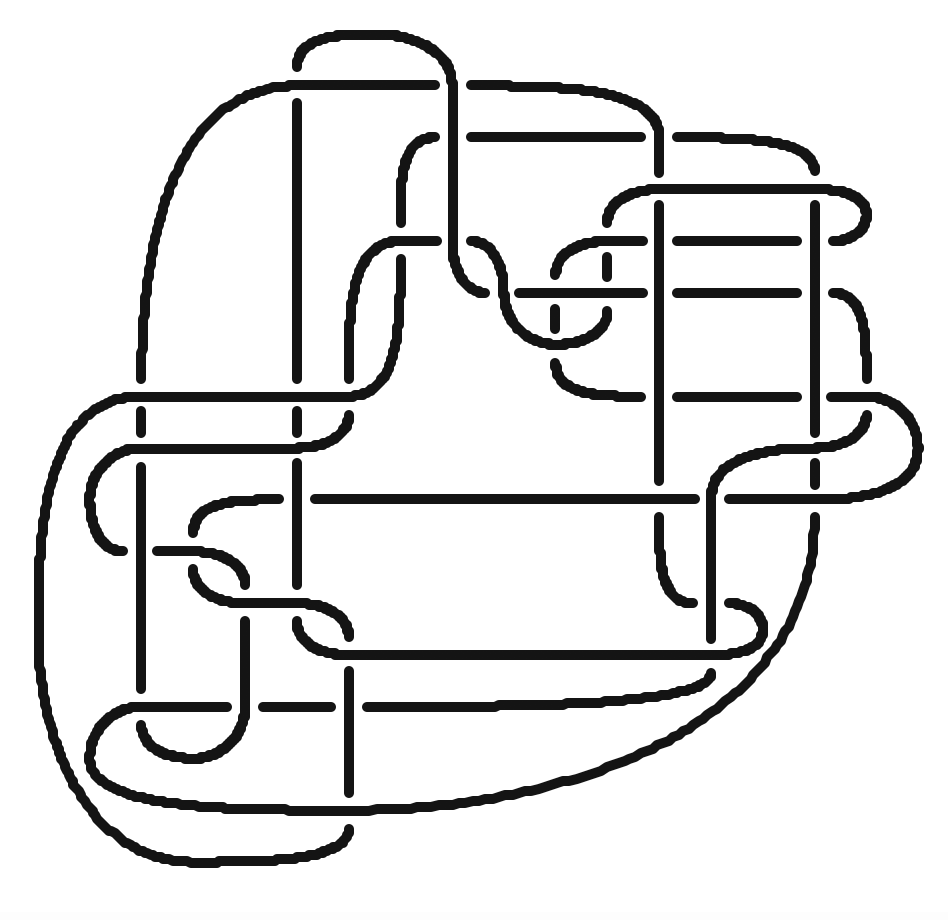}}\\
\subfloat[$K_{12}$]{\includegraphics[scale=.08]{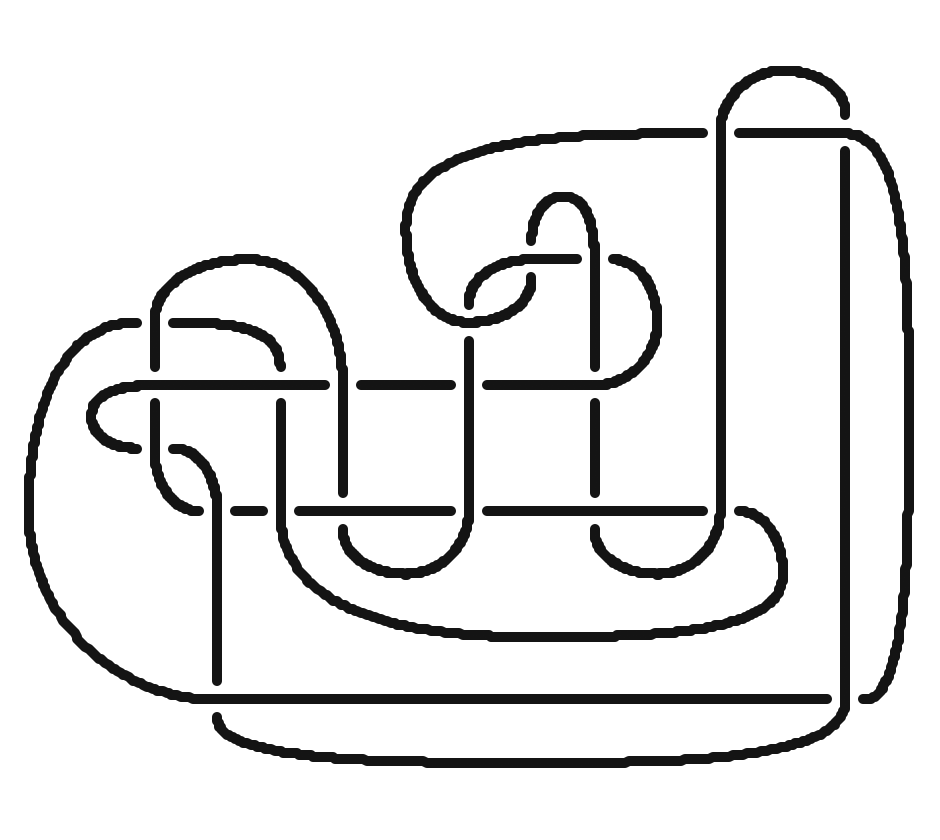}}\quad\quad\quad\quad\quad
\subfloat[$K_{13}$]{\includegraphics[scale=.08]{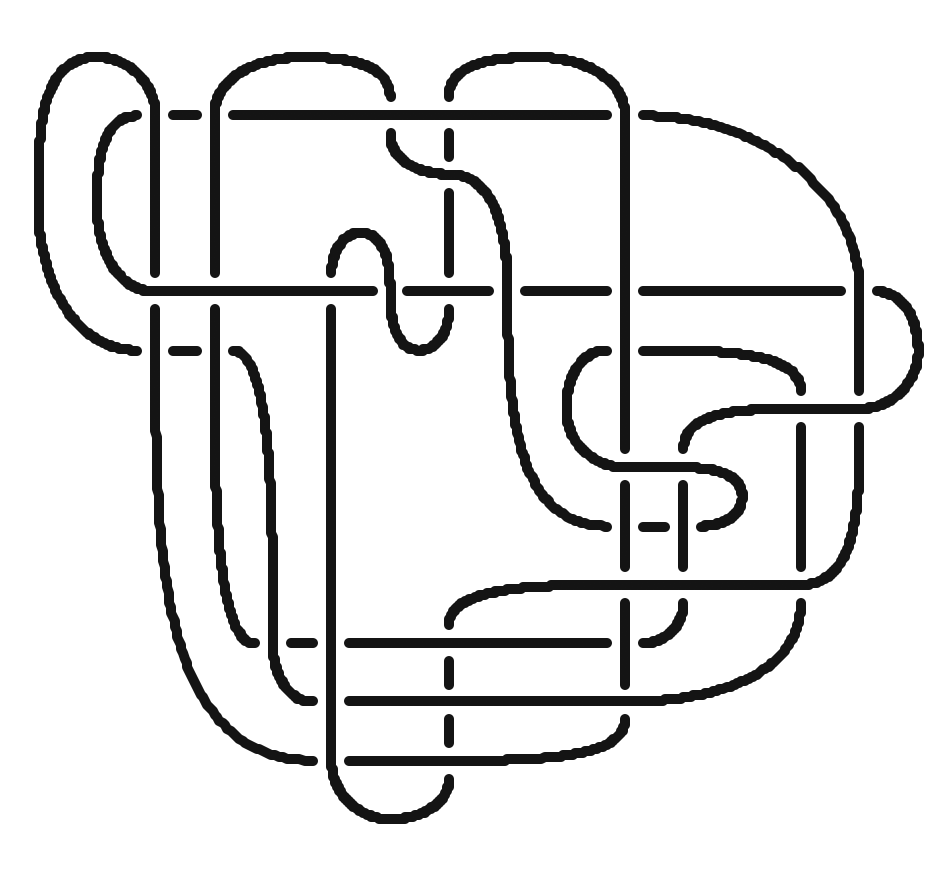}}\quad\quad\quad\quad\quad
\subfloat[$K_{14}$]{\includegraphics[scale=.08]{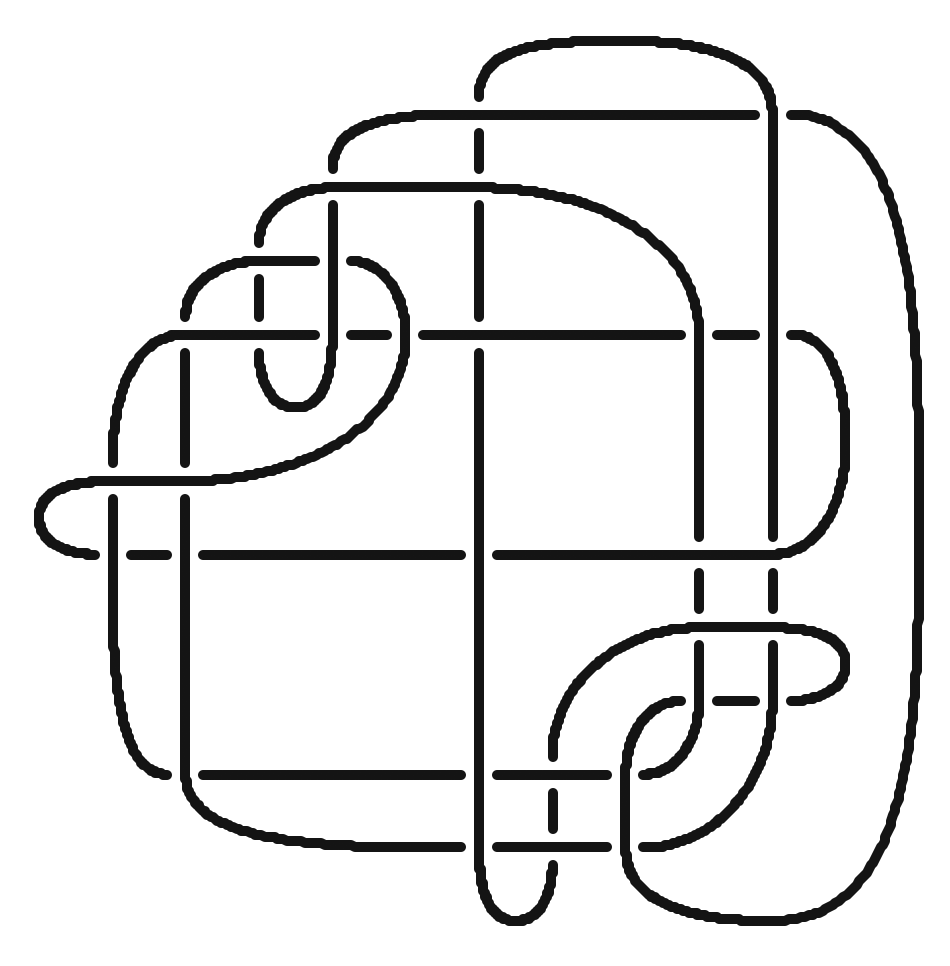}}\\
\subfloat[$K_{15}$]{\includegraphics[scale=.08]{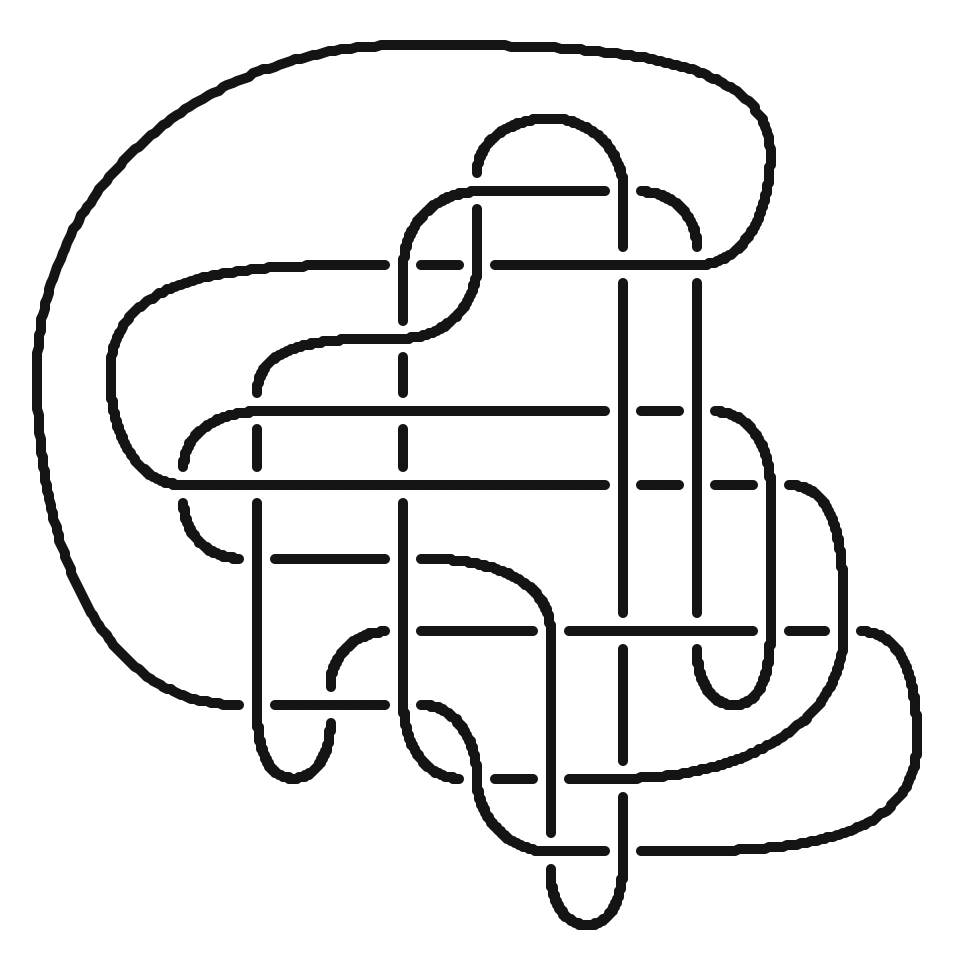}}\quad\quad\quad\quad\quad
\subfloat[$K_{16}$]{\includegraphics[scale=.08]{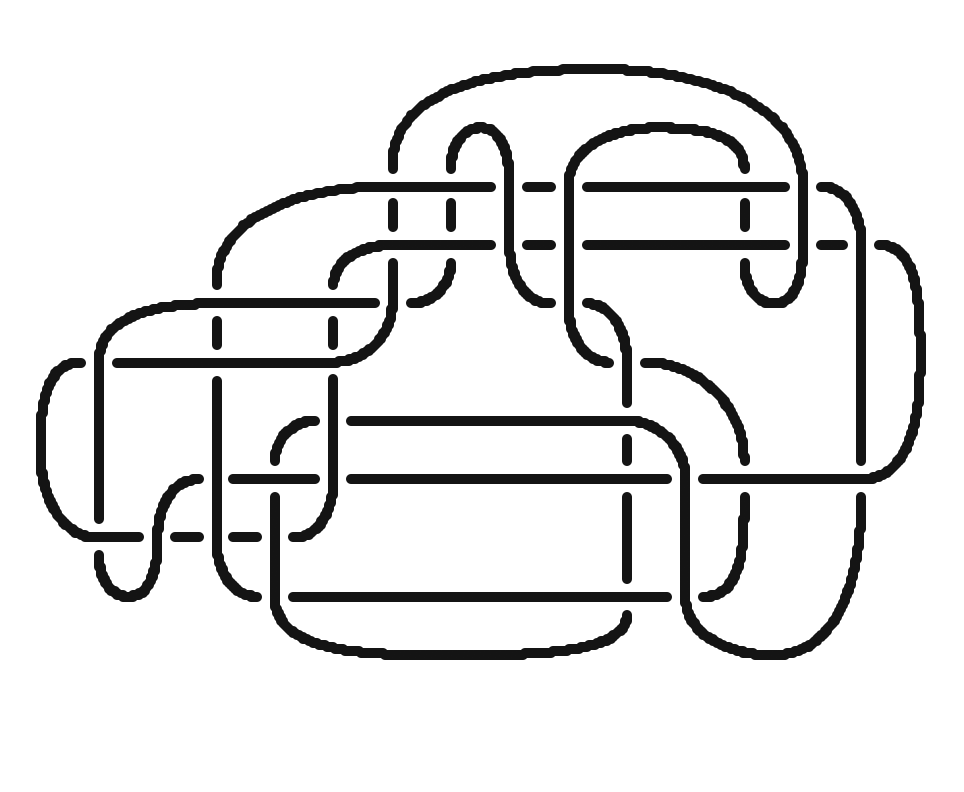}}\quad\quad\quad\quad\quad
\subfloat[$K_{17}$]{\includegraphics[scale=.08]{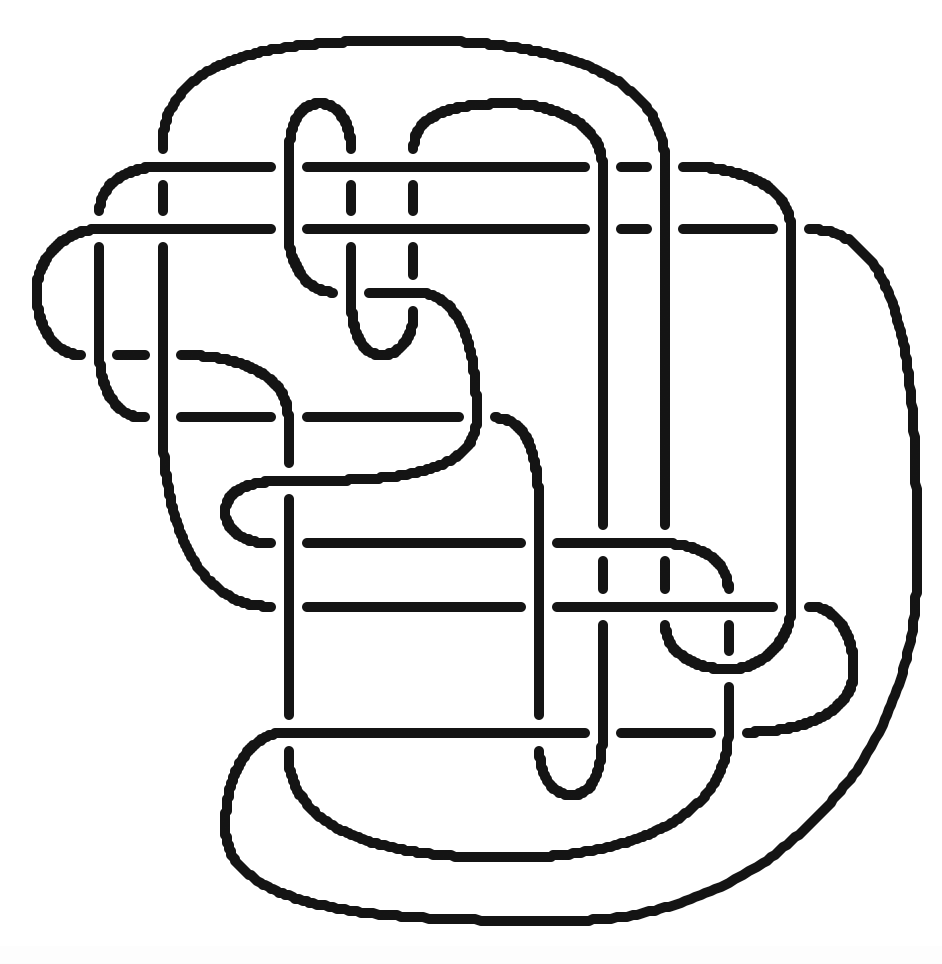}}\\
\subfloat[$K_{18}$]{\includegraphics[scale=.08]{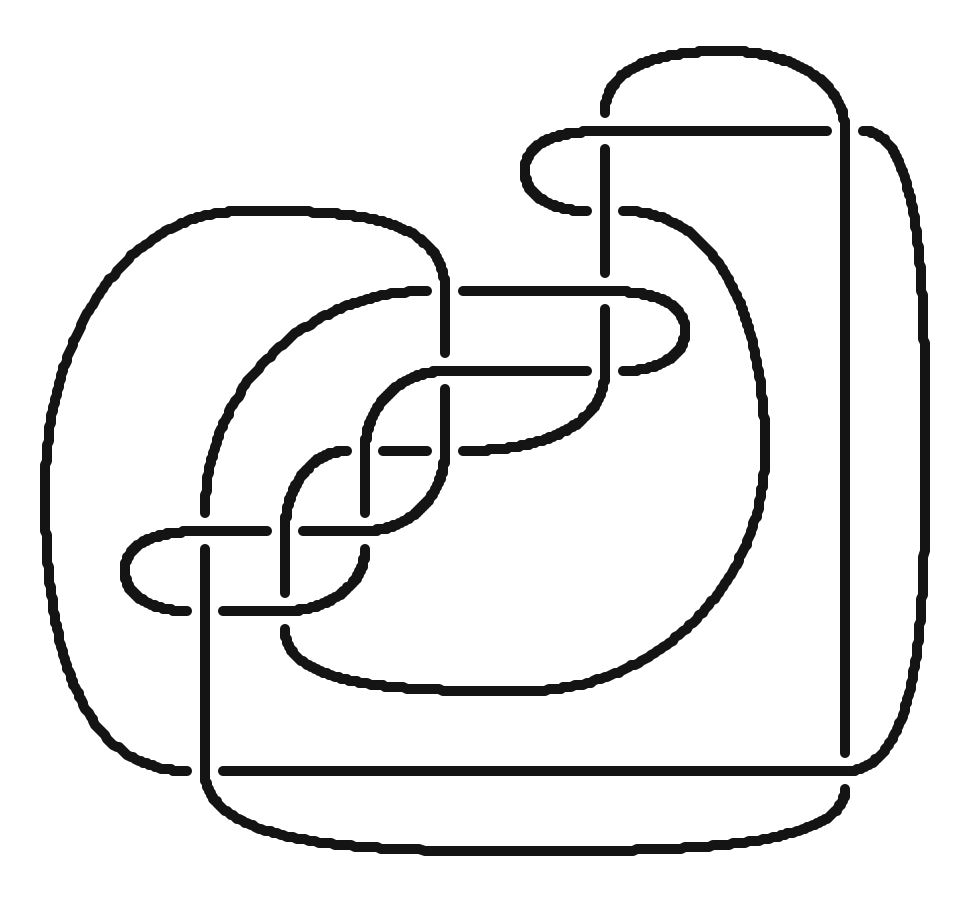}}\quad\quad\quad\quad\quad
\subfloat[$K_{19}$]{\includegraphics[scale=.08]{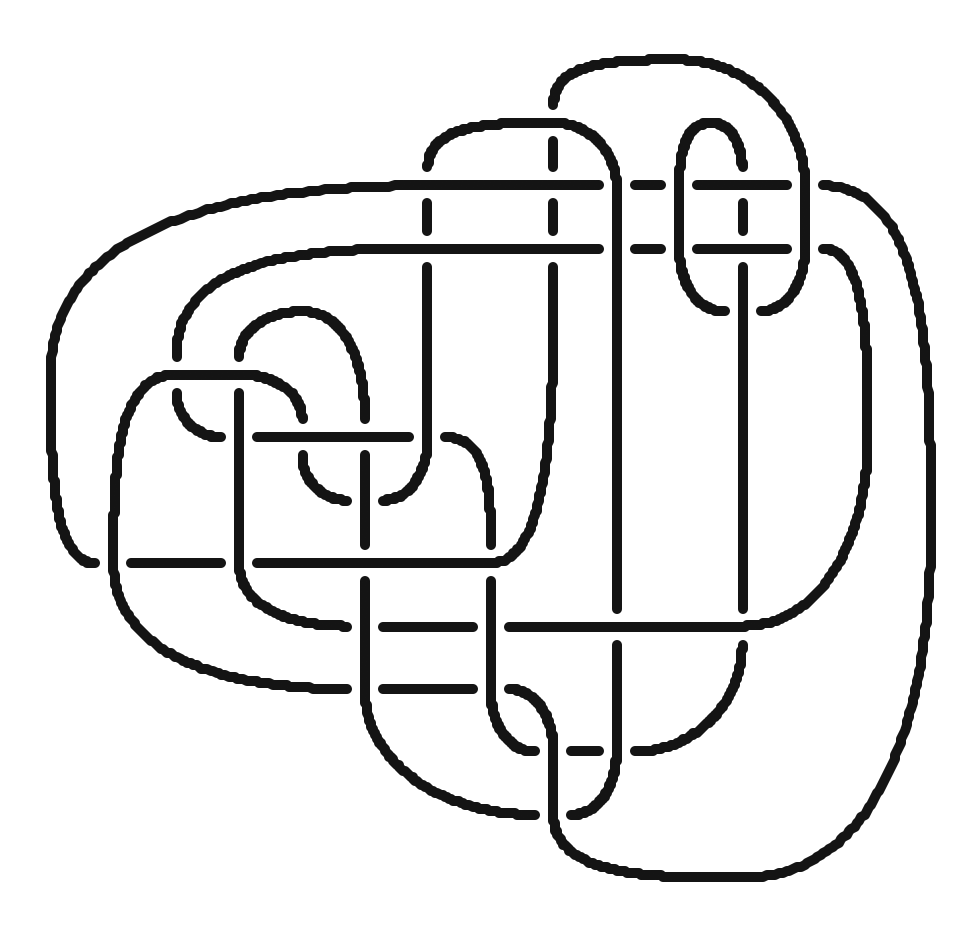}}\quad\quad\quad\quad\quad
\subfloat[$K_{20}$]{\includegraphics[scale=.08]{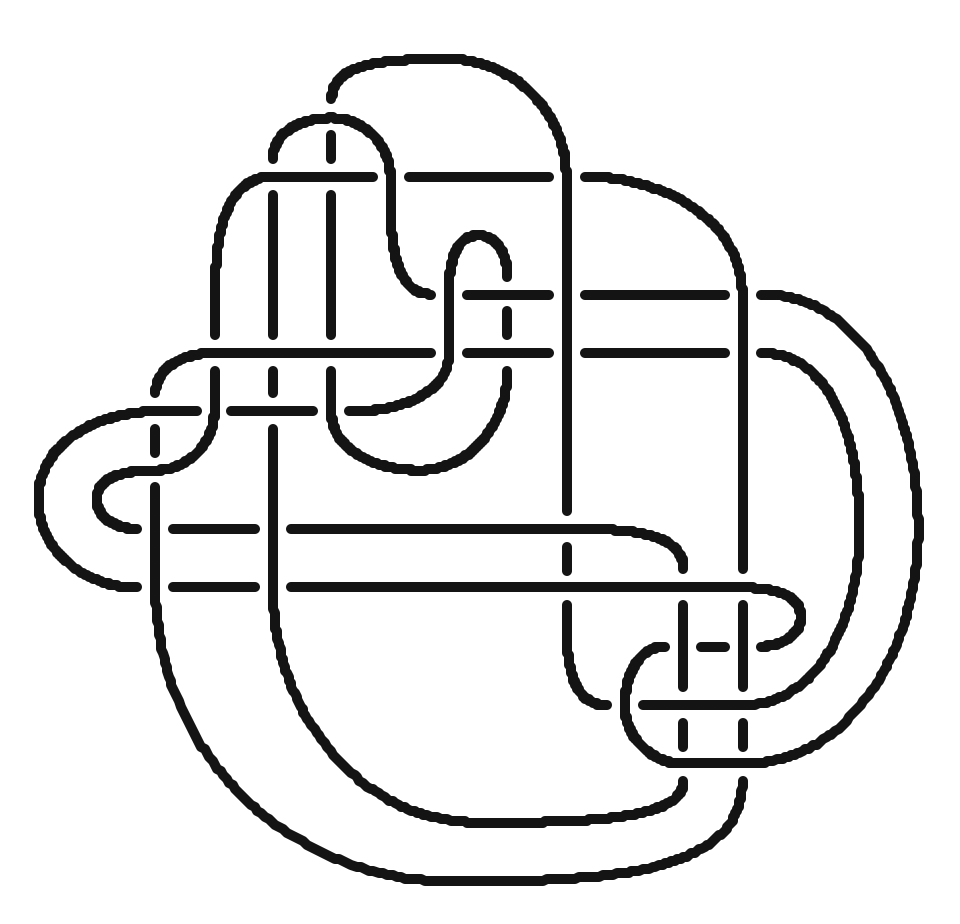}}\\
\subfloat[$K_{21}$]{\includegraphics[scale=.08]{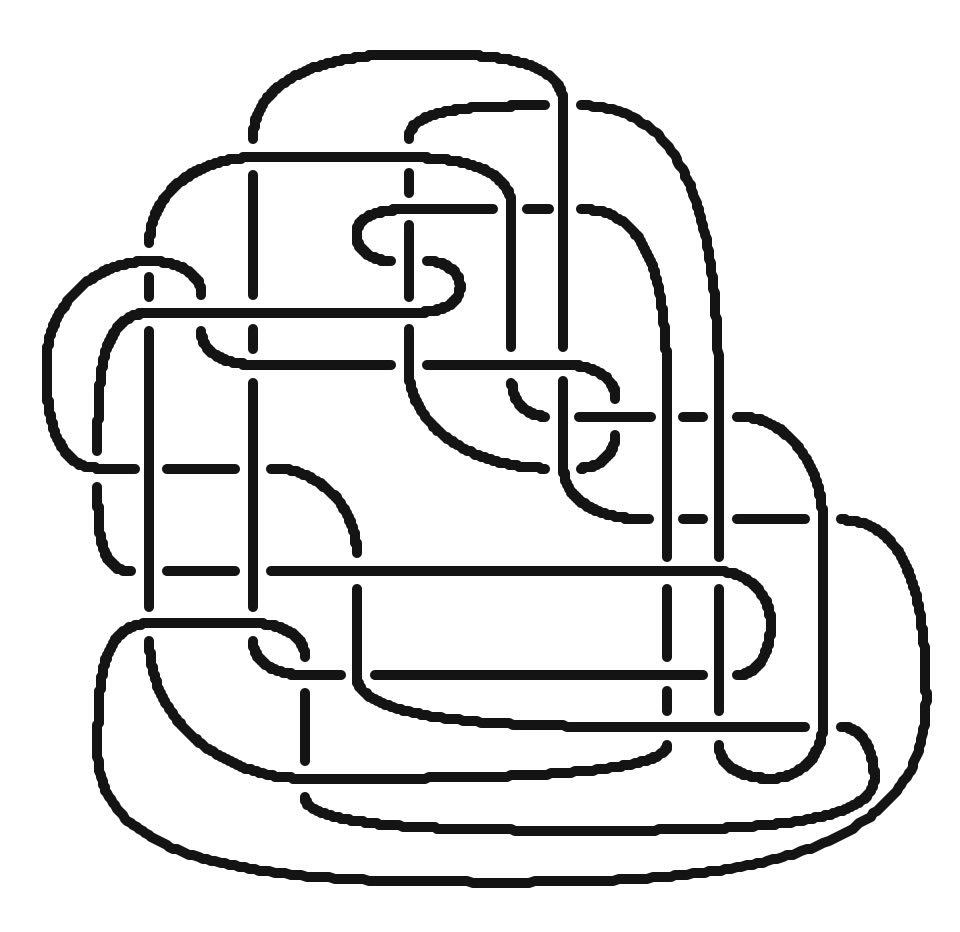}}\quad\quad\quad\quad\quad
\subfloat[$K_{22}$]{\includegraphics[scale=.08]{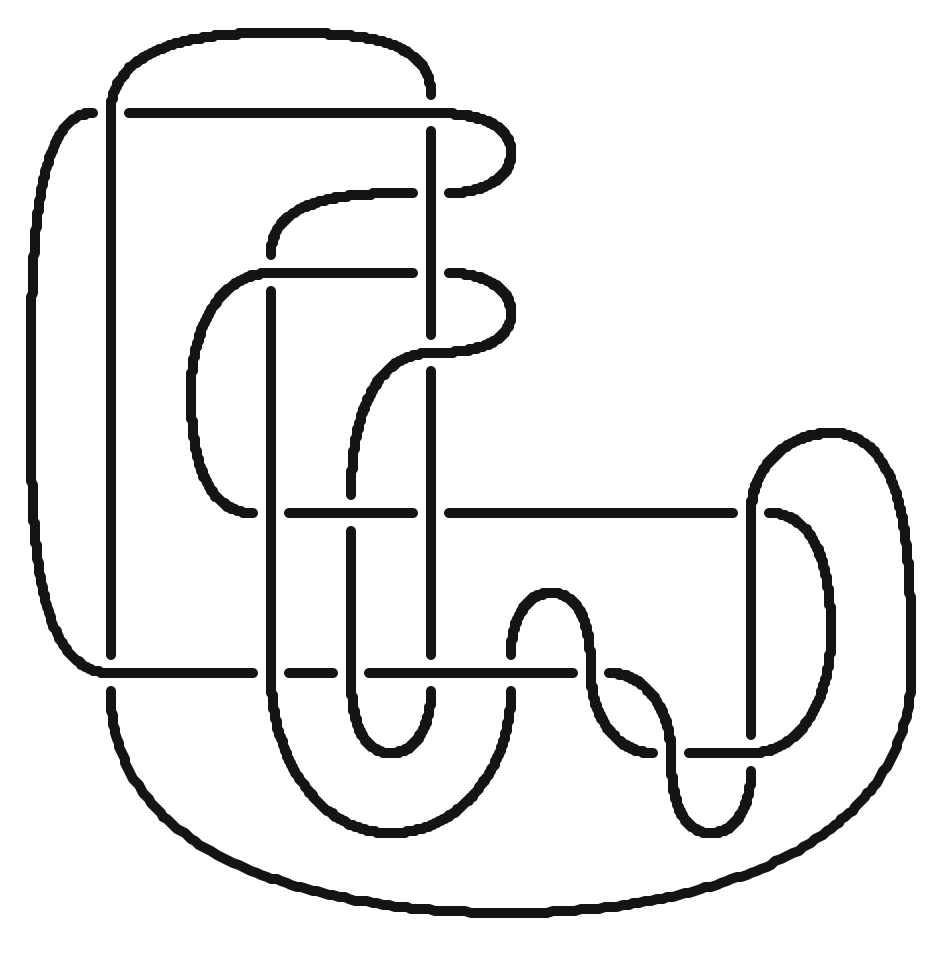}}\quad\quad\quad\quad\quad
\subfloat[$K_{23}$]{\includegraphics[scale=.08]{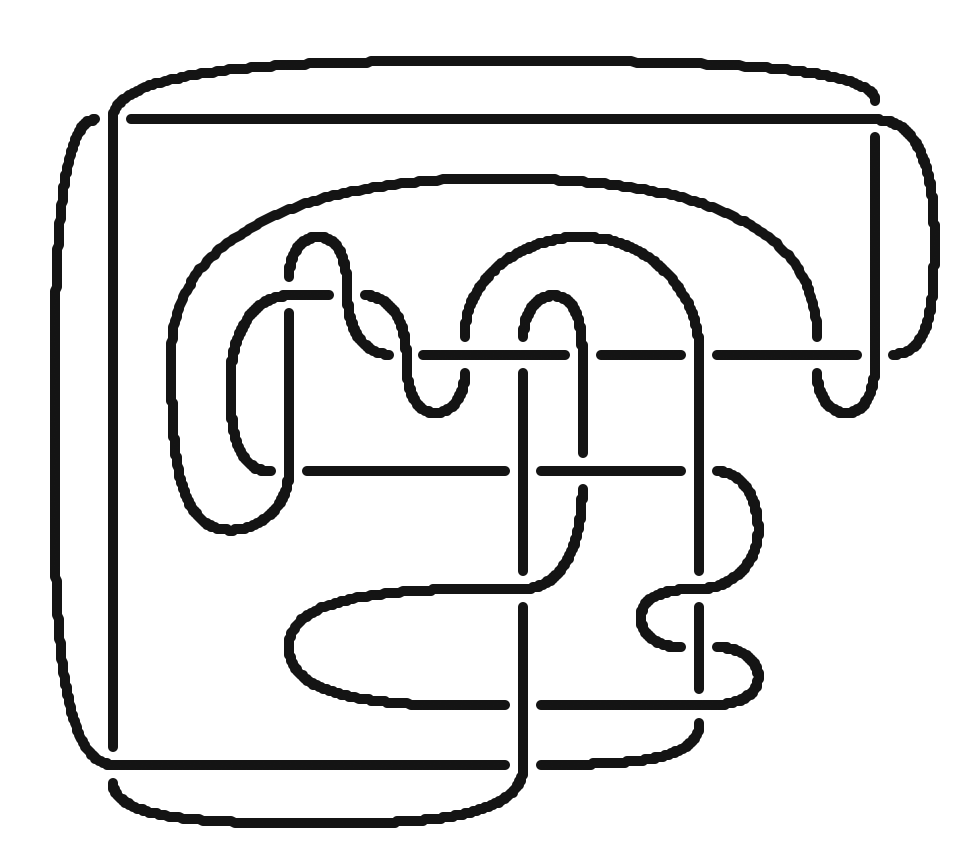}}
\caption{Candidates for BPH-slice knots.}
\label{fig:2}
\end{figure}

\end{document}